\documentclass[11pt]{article}

\usepackage{amsmath}
\usepackage{amsthm}
\usepackage{amsfonts}
\usepackage{amssymb}
\usepackage{graphics}
\usepackage{graphicx}
\usepackage{listings}
\usepackage{enumitem}
\usepackage{multirow}
\usepackage{mathtools}
\usepackage{xcolor}
\usepackage{setspace}
\usepackage[colorlinks=true, linkcolor=red, citecolor=blue, pdfencoding=auto, psdextra]{hyperref}
\usepackage[top=1in, bottom=1in, left=1in, right=1in]{geometry}
\usepackage{makecell}
\usepackage{array}
\usepackage{booktabs}

\newcommand{\R}{\mathbb{R}}
\newcommand{\C}{\mathbb{C}}
\newcommand{\N}{\mathbb{N}}
\newcommand{\Z}{\mathbb{Z}}

\newcommand{\eps}{{\varepsilon}}

\newcommand{\td}{\textup{d}}

\newcommand{\defeq}{:=}

\newcommand{\one}{\mathbf{1}}

\renewcommand{\epsilon}{\eps}
\renewcommand{\phi}{\varphi}
\renewcommand{\iff}{\Leftrightarrow}
\renewcommand{\Pr}{{\mathbf P}}

\DeclareMathOperator*{\argmin}{\operatorname{argmin}}

\newcommand{\rank}{\operatorname{rank}}

\theoremstyle{plain}
\newtheorem{theorem}{Theorem}[section]

\newtheorem{lemma}[theorem]{Lemma}

\newtheorem{claim}[theorem]{Claim}

\theoremstyle{definition}

\newtheorem{remark}[theorem]{Remark}
\newtheorem{setting}[theorem]{Setting}

\newtheorem{algorithm}[theorem]{Algorithm}

\newcommand{\E}[1]{\ensuremath{\mathbf{E}\left[#1\right]}}
\newcommand{\Var}[1]{\ensuremath{\mathbf{Var}\left[#1\right]}}

\newcommand{\sym}[1]{#1_{\scalebox{0.5}[0.5]{\textup{sym}}}}

\newcommand{\bound}[1]{#1^{{}^{{}_{\scalebox{0.5}[0.5]{\textup{end}}}}}}

\newcommand{\myD}{\tau}

\newcommand{\myL}{\textup{L}}
\newcommand{\myM}{\mathcal{M}}
\newcommand{\myC}{\mathcal{C}}
\newcommand{\myT}{\mathcal{T}}
\newcommand{\uppnormE}{\mathcal{H}}
\newcommand{\noise}{Z}
\newcommand{\mcK}{K}
\newcommand{\stdE}{\varsigma}
\newcommand{\upperE}{M}

\newcommand{\oldcomment}{}

\newcolumntype{M}[1]{>{\centering\arraybackslash}m{#1}}

\title{Fast exact recovery  of noisy matrix from few entries: \\the infinity norm approach}
\author{BaoLinh Tran$^1$,
Van Vu$^1$
}
\date{$^1$Department of Mathematics, Yale University}

\begin{document}

\maketitle

\begin{abstract}
The matrix recovery (completion) problem, a central problem in data science and theoretical computer science, is to recover a matrix $A$ from a relatively small sample of entries.

While such a task is impossible in general, it has been shown that  one can recover $A$ exactly in polynomial time, 
with high probability, from a random subset of entries, under three (basic and necessary) assumptions: (1) the rank of $A$ is very small compared to its dimensions (low rank), (2) $A$ has delocalized singular vectors (incoherence), and (3) the sample size is sufficiently large.

There are many different algorithms for the task, including convex optimization by Candes, Recht and Tao \cite{candesRecht2009,candesTao2010,recht2009}, alternating projection by Hardt and Wooters \cite{hardtAP2014,hardtWootersAP2014}, and low rank approximation with gradient descent by Keshavan, Montanari and Oh \cite{keshavan2010,keshavanNoisy2009}.

In applications, it is more realistic to assume that data is noisy. In this case, these approaches provide an approximate recovery with small root mean square error. However, it is hard to transform such approximate recovery to an exact one. 

Recently, results by Abbe et al. \cite{abbefan2017} and Bhardwaj et al. \cite{bhardwajVu2023} concerning approximation in the infinity norm showed that we can achieve exact recovery even in the noisy case, given that the ground matrix has bounded precision. Beyond the three basic assumptions above, they required either the condition number of $A$ is small \cite{abbefan2017} or the gap between consecutive singular values is large \cite{bhardwajVu2023}.   

In this paper, we remove these extra spectral assumptions. As a result, we obtain a  simple algorithm for exact recovery in the noisy case, under only three basic assumptions. The core of the analysi of the algorithm is a new  infinity norm analouge of (recent improved versions of) the clasicall Davis-Kahan  perturbation theorem. Our proof relies on a combinatorial contour intergration argument, and is different from all previous approaches.

\end{abstract}

{\it Acknowledgment.} The research is partially supported by Simon Foundation award SFI-MPS-SFM-00006506 and NSF grant AWD 0010308.
\newpage
\tableofcontents
\newpage



    



\section{Introduction}

\subsection{Problem description}

A large  matrix $A \in \R^{m\times n}$ is hidden, except for a few revealed entries in a set $\Omega \in [m]\times [n]$.
We call $\Omega$ the set of \emph{observations} or \emph{samples}.
The matrix $A_\Omega$, defined by
\begin{equation}  \label{eq:sample-matrix}
    (A_\Omega)_{ij} = A_{ij}
    \ \text{ for } \
    (i, j)\in \Omega,
    \ \text{ and } 0
    \text{ otherwise},
\end{equation}

\noindent is called the \emph{observed} or \emph{sample} matrix.
The task is to recover $A$, given  $A_\Omega$. This is the \emph{matrix recovery} (or \emph{matrix completion}) problem, a central 
problem in data science which has been received lots of attention in recent years, motivated by a number of real-life applications.
{
Examples include building recommendation systems, notably the \textbf{Netflix challenge} \cite{netflix}; reconstructing a low-dimensional geometric surface based on partial distance measurements from a sparse network of sensors \cite{linialSensorNetwork2001, soYeSensorNetwork2005}; repairing missing pixels in images \cite{candesTao2010}; system identification in control \cite{mesbahiSystemId1997}.
See the surveys \cite{liSurvey2019} and \cite{davenportSurvey2016} for more applications.}
In this paper, we focus on \emph{exact recovery}, where we want to recover all entries of $A$ exactly. 

It is standard to assume that the set $\Omega$ is \emph{random}, and researchers have proposed two models: (a) $\Omega$ is sampled uniformly among subsets with the same size, or (b) that $\Omega$ has independently chosen entries, each with the same probability $p$, called the \emph{sampling density}, which can be known or hidden. When doing mathematical analysis, it is often simple to replace the former model by the latter, using a simple conditioning trick. One samples the entries independently, and condition on the event that the sample size equals a given number.


\subsection{Some parameters and notation }  \label{sec:matcom-setting}

To start, we introduce a set of notation: 

\begin{itemize}
    \item Let the SVD of $A$ be given by $A = U\Sigma V^T = \sum_{i = 1}^r \sigma_iu_iv_i^T$, where $r := \rank A$, and 
    the singular values $\sigma_i$ are ordered: $\sigma_1 \ge \sigma_2 \ge \dots \ge \sigma_r $. 

    \item For each $s\in [r]$, let $A_s = \sum_{i = 1}^s \sigma_iu_iv_i^T$ be the best  rank-$s$ approximation of $A$.
    Define $B_s$ analogously for any matrix $B$.

    \item
    When discussing $A$, we denote $N \defeq \max\{m, n\}$.

    \item The \emph{coherence parameter} of $U$ is given by
    \begin{equation}  \label{eq:candesRecht2009-A0}
        \mu(U) \defeq \max_{i\in [m]} \frac{m}{r}\|e_i^TU\|^2
        = \frac{m\|U\|_{2, \infty}}{r},
    \end{equation}
    where the 2-to-$\infty$ norm of a matrix $M$ is given by $\|M\|_{2, \infty} \defeq \sup\{\|Mu\|_\infty: \|u\|_2 = 1\}$, which is the largest row norm of $M$.
    Define $\mu(V)$ similarly.
    When $U$ and $V$ are singular bases of $A$, we let $\mu_0(A) = \max\{\mu(U), \mu(V)\}$ and simply use $\mu_0$ when $A$ is clear from the context.
    
    The notion of coherence appears in many fields, and many works use the definition $\mu(U) = m\|U\|_\infty^2$, where the infinity norm is  the absolute value of the largest entry.
    We stick to the definition above, which is consistent with the parameter $\mu_0$ in many popular papers in this area \cite{candesRecht2009,candesTao2010,recht2009,keshavan2010,keshavanNoisy2009,candesPlan2009}.

    \item The following parameter has also been used: 
    \begin{equation}  \label{eq:candesRecht2009-A1}
        \mu_1 = \max_{i\in [m], j\in [n]}
        \frac{\sqrt{mn}}{\sqrt{r}} |e_i^TU^T Ve_j|
        = \frac{\sqrt{mn}}{\sqrt{r}}\|UV^T\|_\infty.
    \end{equation}
    
    We could not find any widely used name for this parameter in the literature. We temporarily use the name \emph{joint coherence} in this paper.
It is simple to see that  

\begin{equation} \label{boundmu} \mu_1 \le \mu_0\sqrt{r}. \end{equation}

\item We use $C$ to denote a positive universal constant, whose value changes from place to place.
The asymptotics notation 
are used under the assumption that $N \rightarrow \infty$. 
\end{itemize}

\subsection{Basic assumptions}

If we observe only $A_\Omega$, filling out the missing entries is clearly impossible, unless extra assumptions are given. Most  existing works on the topic  made the following three assumptions:

\begin{itemize}
    \item \emph{Low-rank:}
   One assumes that $r:= \rank A $ is much smaller than $\min \{m, n \} $. This assumption is crucial as it reduces the degree of freedom of $A$ significantly, making the problem solvable, at least from an information 
   theory stand point.
   Many papers assume $r $ is bounded ($r=  O(1)$), while $m, n\to \infty$. 
    

    \item \emph{Incoherence:} This assumption ensures that the rows and columns of $A$ are sufficiently ``spread out'', so the information does not concentrate in a small set of entries, which could be easily overlooked by random sampling.
    In technical terms, one requires $\mu_0$, and (sometimes) also $\mu_1$  to be small.
   

    \item {\it Sufficient sampling size/density:} Due to a coupon collector effect, both random sampling models above need at least $N\log N$ observations to avoid empty rows or columns.
    Another lower bound is given by the degree of freedom: One needs to know $r(m + n - 1)$ parameters to compute $A$ exactly.
    A more elaborate argument in \cite{candesTao2010} gives the lower bound $|\Omega| \ge  C rN\log N$. 
    This is equivalent to $p \ge Cr(m^{-1} + n^{-1})\log N$ for the independent sampling model.

    \vskip2mm 

   For more discussion about the necesity of these assumptions, we refer to \cite{candesRecht2009,candesTao2010,davenportSurvey2016}.
   We will refer to these assumptions as the {\it basic assumptions}. These have been assumed in all results we discuss in this paper. 

\end{itemize}

\subsection{Exact recovery in practice: finite precision} \label{sec:exact-recovery}

Let us make an  important comment on what  {\it exact recovery} means. 
Clearly, if an entry of $A$ is irrational, then it is impossible for our computers to write it down exactly, let alone computing it. 
Thus, exact recovery only makes sense when the entries of $A$ have finite precision, which is the case in all real-life applications. 
 To this end, we says that a  matrix  $A$ have a finite precision $\epsilon_0$,  if its entries are integer multiples of a parameter $\epsilon_0 >0$.
For instance, if all entries have two decimal places, then $\epsilon=.01$. 

In many pratical problems, such as compelting/rating  a  recommendation system, the parameter $\epsilon_0$ is actually quite large. For instance, in the most  influential  problem in the field, the Netflix Challenge \cite{netflix}, the entries of $A$ are ratings of movies, which are half integers from 1 to 5, so  $\epsilon_0 =1/2$.  The algorithm we propose and analyze in this paper will exploit this fact to our advantage.  (One should not confuse 
this notion with the machine precision or machine epsilon, which is very small.)



\subsection{A brief summary of existing methods}  \label{sec:matcom-solutions}

There is a huge literature on matrix completion. 
In this section, we try to summarize some of the main methods. 

\begin{itemize}
    \item \emph{Nuclear norm minimization:}
    This method is based on convexifying the intuitive but NP-hard approach of minimizing the rank given the observations.
    This method is guaranteed to achieve exact recovery under perhaps the most general assumptions.
    However, the time complexity includes high powers of $N$, which makes it less practical, and the calculation may be sensitive to noise \cite{candesTao2010} The best-known solvers for this problem run in time $O(|\Omega|^2N^2)$, which is $O(N^4\log^2N)$ in the best case \cite{liSurvey2019}.


    \item \emph{Alternating projections:}
    This is based on another intuitive but NP-hard approach of fixing the rank, then minimizing the RMSE with the observations.
    The basic version of the algorithm switches between optimizing the column and row spaces, given the other, in alternating steps.
    Existing variants of alternating projections run well in practice.


    \item \emph{Low-rank approximation:}
    The general idea here is to view the sample matrix $A_\Omega$ as a rescaled and unbiased random perturbation of $A$. This way, it is natural to  first approximate $A$ by taking a low rank approximation  of $p^{-1} A_\Omega$ (where $p$ is the density). Next, one can  use 
    an extra cleaning step to make the recovery exact. 
    The first step  (truncated SVD in this case) here has only one operation, and runs fast in practice. Thus the main complexity issue is the cleaning step. 
    Our new  algorithm in Section \ref{sec:matcom-our-algo}  belongs to this category, and has a trivial cleaning step (which is just the rounding-off of the entries).
\end{itemize}


Now we discuss the three approaches in more details. 

\subsubsection{Nuclear norm minimization}

This approach starts from the intuitive idea that if $A$ is mathematically recoverable, it has to be the matrix with the lowest rank agreeing with the observations at the revealed entries.
Formally, one would like to solve  the following optimization problem:
\begin{equation}  \label{eq:rank-optimize-prob}
    \text{minimize} \quad \rank X \quad \text{ subject to } \quad X_\Omega = A_\Omega.
\end{equation}

Unfortunately, this problem is NP-hard, and all existing algorithms take doubly exponential time in terms of the dimensions of $A$ \cite{chistov2006}.
To overcome this problem, Candes and Recht  \cite{candesRecht2009},  motivated by an idea from the \emph{sparse signal recovery} problem in the field of \emph{compressed sensing} \cite{candesRomberg2006,donoho2006},  proposed to replace  the rank with the nuclear norm of $X$, leading to 

\begin{equation}  \label{eq:nuclear-optimize-prob}
    \text{minimize} \quad \|X\|_* \quad \text{ subject to } \quad X_\Omega = A_\Omega.
\end{equation}

The paper \cite{candesRecht2009} was shortly followed by Candes and Tao \cite{candesTao2010},  with both improvements and trade-offs, and ultimately by Recht \cite{recht2009}, 
who improved both previous results, proving that $A$ is the unique solution to \eqref{eq:nuclear-optimize-prob},  given the sampling size bound
\begin{equation}  \label{eq:recht2009-sample-size}
    |\Omega| \ge C\max\{\mu_0, \mu_1^2\} rN\log^2 N,
\end{equation}
for the coherence parameters $\mu_0$ and $\mu_1$ defined previously.

If one replaces  $\mu_1$ with $\mu_0\sqrt{r}$ (see \eqref{boundmu}, the RHS becomes $C\mu_0^2 r^2N\log^2 N$.
This attains the optimal power of $N$ while missing slightly from the optimal powers of $r$ and $\log N$.

The key advantage of replacing the rank in Problem \eqref{eq:rank-optimize-prob} with the nuclear norm is that Problem \eqref{eq:nuclear-optimize-prob} is a convex program, which can be further translated into a semidefinite program \cite{candesRecht2009,candesTao2010}, solvable in polynomial time by a number of algorithms. However, convex 
optimization program usually runs slowly in practice. 
The survey \cite{liSurvey2019} mentioned the interior point-based methods SDPT3 \cite{sdpt3} and SeDuMi \cite{sedumi}, which can take up to $O(|\Omega|^2N^2)$ floating point operations (FLOPs) assuming \eqref{eq:recht2009-sample-size}, even if one takes advantage of the sparsity of $A_\Omega$. Indeed, as $\Omega $ is at least $N log N$ (by coupon collector), the number of operations is 
$\Omega (N^4)$, which is too large even for moderate $N$.  An \emph{iterative singular value thresholding} method aiming to solve a regularized version of nuclear norm minimization, trading exactness for performance, has been proposed \cite{caiISVT2010}. It is thus desirable to develop faster algorithms for the problem, and in what follows we discuss two other methods, which achieve this goal, under extra assumptions.

\subsubsection{Modified alternating projections}


The intuition behind this approach is to fix the rank, then attempt to match the observations at much as possible.
If we {\it know} the rank $r$ of $A$ precisely, then it is natural to look at the following optimization problem 

\begin{equation}  \label{eq:low-rank-optimization-prob}
    \text{minimize} \quad \|(A - XY^T)_\Omega\|_F^2 \quad \text{ over } X\in \R^{m\times r}, Y\in \R^{r\times n}.
\end{equation}
This,  unfortunately, like \eqref{eq:rank-optimize-prob}, is NP-hard \cite{davenportSurvey2016}.
There have been many studies proposing variants of \emph{alternating projections}, all of which involve the following basic idea: suppose one already obtains an approximator $X^{(l)}$ of $X$ at iteration $l$, then $Y^{(l)}$ and $X^{(l + 1)}$ are defined by
\begin{equation*}
    Y^{(l)} \defeq
    \argmin_{Y\in \R^{r\times n}} \|(A - X^{(l)}Y^T)_\Omega\|_F^2,
    \quad \quad
    X^{(l + 1)} \defeq
    \argmin_{X\in \R^{r\times n}} \|(A - X(Y^{(l)})^T)_\Omega\|_F^2.
\end{equation*}
The survey \cite{davenportSurvey2016} pointed out that these methods tend to outperform nuclear norm minimization in practice. Oh the other hand,  there are few rigorous guarantees for recovery.
The convergence and final output of the basic algorithm above also depends highly on the choice of $X^{(0)}$ \cite{davenportSurvey2016}.

Jain, Netrapalli and Sanghavi (2012) \cite{jainAP2012} developed one of the first alternating projections variants for matrix completion with rigorous recovery guarantees.
They proved that, under the same setting in Section \ref{sec:matcom-setting} and the sample size condition
\begin{equation*}
    |\Omega| \ge C\mu_0 r^{4.5} \left( \frac{\sigma_1}{\sigma_r} \right)^4 N\log N \log\frac{r}{\eps},
\end{equation*}
the AP algorithm in \cite{jainAP2012} recovers $A$ within an Frobenius norm error $\eps$ in $O(|\Omega|r^2\log(1/\eps))$ time with high probability. 
Since the Frobenius norm is larger than the infinity norm, this gives us an exact recovery if we set $\epsilon =\epsilon_0/3$, where $\epsilon_0$ is the precision level of $A$; see subsection \ref{sec:exact-recovery}.

Compared to the previous approach, there are two new essential requirements here. First one needs to know  the rank of $A$ precisely. Second, there is a strong dependence on the condition number $\kappa:= \sigma_1/\sigma_r $; the result is effective only if $\kappa$ is small.


The condition number factor was reduced to quadratic by Hardt \cite{hardtAP2014} and again by Hardt and Wooters \cite{hardtWootersAP2014} to logarithmic, at 
the cost of an increase in the powers of $r$, $\mu_0$ and $\log N$.

\begin{remark} \label{tryingr} ({\it A problem with trying many ranks})
In practice, the common situation is that we do not know 
the rank $r$ exactly, but have some estimates
(for instance, $r$ is between  known values  $r_{\min}$ and $r_{\max}$). It has been  suggested (see, for instance, \cite{keshavanOh2009}) that one tries all integers in this range as the potential value of $r$. From the complexity view point, this  only increases the running time by a small factor
$r_{\max} -r_{\min}$, which is acceptable. However, the main trouble 
with this idea is that it is not clear that among the ouputs, which one we should choose. If we go for exact recovery, then what should we do if there are two different outputs which agree on $\Omega$ ?  We have not found a rigorious treatment of this issue. 
\end{remark}

\subsubsection{Low rank approximation with Gradient descent}


As discussed earlier, if one assumes the independent sampling model with probability $p$, then the rescaled sample matrix $p^{-1} A_\Omega$ can be viewed as a \emph{random perturbation} of $A$.
Since $\E{A_\Omega/p} = A$, this perturbation is unbiased, and the matrix $E \defeq  p^{-1} A_\Omega - A$ is a random  matrix with mean zero.


Assuming that the rank $r$ is known, 
Keshavan, Montanari and Oh \cite{keshavan2010} first use  the best  rank-$r$ approximation of $p^{-1} A _{\Omega} $
to obtain an approximation of $A$, then add  a cleaning step, using optimization via gradient descent,  to achieve exact recovery. 
Here is the description of their algorithm:

\begin{enumerate} 
   \item \emph{Trimming:} first zero out all columns in $A_\Omega$ with more than $2|\Omega|/m$ entries, then zero out all rows with more than $2|\Omega|/n$ entries, producing a matrix $\widetilde{A_\Omega}$.

    \item \emph{Low-rank approximation:} Compute the best rank-$r$ approximation of $\widetilde{A_\Omega}$ via truncated SVD.
    Let $\textsf{T}_r(\widetilde{A_\Omega}) = \tilde{U}_r\tilde{\Sigma}_r\tilde{V}_r^T$ be the output.

    \item \emph{Cleaning:} Solve for $X, Y, S$ in the following optimization problem:
    \begin{equation}  \label{eq:keshavan-optimization-prob}
        \text{minimize } \quad \bigl\|
            A_\Omega - (XSY^T)_\Omega
        \bigr\|_F^2 \quad
        \text{ for } \quad X\in \R^{m\times r}, Y\in \R^{n\times r}, S\in \R^{r\times r},
    \end{equation}
    using a gradient descent variant \cite{keshavan2010}, starting with $X_0 = \tilde{U}_r$, $Y_0 = \tilde{V}_r$ and $S_0$ be the $r\times r$ matrix minimizing the objective function above given $X_0$ and $Y_0$.
    
    Let $(X_*, Y_*, S_*)$ be the optimal solution.
    Output $X_* S_* Y_*^T$.
\end{enumerate}
The last cleaning step resembles the optimization problem in alternating projections methods, but they used gradient descent instead.
The authors \cite{keshavan2010} showed that the algorithm returns an output arbitrarily close to $A$, given {\it  enough}  iterations in the cleaning step, provided the following sampling size condition:
\begin{equation}  \label{eq:keshavan2010-sample-size}
    |\Omega| 
    \ge C\max\left\{
        \mu_0\sqrt{mn} \left(\frac{\sigma_1}{\sigma_r}\right)^2 r\log N
        , \quad
        \max\{\mu_0, \mu_1\}^2r \min\{m, n\} \left(\frac{\sigma_1}{\sigma_r}\right)^6
    \right\}.
\end{equation}


The powers of $r$ and $\log N$ are optimal by the coupon-collector limit, answering a question from \cite{candesTao2010}. 
On the other hand, the bound depends heavily on the 
condition number $\kappa:= \sigma_1/ \sigma_r$. Furthermore, similar to the situation in the previous subsection, one  needs to know the rank $r$ in advance; see Remark \ref{tryingr}.

In a later paper \cite{keshavanOh2009}, Keshavan and Oh showed that one can compute $r$ (with high probability) if 
the condition number satisfies $\kappa = O(1)$; see also Remark \ref{tryingr}. 
Thus, it seems that the critical extra  assumption for this algorithm (apart from the three basic assumptions)  to be efficient is that the singular values of $A$ are of the same order of magnitude ($\kappa = O(1)$).
This assumption is strong, and we do not know how often it holds in practice.
In many data sets, it has been noted that the leading singular values decay rapidly, meaning $\kappa$ is large.
For instance, Figure \ref{fig:yale-face-sing-val} shows this phenomenon for the Yale face database, which is often used in demonstrations of the Principal Component Analysis \cite{wainwrightbook2019} method for dimensionality reduction.
\begin{figure}  \label{fig:yale-face-sing-val}
    \centering
    \includegraphics[width=0.9\linewidth]{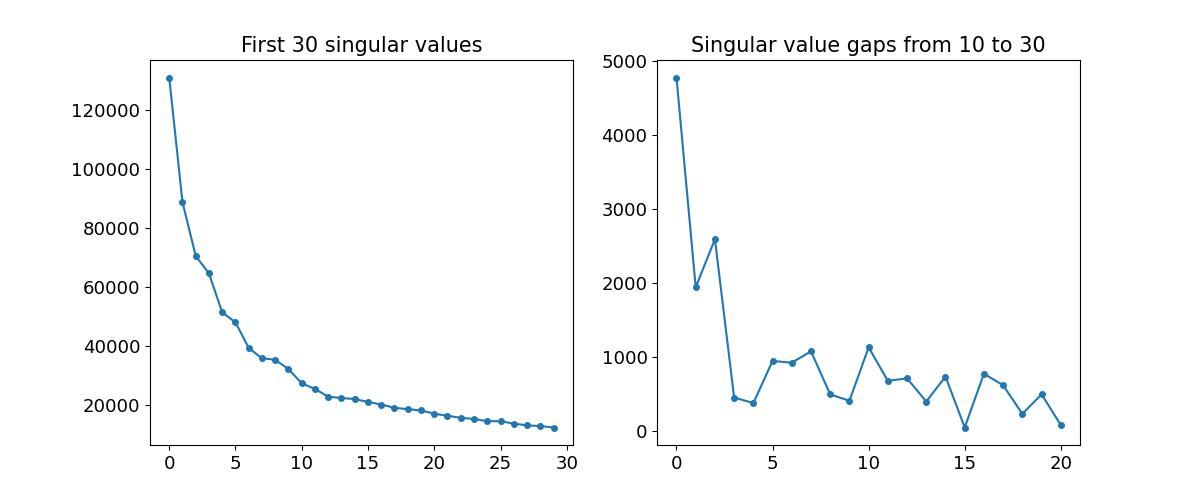}
    \caption{
    The Yale face database \cite{wainwrightbook2019} has $165$ greyscale face images of dimension $243 \times 320$, which can be turned into vectors in $\R^{77760}$ and arranged onto a matrix in $\R^{77760 \times 165}$.
    Center this matrix by subtracting from every column their average.
    Shown above are the first 30 singular values and first 30 consecutive singular value gaps of this centered matrix. The picture on the left shows that  matrix has a small numer of large eigenvalues (about 30) and the rest looks insignificant. On the other hand, the picture on the right shows that several eigenvalue gaps between the significant ones are quite small.    }
\end{figure}

 From the complexity point of view, the first part (low rank approximation) of the algorithm is very fast, in both theory and practice,  as it used truncated SVD only once. On the other hand, the authors of \cite{keshavan2010} did not include a full convergence rate analysis of their (cleaning)  gradient descent part, only briefly mentioning that quadratic convergence is possible.

\subsubsection{Low rank approximation with rounding-off} \label{section:AR}

In this approach, one exploits the fact that $A$ has finite precision (which is a necessary assumption for exact recovery); see subsection \ref{sec:exact-recovery}.  It is clear that if each entry of $A$ is an integer multiple of $\epsilon_0$, then  to achieve an exact recovery, it suffices  to compute each entry with error less than $\epsilon_0/2$, and then round it off. In other words, it is sufficient to obtain an approximation of $A$ in the inifnity norm.  It has been shown, under different extra assumptions, that low rank approximation fullfils this purpose.

The first infinity norm result was obtained by 
Abbe, Fan, Wang, and Zhong  \cite{abbefan2017}. They showed that the  
best rank-$r$ approximation of $p^{-1} A_\Omega$ is close to $A$ in the infinity norm \cite[Theorem 3.4]{abbefan2017}.
Technically, they proved that if   $p \ge 6N^{-1}\log N$, then
\begin{equation*}
    \|p^{-1}(A_\Omega)_r - A\|_\infty
    \le C\mu_0^2\kappa^4\|A\|_\infty \sqrt{\frac{\log N}{pN}},
\end{equation*}
for some universal constant $C$, provided $\sigma_r \ge C\kappa \|A\|_\infty \sqrt{\frac{N\log N}{p}}$, where $\kappa = \sigma_1/\sigma_r$ is the condition number.

If we turn this result into an algorithm (by simply rounding off the approximation), then we face the same
two issues discussed in the previous subsection. We need to know  the rank $r$ and the condition number $\kappa$ has to be small. As discussed, this boils down to the strong assumption that 
the condition number is bounded by a constant ($\kappa= O(1)$).

\vskip2mm 


\noindent {\it Eliminating the condition number.} Very recently, Bhardwaj and Vu \cite{bhardwajVu2023} proposed and analyzed a slightly different algorithm, where they do not need to know the rank of $A$. Their bound on $|\Omega|$ does not include the condition number $\kappa$. Thus, they completely eliminated the condition number. 
However, the cost here is that they need a new assumption on the gaps between consecutive singular values.

As this work is the closest to our new result, let us state their result 
for matrices with integer entries (the precision $\eps_0 = 1$).
One can reduce the  case of general case to this by scaling.

\begin{algorithm}[Approximate-and-Round (\textbf{AR})]  \label{algo:bhardwajVu2023}\ 
\begin{enumerate}
    \item Let $\tilde{A} \defeq p^{-1}A_\Omega$ and compute the SVD: $\tilde{A} = \tilde{U}\tilde{\Sigma}\tilde{V}^T = \sum_{i = 1}^{m\wedge n} \tilde{\sigma_i}\tilde{u_i}\tilde{v}_i^T$.

    \item Let $\tilde{s}$ be the last index such that $\tilde{\sigma}_i \ge \frac{N}{8r\mu}$, where $\mu \defeq N\max\{\|U\|_\infty^2, \|V\|_\infty^2\}$ is known.

    \item Let $\hat{A} \defeq \sum_{i = 1}^{\tilde{s}} \tilde{\sigma_i}\tilde{u_i}\tilde{v}_i^T$.
    
    \item Round off every entry of $\hat{A}$ to the nearest integer. 
\end{enumerate}
\end{algorithm}

They showed that with probability $1 - o(1)$, before the rounding step, $\| \hat{A} - A \|_{\infty} < 1/2$, guaranteeing an exact recovery of $A$, under the following assumptions:  

\begin{itemize}
    \item \emph{Low-rank:} $r = O(1)$.

    \item \emph{Incoherence:} $\mu = O(1)$.

    \item \emph{Sampling density:} $p \ge N^{-1}\log^{4.03}N$.

    \item \emph{Bounded entries:} $\|A\|_\infty \le K_A$ for a known constant $K_A$.
    
    
    \item \emph{Gaps between consecutive singular values:} $\min_{i\in [s]} (\sigma_i - \sigma_{i + 1}) \ge Cp^{-1}\log N$.
\end{itemize}
%

Aside from the first three basic assumptions, the new assumption that the entries are bounded is standard for real-life datasets.
In the step of finding the threshold, it seems that one  needs to know both $r$ and $\mu$, but a closer look at the analysis reveals that it is  possible to relax  to knowing only their upper bounds. (We will do exactly this later in our algorithm, which is a variant of \textbf{AR}.)

As discussed, the  main improvement of \textbf{AR} over the previous spectral approaches  is  the removal of 
the dependence on  the condition number. This removal was based on an  entirely different mathematical analysis, which shows that the leading singular vectors of 
$A$ and  $p^{-1} A_{\Omega}$
are close in the infinity norm. 

This removal, however, comes at the cost of the (new) 
gap assumption.
While the required bound for the gaps is mild (much weaker than what one requires for the application of Davis-Kahan theorem; see \cite{bhardwajVu2023} for more discussion), we do not know 
how often matrices in practice satisfy it.
{
In fact, Figure \ref{fig:yale-face-sing-val} shows that the Yale face database matrix \cite{wainwrightbook2019} has several steep drops in singular value gaps.
}

  The reader may have already noticed  that 
 this assumption goes into the opposite direction of the small condition number assumption. 
Indeed, if the condition number is large, then the  least singular value $\sigma_r$
is considerably smaller than the largest one
$\sigma_1$, which suggests, at the intuitve level at least, that  the gaps between the consecutive singular values are large. So, from the mathematical view point, the situation is quite intriguing. 
We have  two valid theorems with {\it constrasting extra assumptions}  (beyond the three basic assumptions). 
The most logical explanation here should be that  neither assumption is in fact needed. 
This conjecture, in the 
(more difficult) noisy setting, presented in the next section, is the motivation of our study.




\subsection{Recovery with Noise}  \label{sec:matcom-noisy}

 Candes and Plan, in their influential survey \cite{candesPlan2009}, pointed out that
 data is often noisy, and a more realistic model of the recovery problem is to consider 
 $A' = A+ \noise$, for $A$ being the low rank ground truth matrix and $\noise$ the noise. We 
 observe a sparse matrix $A'_\Omega $, where each entry of $A'$ is observed with probability $p$ and set to $0$ otherwise. In other words, we have access to a small random set of noisy entries. Notice that in this case, the truth matrix $A$ is still low rank, but the noisy matrix $A'$, whose entries we  observe,  can have full rank. In what follow, we denote our input by $A_{\Omega, \noise}$, emphasizing the presence of the noise.


Recovery from noisy observation is clearly a harder problem, and most 
papers concerning  noisy recovery aim for recovery in  {\it the normalized Frobenius norm}  (root mean square error; RMSE), rather than exact recovery. 

Continuing  the nuclear norm minimization approach, 
Candes and Plan \cite{candesPlan2009} adapted to the noisy situation by relaxing the constraint on the observations, leading to the following problem:
\begin{equation}  \label{eq:nuclear-optimize-prob-noisy}
    \text{minimize} \quad \|X\|_* \quad \text{ subject to } \quad \left\|X_\Omega - A_{\Omega, \noise} \right\|_F \le \delta,
\end{equation}
where $\delta$ is a known upper bound on $\|\noise_\Omega\|_F$.
The authors showed that, under the same sample size condition in \cite{recht2009}, with probability $1 - o(1)$, the optimal solution $\hat{A}$ 
satisfies 

\begin{equation}  \label{eq:candesPlan2009-error-bound}
    \frac{1}{\sqrt{mn}}\|\hat{A} - A\|_F \le
    C\|\noise_\Omega\|_F \sqrt{\frac{\min\{m, n\}}{|\Omega|}}.
\end{equation}

If one would like the RMSE to be at most $\eps$, then one needs to require 

\begin{equation}  \label{eq:candesPlan2009-sample-size-implicit}
    |\Omega| \ge C\frac{\|\noise_{\Omega} \|_F^2\min\{m, n\}}{\eps^2},
\end{equation}
which grows quadratically with $1/\eps$.

For exact recovery, one needs to turn the approximation in the Frobenius norm 
into an approximation in the infinity norm; see subsection \ref{sec:exact-recovery}. This is a major mathematical challange, and in general, there 
is no efficient way to do this. 
The trivial bound that $\| M \| _{\infty}  \ge \| M\| _{F} $ is too generous. If we use this and then use  
\eqref{eq:candesPlan2009-error-bound} to bound the RHS, then the corresponding 
 bound on $|\Omega |$  in \eqref{eq:candesPlan2009-sample-size-implicit} becomes larger than $mn$, which is meaningless.  This is the common situation with all Frobenius norm bounds 
 discussed in this section. 


\vskip2mm 

Concerning the alternating method, a corollary of  \cite[Theorem 1]{hardtWootersAP2014} shows  that we can obtain 
an approximation $\hat A$ of rank $r$, where 

\begin{equation} \label{noisyAP} 
    \| \hat A - A \|  \le (2+ o(1)) \| \noise \|  + \epsilon \sigma_1,
\end{equation}
given that
\begin{equation*}
    p = \tilde \Omega \left(
        \frac{1}{n} \Bigl(1 + \frac{\| \noise \|_F } { \epsilon \sigma_1 }\Bigl)
    \right)^2.
\end{equation*}
The bound here is in the spectral norm, and one 
can translate into Frobenius norm by the fact that $\| M \| _F \le \sqrt {\rank M } \| M \|$. Again, it is not clear of how to obtain exact recovery from here.

\vskip2mm 

 Continuing the spectral approach, Keshavan, Montanari and Oh \cite{keshavanNoisy2009} also  
 extended their result from \cite{keshavanOh2009} to the noisy case, using the same 
 algorithm. They proved that with the same sample size condition as \eqref{eq:keshavan2010-sample-size}, the output satisfies w.h.p.
\begin{equation}  \label{eq:keshavanNoisy2009-error-bound}
    \|\hat{A} - A\|_F \le
    C \left(\frac{\sigma_1}{\sigma_r}\right)^2 \frac{r^{1/2}mn}{|\Omega|}\|\noise_\Omega\|_{\textup{op}}.
\end{equation}
If one would like to have  $\frac{1}{\sqrt {mn}} \|\hat{A} - A\|_F \le \eps$, this translates to the following  sample size condition:

\begin{equation}  \label{eq:keshavanNoisy2009-sample-size-implicit}
    |\Omega| \ge C\frac{\sigma^2rN}{\eps^2}\left(\frac{\sigma_1}{\sigma_r}\right)^2, 
\end{equation} where the dependence on $\epsilon$ is again quadratic. 

{
Chatterjee \cite{chatterjee2012} uses a spectral approach with a fixed truncation point independent from $\rank A$, and thus does not require knowing it.
He required that $p \ge CN^{-1}\log^6 N$ and achieved the bound
\begin{equation*}
    \E{\frac{1}{mn}\|\hat{A} - A\|_F^2}
    \le C\min\left\{
        \sqrt{\frac{r}{p}\left(\frac{1}{m} + \frac{1}{n}\right)},
        1
    \right\} + o(N).
\end{equation*}
The advantage over previous methods is the absence of the incoherence assumption.
However, to translate the bound in expectation above to obtain $\frac{1}{\sqrt {mn}} \|\hat{A} - A\|_F \le \eps$ with probability at least $1 - \delta$, assuming Markov's inequality is used, one will need
\begin{equation*}
    p \ge \frac{Cr}{\eps^4\delta^2}\left(\frac{1}{m} + \frac{1}{n}\right),
\end{equation*}
The dependence on $\eps$ grows substantially faster than \cite{keshavanNoisy2009}.
In fact, P. Tran's and Vu's recent result in random perturbation theory \cite{phuctranVu2023} can be used to prove a high-probability mean-squared-error bound, again without incoherence, requiring only a quadratic dependence on $\eps$.
}


\vskip2mm 

If we insist on exact recovery, the only  approach which  adapts well  to the noisy situation is the infinity norm approach.  As a matter of fact, the  infinity norm bounds presented in  Section 2.4
hold in both 
noiseless and noisy case (with some modification).  The reason is this:  even  in the noiseless case, one already views the (rescalled) input matrix $p^{-1} A _{\Omega} $ as the sum of $A$ and a random matrix $E$. Thus, adding a new noise matrix $\noise$ just changes $E$ to $E+\noise$. This changes few parameters in the analysis, but the key mathematical arguments remain valid. 

The result by Abbe et al. \cite[Theorem 3.4]{abbefan2017}  yields 
the same approximation as in the noiseless case, given 

\begin{equation}  \label{eq:abbefanNoisy2017-sample-size-cond}
    p \ge C^2\eps^{-2} \mu_0^4\kappa^8(\|A\|_\infty + \sigma_\noise)^2 N^{-1}\log N,
\end{equation}
where $\sigma_\noise$ is the standard deviation of each entry of $\noise$. If we set $\epsilon < \epsilon_0/3 $ (see subsection \ref{sec:exact-recovery}), then again rounding off 
would give as an exact recovery.
Similarly, algorithm {\bf AR} works in the noisy case; see \cite{bhardwajVu2023} for the exact statement.
{
The paper only proves this fact for the case where $A$ has integer entries, but a careful examination of the proof shows that for a general absolute error tolerance $\eps$, one requires $m, n = \Theta(N)$ and
\begin{equation}  \label{eq:bhardwajVu2023-sample-size-cond}
    p \ge C \eps^{-5} \max\left\{
        C'(r, \|A\|_\infty, \mu), \log^{3.03}N
    \right\} N^{-1}\log N,
\end{equation}
where $\mu = N^{1/2}\max\{\|U\|_\infty, \|V\|_\infty\}$ and $C'$ is a term depending on $r$, $\|A\|_\infty$ and $\mu$ only.
}



\vskip2mm 

\noindent {\it Summary.} To summarize, in the noisy case, the infinity norm approach is currently the only one that yields exact recovery.  The latest results in this directions, \cite{abbefan2017} and \cite{bhardwajVu2023}, however, requires the  extra assumptions that the condition number is small and the gaps are large, respectively.  As disucssed at the end of the previous subsection, these conditions  constrast each other, and we conjecture that both of them could be removed.  This leads to the main question of this paper: 

\vskip2mm 

\noindent \textbf{Question 1.} {\it Can we use the infinity norm approach to obtain  exact recovery 
in the noisy case with only the three basic assumptions (low rank, incoherence, density) ? }

\vskip2mm 


\subsection{New results: an affirmative answer to Question 1}  \label{sec:matcom-our-algo}

\subsubsection {An overview and the general setting }
The  goal of this paper is to give an affirmative answer to   Question 1, 
in a sufficiently general  setting.  We  will  show that a variant of Algorithm \textbf{AR} will do the job. The technical core  is a new mathematical method to prove infinity norm estimates. This is  entirely different from all previous techniques, and is  of independent interest. 

\vskip2mm 

This affirmative answer  is not only a unifying result for the noisy case, providing the first and efficient exact recovery algorithm using only three basic assumptions.

Resctriced to the noiseless case, it also provides some improvements.  As discussed in Section 1.4, in the noiseless case, the only approach which does not need  extra assumptions is the nuclear norm minimization
(subsection 2.1).  However,  the running time for this approach is not great, and various attemps have been made to improve the running time, leading to spectral algorithms such as those by Keshavan et al., at the cost of new assumptions, most notably the small condition number one.

Our algorithm achieve both goals. It uses only the three basic assumptions, and is also the simplest from the practical view point. 
Indeed, the algorithm is basically a truncated SVD, which is  effective in both theory and practice. Compared to the spectral approach by Keshavan et al, it does not need the second, cleaning,  phase. 
Finally, the  density bound  is comparable to all previous works; see Remarks \ref{comparison} and \ref{quadratic}.


\begin{setting}[Matrix completion with noise]  \label{set:matrix-completion}
    Consider the truth matrix $A$, the observed set  $\Omega$, and noise matrix $\noise$. We assume  
    \begin{enumerate}
        \item \emph{Known bound on entries:}  We assume $\|A\|_\infty \le K_A$ for some known parameter $K_A$.
        This is the case for most real-life applications, as entries have physical meaning. 
        For instance, in the Netflix Challenge $K_A=5$.

        \item \emph{Known bound on rank:} We 
        {\it do not assume}  the knowledge of the rank $r$, but assume that we know some 
        upper bound  $r_{\max} $.

        \item \emph{Independent, bounded, centered noise:} $\noise$ has independent entries satisfying $\E{\noise_{ij}} = 0$ and $\E{|\noise_{ij}|^l} \le K_\noise^l$ for all $l\in \N$ and $i\in [m]$, $j\in [n]$. We assume the knowledge of the upper bound 
        $K_{\noise}$. We do not require the entries to have the same distribution or even the same variance. 
\end{enumerate}
\end{setting}

We allow the  parameters  $r$, $r_{\max}$, $K_A$, $K_\noise$ to depend on $m$ and $n$. 

\subsubsection {Our algorithm and theorem}

We propose the following algorithm to recover $A$:

\begin{algorithm}[\textbf{AR2}]  \label{algo:matrix-completion}
    Input: the $m\times n$ matrix $A_{\Omega, \noise}$ and the discritization unit $\epsilon_0$ of $A$'s entries.
    
    \begin{enumerate}
        \item \emph{Sampling density estimation:}
        Let $\hat{p} \defeq (mn)^{-1}|\Omega|$.

        \item \emph{Rescaling:}
        Let $\hat{A} :=  \hat{p}^{-1}A_{\Omega, \noise}$.
    
        \item \emph{Low-rank approximation:}
        Compute the truncated SVD $\hat{A}_{r_{\max}} = \sum_{i = 1}^{r_{\max}} \hat{\sigma}_i \hat{u}_i \hat{v}_i^T$.
        
        Take the largest index $s \le r_{\max} - 1$ such that $\hat{\sigma}_s - \hat{\sigma}_{s + 1} \ge 20(K_A + K_\noise)\sqrt{\frac{r_{\max}(m + n)}{\hat{p}}}$.
        
        If no such $s$ exists, take $s = r_{\max} $.
        Let $\hat{A}_s \defeq \sum_{i\le s} \hat{\sigma}_i \hat{u}_i \hat{v}_i^T$,


        \item \emph{Rounding off:}
        Round each entry of $\hat{A}_s$ to the nearest multiple of $\epsilon_0$.
        Return $\hat{A}_s$.
        
    \end{enumerate}
\end{algorithm}

Compared to \nameref{algo:bhardwajVu2023}, a minor difference is that 
we use an estimate $\hat{p}$ of $p$, which is very accurate with high probability.
We next use a different cutoff point for the truncated SVD step that does not require knowing the parameter $\mu_0$.

From the complexity view point, our algorithm very fast. It is basically just truncated SVD, 
taking only $O(|\Omega|r) = O(pmnr)$ FLOPs.
Our main theorem below gives sufficient conditions for exact recovery.

\begin{theorem}  \label{thm:matrix-completion}
    There is a universal constant $C > 0$ such that the following holds.
    Suppose $r_{\max} \le \log^2 N$.
    Under the model \ref{set:matrix-completion}, assume the following:
    \begin{itemize}
        \item \emph{Large signal:} $\sigma_1 \ge 100r \mcK\sqrt{\frac{r_{\max}N}{p}}$, 
        for $\mcK \defeq K_A + K_\noise$.
        
    
        \item \emph{Sampling density:}
        \begin{equation}  \label{eq:matcom-sample-density-cond}
        \begin{aligned}
            p \ge C\left(\frac{1}{m} + \frac{1}{n}\right)
            \max \left\{
                \log^4N, \
                \frac{r^3 \mcK^2}{\eps_0^2}
                \left(1 + \frac{\mu_0^2}{\log^2 N}\right)
            \right\} \log^6N .
        \end{aligned}
        \end{equation}
    \end{itemize}
    Then with probability $1 - O(N^{-1})$, the first three steps of \nameref{algo:matrix-completion} recovers every entry of $A$ within an absolute error $\eps_0/3$.
    Consequently, if all entries  are multiples integer 
    of $\eps_0$, the rounding-off step recovers $A$ exactly.
\end{theorem}

\subsubsection{Remarks}

Notice that we have  removed the gap condition from \cite{bhardwajVu2023}, and thus  obtained 
an exact recovery algorithm, using only the three basic assumptions: low rank,
incoherence, density.  

\vskip2mm 

{\it Well, almost! }  The reader, of course,  has noticed that we have (formally, at least) a new  assumption that the leading singular value $\sigma_1$ of $A$ has to be sufficiently large (large signal). But is this really a new assumption ?

It is clear that, even at heuristic level, that an assumption on the magnitude of 
$\sigma_1$ is necessary.   From a pure engineering view point,  if the intensity of the noise dominates the signal of 
the data, then the noisy data is often totally corrupted. 
Now comes the mathematical rigour. It is a well known fact in random matrix theory, called the BBP threshold phenomenon \cite{deform:baik2005}, that if 
$\| \noise \| \ge c \|A\| = c \sigma_1$, for a specific constant $c$, then the sum 
$A+\noise$ behaves like a random matrix; see \cite{deform:baik2005,deform:feral2006,deform:peche2006,deform:capitaine2007,deform:benaych-georges2011,deform:haddadi2021} for many results. For instance, the leading singular vectors of 
$A+\noise$ look totally random and have nothing to do with the leading singular vectors of $A$. This shows that there  is no chance that
one can recover $A$ from the (even fully observed) noisy matrix $A+\noise $.  In the results discussed earlier, this large signal assumption  (if not explictly stated) is  implicit 
in the setting of the model.

\vskip2mm 

 Second, we emphasize that the bound required for $\sigma_1$ is very mild. In most 
  cases, it is automatically satisfied by  the simple fact that 
 $\sigma_1 ^2 \ge r^{-1} \| A \|_F ^2$. 
    To see this, let us consider the base case when $ n = \Theta(m)$,  $\frac{1}{n} \| A \|_F = \Theta(1)$ (a normalization, for convenience), $r = O(1)$ .
    In this case, 
    
    \begin{equation}  \label{boundonnorm}
        \sigma_1 \ge r^{-1} \| A \| _F  = \tilde \Omega (n ).
    \end{equation}

    If we assume $K_A, K_{\noise} =  O(1)$, then our requirement on $\sigma_1$ in the theorem becomes 
    \begin{equation*}
        \sigma_1  =   \Theta (\sqrt {n/p}), 
    \end{equation*}
    which is guaranteed automatically (with room to spare) by \eqref{boundonnorm}, since we need $p =\Omega ( \log n/n )$.  As a matter of fact, the requirement $ \sigma_1  =   \Theta (\sqrt {n/p})$ is consistend with the BBP phenomenon discussed above.  Indeed, if we consider the rescalled input 
    
    $$p^{-1} A_{\Omega, \noise } = p^{-1} A _{\Omega } + p^{-1} \noise , $$ then 
    the $\sqrt {n/p} $ is the order of magnitude of the 
    spectral norm of noisy part $\| \frac{1}{p} \noise_{\Omega} \|$, and the spectral norm of 
    $p^{-1} A_{\Omega}$ is approximately $\| A \|$, with high probability. (It is well known that one can approximate the spectral norm by random sampling.)



\begin{remark}[Density bound] \label{comparison} 
    The condition \eqref{eq:matcom-sample-density-cond} looks complicated, but in the base case where $A$ has constant rank, constantly bounded entries, and uniformly random singular vectors, we have $\mu_0 = O(\log N)$, $K_A, K_\noise = O(1)$ and 
    $r_{\max} = O(1)$, it reduces to
    \begin{equation}  \label{eq:matcom-sample-density-cond-simple}
        p \ge C \max \left\{
            \log^4N, \ \eps_0^{-2}
        \right\}
        \left(m^{-1} + n^{-1}\right) \log^6N,
    \end{equation}
    which is equivalent to $|\Omega| \ge CN\log^6 N\max\{\log^4 N, \eps_0^{-2}\}$ in the uniform sampling model.
    The power of $\log N$ can be further reduced but the details are tedious, and the improvement is not really important from the pratical view point.

Even when reduced to the noiseless case, this bound is comparable, up to a polylogarithmic factor, to all previous results, while having the key advantage that it does not contain a  power of the condition number. 
\end{remark}

\begin{remark}[Quadratic growth in $1/\epsilon_0$] \label{quadratic}
    Notice that $|\Omega |$ is  $O(N\log^{10}N)$ until $\eps_0  < \log^{-2}N$, then it grows quadratically with $1/\eps_0$.
    Interestingly, this quandratic dependence also appears in all results discussed in Section \ref{sec:matcom-noisy}, for both RMSE and exact recoveries.
\end{remark}

\begin{remark}[Bound on $r_{\max}$]  \label{remark:r_max-bound}
    The condition $r_{\max} \le \log^2 N$ in Theorem \ref{thm:matrix-completion} can be avoided, at the cost of a more complicated sampling density bound.
    The full form of our bound is
    \begin{equation}  \label{eq:matcom-sample-density-cond-full}
    \begin{aligned}
        p \ge C\left(\frac{1}{m} + \frac{1}{n}\right)
        \max \left\{
            \log^{10}N, \
            \frac{r^4r_{\max}\mu_0^2\mcK^2}{\eps_0^2},
            \frac{r^3 \mcK^2}{\eps_0^2}
            \left(1 + \frac{\mu_0^2}{\log^2 N}\right)
            \left(1 + \frac{r^3\log N}{N}\right) \log^6N
        \right\}.
    \end{aligned}
    \end{equation}
    The proof of the more general version of Theorem \ref{thm:matrix-completion} with Eq. \eqref{eq:matcom-sample-density-cond-full} replacing Eq. \eqref{eq:matcom-sample-density-cond} will be in Appendix \ref{sec:matrix-completion-full-proof}.
    We do not know of a natural setting where one expects $r_{\max} > \log^2 N$. Nevertheless this shows that our technique does not require any extra condition besides the large signal assumption.
\end{remark}

{\bf \noindent  Structure of the rest of the paper.} In the next section, we will prove Theorem \ref{thm:matrix-completion}, asserting the correctness of our algorithm, \nameref{algo:matrix-completion}.
We will reframe the problem from a matrix perturbation perspective, then introduce 
our main tool, Theorem 
\ref{thm:dk-matcom}.  We will give a short proof of Theorem \ref{thm:matrix-completion} using Theorem \ref{thm:dk-matcom}. This concludes the next section. 

In Section 3, we discuss Theorem \ref{thm:dk-matcom} from the matrix perturbation point of view. We first discuss the classical Davis-Kahan-Wedin theorem, which gives a perturbation bound in the spectral norm, together with a number of recent improvements. 
These improvements make use of modern assumptions on $A$ and $E$, but still using the spectral norm. Our main mathematical  result of the paper is Theorem \ref{thm:main-deterministic}. This has been movitaved by these rencent improvements, but focuses on the infinity norm.  

Theorem \ref{thm:main-deterministic} is determinisitc
and is sharp under quite general conditions. Furthermore, it is interesting to point out that in the deterministic setting, the critical {\it incohenrence } assumption is replaced by a more general assumption on the relation between the matrix $E$ and the singular vectors of $A$. 

The form of incohenrece discussed in the introduction only takes shape when we start to apply the deterministic theorem to the random setting, in which case the matrix $EE^T$ has  a particular form. 

We believe that Theorem \ref{thm:main-deterministic} is of independent interest. To start, there have been very few 
non-trivial estimates in matrix theory which involves the infinity norm. Next, our mathematical approach to the 
problem is new, and totally different from the previous ones. It combines tools from the classical perturbation theory (contour representation), with new ideas from combinatorics and random matrix theory. 

Most of the rest of the paper is devoted to the proof of Theorem \ref{thm:main-deterministic}, which is fairly technical. We write a brief summary of the proof at the beginning of Section 4, which leads to a number of key lemmas, whose proofs will appear later. 

We next use Theorem \ref{thm:main-deterministic}
to derive a random version, Theorem \ref{thm:main-random}(where $E$ is random). Theorem 
\ref{thm:dk-matcom} mentioned above is an easy corollary of 
Theorem \ref{thm:main-random}; for the proof see the last part of Section 3.

 The deriviation of Theorem \ref{thm:main-random} from Theorem \ref{thm:main-deterministic} is not straightforward. 
 It has turned out that 
some parameters that appear in Theorem \ref{thm:main-deterministic} are not easy to estimate in the random setting. This task requires a delicate moment computation, and we need to dedicate the whole Section 5 for the proof of Theorem \ref{thm:main-random}.

\section{Proof of guarantees for recovery}  \label{sec:matcom-proof}

\subsection{Proof sketch}

Theorem \ref{thm:matrix-completion} aims to bound the difference $\hat{A}_s - A$ in the infinity norm.
We achieve this using  a series of intermediate comparisons, outlined below.
\begin{enumerate}
    \item Let $\rho \defeq \hat{p}/p$.
    We bound $\|A - \rho^{-1}A\|_\infty$.
    As shown by a Chernoff bound, $\rho$ is very close to $1$, making this error small.
    It remains to bound
    \begin{equation*}
        \|\rho^{-1}A - \hat{A}_s\|_\infty = \rho^{-1}\|A - (p^{-1}A_{\Omega, \noise})_s\|_\infty.
    \end{equation*}
    Let $\tilde{A} \defeq p^{-1}A_{\Omega, \noise}$, this is equivalent to bounding $\|A - \tilde{A}_s\|_\infty$, since $\rho = \Theta(1)$.

    \item We bound $\|A - A_s\|_\infty$.
    Some light calculations give $\|A - A_s\|_\infty \le \sigma_{s + 1}\|U\|_{2,\infty}\|V\|_{2,\infty}$.
    Using the fact that $\hat{\sigma}_i - \hat{\sigma}_{i + 1}$ is small for all $i > s$, we can deduce that $\hat{\sigma}_{s + 1}$ is small, making $\sigma_{s + 1}$ small too.
    Coupled with the incoherence property, this error will be small.

    \item We bound $\|A_s - \tilde{A}_s\|_\infty$.
    Most of the heavy lifting is done here.
    We discuss it in detail below.
\end{enumerate}
Observe that $\E{A_\Omega} = pA$ and $\E{\noise_\Omega} = 0$, we have
\begin{equation*}
    \mathbf{E}\bigl[\tilde{A}\bigr]
    = \E{p^{-1}A_{\Omega, \noise}}
    = \E{p^{-1}(A_\Omega + \noise_\Omega)}
    = \E{p^{-1}A_\Omega} + \E{p^{-1}\noise_\Omega} = A.
\end{equation*}
Let $E \defeq \tilde{A} - A$.
The above shows that $E$ is a random matrix with mean $0$.
From a \emph{matrix perturbation theory} point of view, $\tilde{A}$ is an \emph{unbiased} perturbation of $A$.
Establishing a bound on $(A + E)_s - A_s$ in the infinity norm is one of the major goals of perturbation theory, and is the main technical contribution of our paper.

Since both $A_\Omega$ and $\noise$ have independent entries, so does $E$.
The entries of $E$ satisfy
\begin{equation}  \label{eq:matcom-E-defn}
    E_{ij} = \begin{cases}
    p^{-1}(A_{ij} + \noise_{ij}) - A_{ij}
    = A_{ij}(p^{-1} - 1) + p^{-1}\noise_{ij}
    & \text{ with probability } p,
    \\
    -A_{ij}
    & \text{ with probability } 1 - p.
\end{cases}
\end{equation}
Consider the moments of $E_{ij}$.
For each $l\ge 2$, e have
\begin{equation}  \label{eq:matcom-E-moments-bound}
\begin{aligned}
    \E{|E_{ij}|^l}
    & = \frac{\E{|A_{ij}(1 - p) + \noise_{ij}|^l}}{p^{l - 1}} + (1 - p)|A_{ij}|^l
    \le \frac{(K_A(1 - p) + K_\noise)^l}{p^{l - 1}} + (1 - p)K_A^l
    \\
    & \le \frac{1}{p^{l - 1}} \left(
        \sum_{k = 0}^{l - 1} \binom{l}{k} K_A^k(1 - p)^k K_\noise^{l - k}
        + K_A^l(1 - p)^l + p^{l - 1}(1 - p)K_A^l
    \right)
    \\
    & \le \frac{1}{p^{l - 1}} \left(
        \sum_{k = 0}^{l - 1} \binom{l}{k} K_A^k K_\noise^{l - k}
        + K_A^l
    \right)
    = \frac{(K_A + K_\noise)^l}{p^{l - 1}}
    .
\end{aligned}
\end{equation}

For the third step, we prove the following theorem. This is the technical core of the paper. The essence of this theorem is that under certain conditions, a low rank approximation of the noisy matrix $A+Z$ approximates $A$ very well in the infinity norm. This can be seen as an improved version of the main theorem of \cite{bhardwajVu2023}. The approach here is entirely different from that of \cite{bhardwajVu2023}. 

\begin{theorem}  \label{thm:dk-matcom}
    Consider a fixed matrix $A\in \R^{m\times n}$. and a random matrix $E \in \R^{m\times n}$ with independent entries satisfying $\E{E_{ij}} = 0$ and
    $\E{|E_{ij}|^l} \le p^{1 - l}K^l$ for some $K > 0$ and $0 < p < 1$.
    Let $\tilde{A} = A + E$.
    Let $s\in [r]$ be an index satisfying
    \begin{equation*}
        \delta_s \defeq \sigma_s - \sigma_{s + 1} \ge 40rK\sqrt{N/p},
    \end{equation*}
    There are constant $C$ and $C'$ such that, if $p \ge C(m^{-1} + n^{-1})\log N$ where $N = m + n$, then
    \begin{equation}  \label{eq:dk-matcom}
        \|\tilde{A}_s - A_s\|_\infty
        \le C' \frac{(\log N + \mu_0)\log^2 N}{\sqrt{mn}}
        \cdot
        r\sigma_s \left(
            \frac{K}{\sigma_s}\sqrt{\frac{N}{p}}
            + \frac{rK\sqrt{\log N}}{\delta_s\sqrt{p}}
            + \frac{r^2\mu_0 K\log N}{p\delta_s\sqrt{mn}}
        \right).
    \end{equation}
\end{theorem}

\begin{remark}  \label{remark:dk-matcom-deterministic}

    This theorem essentially gives a bound on the perturbation of the best rank-$s$ approximation, given that the perturbation on $A$ is random with sufficiently bounded moments and the singular values up to rank $s$ are sufficiently separated from the rest. In matrix theorem, we know very few estimates using the inifnity norm, and we believe that this theorem is of independent interest.

    In practice, when applied to the matrix completion problem, 
    this separation property is guaranteed thanks to the fact that the matrix $A$ has low rank, thus such index $s$ exists. We do not need to make an extra assumption here.

    Consider the case $m = \Theta(n)$, making both $\Theta(N)$.
    Suppose either of the following two popular scenarios: $A$ is completely incoherent, namely $\mu_0 = O(1)$;
    or $A$ has independent Gaussian $N(0, 1)$ entries, making $\mu_0 = O(\log N)$.
    The theorem is just the random version of the following:
    \begin{equation}  \label{eq:dk-matcom-deterministic}
        \text{If } \delta_s \ge C_1r\|E\|,
        \quad
        \text{then }
        \|\tilde{A}_s - A_s\|_\infty
        \le \frac{C'\log^3 N}{N}
        r\sigma_s \left(\frac{\|E\|}{\sigma_s} + \frac{r\|U^TEV\|_\infty}{\delta_s} \right).
    \end{equation}
    To see this, note that in the above setting, the first factor in the bound is simply $O\left(\frac{\log^3N}{N}\right)$.
    As shown near the end (Section \ref{sec:main-proof-random}), with probability $1 - O(N^{-1})$,
    \begin{equation*}
        \|E\| = O\left(K\sqrt{\frac{N}{p}}\right),
        \quad
        \|U^TEV\|_\infty = O\left(
            \frac{r\mu_0 K\log N}{p\sqrt{mn}} + \frac{K\sqrt{\log N}}{\sqrt{p}}
        \right).
    \end{equation*}
    Plugging the above into Eq. \eqref{eq:dk-matcom-deterministic} recovers Eq. \eqref{eq:dk-matcom}.
    Eq. \eqref{eq:dk-matcom-deterministic} is important because it reveals the structure of the problem.
   
\end{remark}

\subsection{Proof of  Theorem \ref{thm:matrix-completion} using Theorem \ref{thm:dk-matcom}} 

In this section, we give a short proof of Theorem \ref{thm:matrix-completion} using Theorem \ref{thm:dk-matcom}.

\begin{proof}
    For convenience, let $K \defeq K_{A, \noise} = K_A + K_\noise$.
    Recall that in the discussion, we defined the following terms:
    \begin{equation*}
        \rho \defeq \frac{\hat{p}}{p},
        \quad \tilde{A} = \rho \hat{A} = p^{-1}A_{\Omega, \noise},
        \quad E = \tilde{A} - A,
    \end{equation*}
    and $E$ is a random matrix with mean $0$, independent entries satisfying $\E{|E_{ij}|^l} = p^{1 - l}K^l$ by Eq. \eqref{eq:matcom-E-moments-bound}.
    Recall that we have the SVD $\hat{A} = \sum_i \hat{\sigma}_i\hat{u}_i\hat{v}_i^T$.
    Similarly, denote the SVD of $\tilde{A}$ by
    \begin{equation*}
        \tilde{A} = \sum_i^{\min\{m, n\}} \tilde{\sigma}_i \tilde{u}_i \tilde{v}_i^T.
    \end{equation*}
    We have the relation $\tilde{u}_i = \hat{u}_i$, $\tilde{v}_i = \hat{v}_i$ and $\tilde{\sigma}_i = \rho^{-1}\hat{\sigma}_i$ for each $i$.
    
    From the sampling density assumption, a standard application of concentration bounds \cite{hoeffding1963,chernoff1952} guarantees that, with probability $1 - O(N^{-2})$,
    \begin{equation}  \label{eq:matcom-proof-fact1}
        0.9 \le 1 - \frac{1}{\sqrt{N}}
        \le 1 - \frac{\log N}{\sqrt{pmn}}
        \le \rho
        \le 1 + \frac{\log N}{\sqrt{pmn}}
        \le 1 + \frac{1}{\sqrt{N}}
        \le 1.1
        .
    \end{equation}
    Furthermore, an application of well-established bounds on random matrix norms gives
    \begin{equation}  \label{eq:matcom-proof-fact2}
        \|E\|
        \le 2K\sqrt{N/p}
        ,
    \end{equation}
    with probability $1 - O(N^{-1})$.
    See \cite{bandeira2014,vu2005}, \cite[Lemma A.7]{tranVu2022} or \cite{bandeira2014} for detailed proofs.
    Therefore we can assume both Eqs. \eqref{eq:matcom-proof-fact1} and \eqref{eq:matcom-proof-fact2} at the cost of an $O(N^{-1})$ exceptional probability.
    %
    %
    The sampling density condition (Eq. \eqref{eq:matcom-sample-density-cond} is equivalent to the conjunction of two conditions:
    \begin{align}
        \label{eq:matcom-sample-density-cond-2}
        p & \ge C\left(\frac{1}{m} + \frac{1}{n}\right)\log^{10}N,
        \\
        \label{eq:matcom-sample-density-cond-3}
        p & \ge \frac{Cr^3 \mcK^2}{\eps^2}
        \left(1 + \frac{\mu_0^2}{\log^2 N}\right)
        \left(\frac{1}{m} + \frac{1}{n}\right)\log^6N.
    \end{align}


    Before entering the three steps outlined in the proof sketch, we show that the SVD step is guaranteed to choose a valid $s\in [r]$ such that $\hat{\delta}_s \ge 20K\sqrt{r_{\max}N/\hat{p}}$.
    Choose an index $l\in [r]$ such that $
        \delta_l \ge \sigma_1/r
    $, which exists since $\sum_{l\in [r]} \delta_l = \sigma_1$.
    We have, by Weyl's inequality,
    \begin{equation*}
        \tilde{\delta}_l \ge \delta_l - \|E\| \ge \frac{\sigma_1}{r} - 2K\sqrt{\frac{N}{p}}
        \ge (100r_{\max}^{1/2} - 4)K\sqrt{\frac{N}{p}} \ge 90K\sqrt{\frac{r_{\max}N}{p}}.
    \end{equation*}
    Therefore
    \begin{equation*}
    \begin{aligned}
        \hat{\delta}_l \ge \rho^{-1}\tilde{\delta_l}
        \ge 90\rho^{-1}\mcK\sqrt{\frac{r_{\max}N}{p}}
        \ge 80\rho^{-1/2}\mcK\sqrt{\frac{r_{\max}N}{p}}
        = 80\mcK\sqrt{\frac{r_{\max}N}{\hat{p}}},
    \end{aligned}
    \end{equation*}
    so the cutoff point $s$ is guaranteed to exist.
    To see why $s\in [r]$, note that, again by Weyl's inequality,
    \begin{equation*}
        \tilde{\delta}_{r + 1} \le \tilde{\sigma}_{r + 1}
        \le \sigma_{r + 1} + \|E\| = \|E\| \le 2K\sqrt{N/p}.
    \end{equation*}
    Therefore,
    \begin{equation*}
        \hat{\delta}_{r + 1} = \rho^{-1}\tilde{\delta}_{r + 1}
        \le 2\rho^{-1}K\sqrt{\frac{N}{p}}
        \le 3\rho^{-1/2}K\sqrt{\frac{N}{p}}
        = 3K\sqrt{\frac{N}{\hat{p}}}
        < 20K\sqrt{\frac{r_{\max}N}{\hat{p}}}.
    \end{equation*}
    
    We want to show that the first three steps of \nameref{algo:matrix-completion} recover $A$ up to an absolute error $\eps$, namely $\|\hat{A}_s - A\|_\infty \le \eps$.
    We will now follow the three steps in the sketch.

    \begin{enumerate}
        \item \emph{Bounding $\|A - \rho^{-1}A\|_\infty$.}
        We have
        \begin{equation*}
            \|A - \rho^{-1}A\|_\infty
            = \left|\rho^{-1} - 1 \right| \|A\|_\infty
            \le \frac{K_A}{.9\sqrt{N}} < \eps/4.
        \end{equation*}
    
        \item \emph{Bounding $\|A_s - A\|_\infty$.}
        Firstly, we bound $\sigma_{s + 1}$.
        We have, since $\tilde{\delta}_s$ is the last singular value gap of $\tilde{A}$ such that $\tilde{\delta}_s \ge 20\rho^{-1/2}\mcK\sqrt{\frac{r_{\max}N}{p}}$, then
        \begin{equation}  \label{eq:matcom-proof-sigma(s+1)-bound}
            \sigma_{s + 1} \le \tilde{\sigma}_{s + 1} + \|E\|
            \le 20\rho^{-1/2}r\mcK \sqrt{\frac{r_{\max}N}{p}} + 2\mcK \sqrt{\frac{N}{p}}
            \le 22r\mcK\sqrt{\frac{r_{\max}N}{p}}.
        \end{equation}
        For each fixed indices $j, k$, we have
        \begin{equation*}
        \begin{aligned}
            |(A_s - A)_{jk}|
            & = \left| U_{j, \cdot}^T\Sigma_{[s + 1, r]}V_{k, \cdot} \right|
            \le \sigma_{s + 1} \|U\|_{2, \infty} \|V\|_{2, \infty}
            \le 22r\mcK\sqrt{\frac{r_{\max}N}{p}} \frac{r\mu_0}{\sqrt{mn}}
            \\
            &
            = \sqrt{
                \frac{22^2r^4r_{\max}\mu_0^2\mcK^2}{p}
                \left(\frac{1}{m} + \frac{1}{n}\right)
            }
            \le \sqrt{
                \frac{22^2rr_{\max}}{C\log^4 N}
            }
            \le \eps/4,
        \end{aligned}
        \end{equation*}
        where the last inequality comes from Eq. \eqref{eq:matcom-sample-density-cond-3} and the assumption that $r_{\max} \le \log^2 N$, if $C$ is large enough.
        Since this holds for all pairs $(j, k)$, we have
        $\|A_s - A\|_\infty \le \eps/4$.

        \item \emph{Bounding $\|\tilde{A}_s - A_s\|_\infty$, with probability $1 - O(N^{-1})$.}
        The condition \eqref{eq:matcom-sample-density-cond-2} guarantees that we can apply Theorem \ref{thm:dk-matcom}.
        We will get, by Eq. \eqref{eq:dk-matcom}
        \begin{equation}  \label{eq:matcom-proof-step3-temp-1}
        \begin{aligned}
            & \|\tilde{A}_s - A_s\|_\infty
            = C' \frac{(\mu_0 + \log N)\log^2 N}{\sqrt{mn}}
            \cdot r\mcK\left(
                \sqrt{\frac{N}{p}}  +  \frac{\log N}{p}\frac{\sigma_s}{\delta_s}
            \right)
            .
        \end{aligned}
        \end{equation}
        Under the assumption $r_{\max} \le \log^2 N$, we can further simplify.
        We have
        \begin{equation*}
            \delta_s \ge \tilde{\delta}_s - 2\|E\|
            \ge 20\rho^{-1/2} \sqrt{\frac{r_{\max}N}{p}} - 4K \sqrt{\frac{N}{p}}
            \ge 18\sqrt{\frac{r_{\max}N}{p}}.
        \end{equation*}
        Therefore, by Eq. \eqref{eq:matcom-proof-sigma(s+1)-bound}, $\sigma_{s + 1} < 2\delta_s$, so
        \begin{equation*}
            \frac{\sigma_s}{\delta_s}
            = 1 + \frac{\sigma_{s + 1}}{\delta_s}
            < 1 + 2r \le 3r.
        \end{equation*}
        Back to Eq. \eqref{eq:matcom-proof-step3-temp-1}, consider the second factor of the right-hand side, we have
        \begin{equation*}
            \frac{\log N}{p}\frac{\sigma_s}{\delta_s}
            \le \frac{3r\log N}{p}
            \le 3r\sqrt{\frac{N}{p}} \frac{\log N}{\sqrt{pN}}
            \le \frac{3r}{\log^4 N} \cdot \sqrt{\frac{N}{p}}
            < \sqrt{\frac{N}{p}}.
        \end{equation*}
        Therefore Eq. \eqref{eq:matcom-proof-step3-temp-1} becomes
        \begin{equation*}
            \|\tilde{A}_s - A_s\|_\infty
            \le 2C' \frac{(\mu_0 + \log N)r\mcK\sqrt{N} \log^2 N}{\sqrt{pmn}}.
        \end{equation*}
        We want this term to be at most $\eps/4$.
        We have
        \begin{equation*}
        \begin{aligned}
            2C' \frac{(\mu_0 + \log N)r\mcK\sqrt{N}\log^2 N}{\sqrt{pmn}}
            \le \frac{\eps}{4}
            \ \iff \
            p \ge (2C')^2r^2K^2(\mu_0 + \log N)^2 \cdot \frac{N}{mn} \log^4 N
            \\
            \ \iff \
            p \ge (2C')^2r^2K^2\left(1 + \frac{\mu_0}{\log N}\right)^2
            \left(\frac{1}{m} + \frac{1}{n}\right)
            \log^6 N.
        \end{aligned}
        \end{equation*}
        This is satisfied by the condition \eqref{eq:matcom-sample-density-cond-3}, when $C$ is large enough.
        The third step is complete.
    \end{enumerate}
    Now we combine the three steps.
    We have, by triangle inequality,
    \begin{equation*}
    \begin{aligned}
        \|\hat{A}_s - A\|_\infty
        = \|\rho^{-1}\tilde{A}_s - A\|_\infty
        & \le \|A - \rho^{-1}A\|_\infty + \rho^{-1}\left(
             \|\tilde{A}_s - A_s\|_\infty + \|A_s - A\|_\infty
        \right)
        \\
        & \le \frac{\eps}{4} + 1.2\left(\frac{\eps}{4} + \frac{\eps}{4}\right)
        < .9\eps.
    \end{aligned}
    \end{equation*}
    The total exceptional probability is $O(N^{-1})$.
    The proof is complete.
\end{proof}

\section{Davis-Kahan-Wedin theorem in the infinity norm}  \label{sec:intro-DK}


Now that our matrix completion algorithm (\nameref{algo:matrix-completion}) has been verified, 
we will focus on proving Theorem \ref{thm:dk-matcom}.
From this point onwards, let us put aside the matrix completion context and simply consider 
matrix perturbation model $\tilde{A} = A + E$. Here the noise matrix $E$ can be {\it deterministic}. The heart of the matter is a perturbation estimate in the inifnity norm.  We are going to disucss this result in steps. First, we describe the setting. Next, we 
state recent developments concerning spectral norm estimates. Finally, we 
describe the new result in the inifnity norm, and make a connection with previous results.


\begin{setting}[Matrix perturbation]  \label{set:main}
    Consider a fixed $m\times n$ matrix $A$ with SVD
    \begin{equation*}
        A = U\Sigma V^T = \sum_{i=1}^r \sigma_i u_i v_i^T,
        \quad \text{ where } \sigma_1 \ge \sigma_2 \ge \ldots \sigma_r.
    \end{equation*}
    Consider a $m\times n$ matrix $E$, which can be deterministic or random, which we called the \emph{perturbation matrix}.
    Let $\tilde{A} = A + E$ be the \emph{perturbed matrix} with the following SVD:
    \begin{equation*}
        \tilde{A} = \tilde{U}\tilde{\Sigma}\tilde{V}^T
        = \sum_{i = 1}^{m\wedge n} \tilde{\sigma}_i \tilde{u}_i \tilde{v}_i^T
        \quad \text{ where } \tilde{\sigma}_1 \ge \tilde{\sigma}_2 \ge \ldots \tilde{\sigma}_r.
    \end{equation*}
    Define the following terms related to $A$ and $\tilde{A}$:
    \begin{enumerate}
        \item For $k\in [r]$, $\delta_k \defeq \sigma_k - \sigma_{k + 1}$, using $\sigma_{r + 1} = 0$, and let $\Delta_k \defeq \delta_k \wedge \delta_{k - 1}$.
        
        
        \item For $S\subset [r]$, let $\sigma_S \defeq \min \{\sigma_i: i\in S\}$ and $\Delta_S \defeq \min\{|\sigma_i - \sigma_j|: i\in S, j\in S^c\}$.
        
        \item For $S\subset [r]$, define the following matrices:
        \begin{equation*}
            V_S \defeq \bigl[v_i\bigr]_{i\in S},
            \quad
            U_S \defeq \bigl[u_i\bigr]_{i\in S},
            \quad
            A_S \defeq \sum_{i\in S} \sigma_i u_i v_i^T.
        \end{equation*}
        When $S = [s]$ for some $s\in [r]$, we also use $V_s$, $U_s$, $A_s$ respectively to denote the three above.
    \end{enumerate}
    Define analogous notations $\tilde{\delta}_k$, $\tilde{\Delta}_k$, $\tilde{\sigma}_S$, $\tilde{\Delta}_S$, $\tilde{V}_S$, $\tilde{U}_S$, and $\tilde{A}_S$ for $\tilde{A}$.
\end{setting}

\noindent\textbf{Some extra notation.} To aid the presentation, we define the notation:
\begin{itemize}
    \item $a_1 \wedge a_2 \wedge \ldots \wedge a_n \defeq \min\{a_1, a_2, \ldots, a_n\}$ for every $a_1, a_2, \ldots, a_n\in \R$.

    \item $a_1 \vee a_2 \vee \ldots \vee a_n \defeq \max\{a_1, a_2, \ldots, a_n\}$ for every $a_1, a_2, \ldots, a_n\in \R$.

    \item For each $a\in \R$, let $[a] \defeq \{x\in \Z: 1\le x\le a\}$, and $[[a, b]] \defeq \{x\in \Z: a\le x\le b\}$ for $a \le b$.
\end{itemize}

\subsection{The Davis-Kahan-Wedin theorem for the spectral norm}

One of the most well-known results in perturbation theory is the \textbf{Davis-Kahan $\sin\Theta$ theorem} \cite{daviskahan1970}, which bounds the change in eigenspace projections by the ratio between the perturbation and the eigenvalue gap.
The extension for singular subspaces, proven by Wedin \cite{wedin1972}, states that:
\begin{equation}  \label{eq:DK-original}
    \|\tilde{U}_s\tilde{U}_s^T - U_sU_s^T\|
    \vee \|\tilde{V}_s\tilde{V}_s^T - V_sV_s^T\|
    \le \frac{C\|E\|}{\delta_s}
    \ \text{ for a universal constant } C.
\end{equation}

There are three challenges if one wants to apply this theorem in order to attack our problem: 
\begin{enumerate}
    \item The inequality above only concerns the change in the singular subspace projections, while the change in the low-rank approximation $\tilde{U}_s\tilde{\Sigma}_s\tilde{V}_s^T$ is needed.

    \item The bound on the right-hand side requires the spectral gap-to-noise ratio at index $s$ to be large to be useful, which is a strong assumption.

    \item The left-hand side is the operator norm, while an infinity norm bound is needed for exact recovery after rounding.
\end{enumerate}

A key observation is that, similarly to the Frobenius norm bound in \cite{candesPlan2009}, Eq. \eqref{eq:DK-original} works for all perturbation matrices $E$.
Per the discussion in \cite{phuctranVu2023}, the worst case (equality) only happens when there are special interactions between $E$ and $A$.
A series of papers by Vu and coauthors \cite{orourkeVuWang2013,phuctranVu2023} exploited the improbability of such interactions when $E$ is random and $A$ has low rank, and improved the bound significantly. For instance, O'Rourke, Vu and Wang \cite{orourkeVuWang2013} proved the following:
\begin{equation*}  \label{eq:orourkeVuWang}
     \|\tilde{V}_s\tilde{V}_s^T - V_sV_s^T\|
     \le C \sqrt{s} \left(
        \frac{\|E\|}{\sigma_s} + \frac{\sqrt{r}\|U^TEV\|_\infty}{\delta_s} + \frac{\|E\|^2}{\delta_s\sigma_s}
     \right),
\end{equation*}
with high probability, effectively turning the \emph{noise-to-gap} on the right-hand side of Eq. \eqref{eq:DK-original} into the \emph{noise-to-signal ratio}, which can be much smaller than the former in many cases.

P. Tran and Vu then \cite{phuctranVu2023} improved the third term, at the cost of an extra factor of $\sqrt{r}$, which does not matter when $A$ has constant rank.
They showed that when
\begin{equation*}
    \frac{\|E\|}{\sigma_s} \vee \frac{2r\|U^TEV\|_\infty}{\delta_s} \vee \frac{\sqrt{2r}\|E\|}{\sqrt{\delta_s\sigma_s}}
    \le \frac{1}{8},
\end{equation*}
then
\begin{equation*}
    \|\tilde{V}_s\tilde{V}_s^T - V_sV_s^T\|
    \le C r \left(
        \frac{\|E\|}{\sigma_s} + \frac{2r\|U^TEV\|_\infty}{\delta_s} + \frac{2ry}{\delta_s\sigma_s}.
     \right),
\end{equation*}
replacing the term $\|E\|^2$ in \cite{orourkeVuWang2013} with the smaller $y \defeq \frac{1}{2}\max_{i \neq j} (|u_i^TEE^Tu_j| + |v_i^TE^TEv_j|)$.
It is clear that this quantity is at most  $\|E\|^2$. However, it  can be significantly smaller in many cases, notably when $E$ is \emph{regular} \cite{phuctranVu2023}, meaning there is a common $\overline{\sigma}$ such that:
\begin{equation}  \label{eq:regular-matrix}
    \overline{\sigma}^2 = \frac{1}{m}\sum_{i = 1}^m \Var{E_{ij}}
    = \frac{1}{n}\sum_{j = 1}^n \Var{E_{ij}}
    \ \text{ for all } i\in [m],\ j\in [n].
\end{equation}

The notion of random regular matrices cover most models of random matrices used in practice, such as Wigner matrices or genearalized Wigner matrices. 

\vskip2mm 

Their method is  easily adaptable to prove similar bounds for the deviation of other spectral entities, with respect to difference matrix norms.
The main idea is to write the difference of the two entities as a contour integral, then use a combinatorial expansion to split this integral into many subsums, each of which can be treated using tools from complex analysis, linear algebra, and combinatorics.
See Section 3 of \cite{phuctranVu2023} for a detailed discussion.

\subsection{A new Davis-Kahan-Wedin theorem for the infinity norm}


Our main result can be seen as the ìnfinity norm version of the results discussed above. 
Proving ìnfinity norm bounds is a recent trend in matrix theory. In particular, most 
recent progesses concerning universality of random matrices crucially rely on strong infinity norm bounds of the eigenvectors of the matrix {\bf cite Yau et al book which you read}. 
Obtaining a strong  infinity norm bound for a vector or a matrix is usually a highly non-trivial task,  as demonstrated in the theory of random matrices, and we hope that the method introduced in this paper will be of independent interest.

{\oldcomment{}
\begin{theorem}  \label{thm:main-deterministic}
    Consider the objects in Setting \ref{set:main}.
    Define the following terms:
    \begin{equation}  \label{eq:main-determistic-myD12} 
    \begin{aligned}
        \myD_1 & \defeq \max_{0\le a\le 10\log(m + n)} \frac{1}{\sqrt{r}}
        \max\left\{
            \frac{\left\|(EE^T)^aU\right\|_{2, \infty}}{\|E\|^{2a}}, \
            \frac{\left\|(EE^T)^aEV\right\|_{2, \infty}}{\|E\|^{2a + 1}}
        \right\},
        \\
        \myD_2 & \defeq \max_{0\le a\le 10\log(m + n)} \frac{1}{\sqrt{r}}
        \max\left\{
            \frac{\left\|(E^TE)^aV\right\|_{2, \infty}}{\|E\|^{2a}}, \
            \frac{\left\|(E^TE)^aE^TU\right\|_{2, \infty}}{\|E\|^{2a + 1}}
        \right\}.
    \end{aligned}
    \end{equation}
    Consider an arbitrary subset $S\subset [r]$.
    Suppose that $S$ satisfies
    \begin{equation}  \label{eq:main-deterministic-R123}
    \begin{aligned}
        \frac{\|E\|}{\sigma_S}
        \ \vee \ \frac{2r\|U^TEV\|_\infty}{\Delta_S}
        \ \vee \ \frac{\sqrt{2r}\|E\|}{\sqrt{\Delta_S\sigma_S}} \le \frac{1}{8},
    \end{aligned}
    \end{equation}
    Then there is a universal constant $C$ such that
    \begin{align}
        \label{eq:s-sing-vec-entry}
        \left\|\tilde{V}_S\tilde{V}_S^T - V_SV_S^T\right\|_\infty
        & \le C \myD_1^2
        r\left(
            \frac{\|E\|}{\sigma_S}
            + \frac{2r\|U^TEV\|_\infty}{\Delta_S}
            + \frac{2ry}{\Delta_S\sigma_S}
        \right),
        \\[3pt]
        \label{eq:s-sing-vec-row}
        \left\|\tilde{V}_S\tilde{V}_S^T - V_SV_S^T\right\|_{2, \infty}
        & \le C \myD_1
        r\left(
            \frac{\|E\|}{\sigma_S}
            + \frac{2r\|U^TEV\|_\infty}{\Delta_S}
            + \frac{2ry}{\Delta_S\sigma_S}
        \right),
    \end{align}
     where
     \begin{equation*}
         y \defeq \frac{1}{2}\max_{i \neq j} (|u_i^TEE^Tu_j| + |v_i^TE^TEv_j|).
     \end{equation*}
    When $S = [s]$ for some $s\in [r]$, we also have
    \begin{equation}  \label{eq:s-rank-approx-entry}
        \|\tilde{A}_s - A_s\|_\infty
        \le C
        \myD_1\myD_2
        \sigma_s r\left(
            \frac{\|E\|}{\sigma_s}
            + \frac{2r\|U^TEV\|_\infty}{\delta_s}
            + \frac{2ry}{\delta_s\sigma_s}
        \right),
    \end{equation}
\end{theorem}
}
Analogous bounds
for $U$ and $\tilde{U}$ hold, with $U$ and $V$ swapped.

It is worth noting that the term
\begin{equation*}
    r\left(
        \frac{\|E\|}{\sigma_s}
        + \frac{2r\|U^TEV\|_\infty}{\delta_s}
        + \frac{2ry}{\delta_s\sigma_s}
    \right)
\end{equation*}
is identical to the bound on $\bigl\|\tilde{V}_S\tilde{V}_S^T - V_SV_S^T\bigr\|_{\textup{op}}$ in \cite{phuctranVu2023}, which is tight when $r = O(1)$.
An routine corollary of their main result also implies that
\begin{equation*}
    r\sigma_s\left(
        \frac{\|E\|}{\sigma_s}
        + \frac{2r\|U^TEV\|_\infty}{\delta_s}
        + \frac{2ry}{\delta_s\sigma_s}
    \right)
\end{equation*}
is a tight upper bound for $\|\tilde{A}_s - A_s\|$ when $r = O(1)$.



\bigskip

{\bf \noindent Sharpness of the results.}
The terms
$\myD_1$ and $\myD_2$
play the roles of the coherence parameters in the matrix completion setting.
Practically, one replaces them with upper bounds when applying Theorem \ref{thm:main-deterministic}, as the theorem still works after such substitutions.
A trivial choice is $1/\sqrt{r}$ for both, since we have
\begin{equation*}
    \|(EE^T)^aU\|_{2, \infty} \le \|E\|^{2a},
    \quad \|(E^TE)^aEU\|_{2, \infty} \le \|E\|^{2a + 1},
\end{equation*}
and analogously for $V$.
This is the best we can have in worst case analysis.

However, if $E$ and $A$ interact favorably, then we can get much better estimates. Let us first consider a bound from below. 
Setting $a = 0$, 
we get from Eq. \eqref{eq:main-determistic-myD12} the lower bounds
\begin{equation*}
    \myD_1 \ge \frac{1}{\sqrt{r}}\|U\|_{2, \infty} = \sqrt{\frac{\mu(U)}{m}},
    \quad
    \myD_2 \ge \frac{1}{\sqrt{r}}\|V\|_{2, \infty} = \sqrt{\frac{\mu(V)}{n}},
\end{equation*}
where $\mu(U)$ and $\mu(V)$ are the individual incoherence parameters from Eq. \eqref{eq:candesRecht2009-A0}.

If these lower bounds are the truth,  then one gets, in philosophy, the following bounds from Theorem \ref{thm:main-deterministic} (when $r = O(1)$):
\begin{equation}  \label{eq:D\upperE-philosophy}
\begin{aligned}
    \left\|
        \tilde{V}_S\tilde{V}_S^T - V_SV_S^T
    \right\|_\infty
    & \le C\frac{\mu(V)}{n} \left\|
        \tilde{V}_S\tilde{V}_S^T - V_SV_S^T
    \right\|,
    \\[3pt]
    \left\|
        \tilde{V}_S\tilde{V}_S^T - V_SV_S^T
    \right\|_{2, \infty}
    & \le C\sqrt{\frac{\mu(V)}{n}} \left\|
        \tilde{V}_S\tilde{V}_S^T - V_SV_S^T
    \right\|,
    \\[3pt]
    \left\|
        \tilde{A}_s - A_s
    \right\|_\infty
    & \le C\sqrt{\frac{\mu(U)\mu(V)}{mn}} \left\|
        \tilde{A}_s - A_s
    \right\|.
\end{aligned}
\end{equation}
These are the best possible bounds one can hope to produce with Theorem \ref{thm:main-deterministic}.
\emph{But how good are they ?} To answer this question, let us consider a simple case where $r = O(1)$, $\mu(V)= O(1)$, and
$m = \Theta(n)$.
Assume the best possible case for the parameters $\tau_2$, which is that $\tau_2 = \sqrt{\mu(V)/n} = O(n^{-1/2})$.
In this case, Eq. \eqref{eq:s-sing-vec-entry} asserts that 
\begin{equation*}
    \left\| \tilde{V}_S \tilde{V}_S^T - V_SV_S^T \right\|_{\infty}
    = O \left(
        \frac{1}{n} \left\| \tilde{V}_S \tilde{V}_S^T - V_SV_S^T \right\|
    \right).
\end{equation*}
On the other hand, we have
\begin{equation*}
    \left\| \tilde{V}_S \tilde{V}_S^T - V_SV_S^T \right\|_{\infty}
    = \Omega \left(
        \frac{1}{n} \left\| \tilde{V}_S \tilde{V}_S^T - V_SV_S^T \right\|_F
    \right)
    = \Omega \left(
        \frac{1}{n} \left\| \tilde{V}_S \tilde{V}_S^T - V_SV_S^T \right\|.
    \right)
\end{equation*}
Therefore, our bound says that in the best case scenario, the largest entry of the matrix is of the same magnitude as the average one, making Eq. \eqref{eq:s-sing-vec-entry} sharp.
The sharpness (in the best case) of Eq. \eqref{eq:s-sing-vec-row} and Eq. \eqref{eq:s-rank-approx-entry} can be argued similarly.

\vskip2mm 

Now, the question is: \emph{How close can we come to the best-case scenarios for $\myD_1$ and $\myD_2$?}
It turns out that when $E$ is random and the entries of $U$ (and respectively $V$) are sufficiently spread out, then, given the above lowr bounds, we can achieve very close upper bounds for $\myD_1$ (and respectively $\myD_2$), which are only a polylogartihmic term appart (see the discussion on the random setting below)

\vskip2mm

{\bf \noindent The general incoherence assumption.} Let us point out that in Theorem \ref{thm:main-deterministic}, the incoherence assumption 
(hidden in the definition of $\tau_1, \tau2$ is quite different from 
the incoherence assumption discussed in the introduction, concerning the matrix completion problem. The original incoherence assumption starts to appear when we consider $E$ to be a random matrix. Assume, for simplicity, that $E$ is $n$ by $n$ with iid  entries $N(0,1)$, then  it is easy to show that the matrix $EE^T/ \| E \|^2$ is (with high probability) is close to 
the indentity matrix $I_n$.  Thus, the quantity 
$\frac{ \| (EE^T) U \|_{2,\infty}} { \| E \|^2 } $ is small if the rows of 
$U$ are short.

\bigskip

{\bf \noindent The random setting.}
Now we derive a corrollary which addresses the case when $E$ is random with independent entries.
Recall that $N = m + n$.
In this case, the term $\|U^TEV\|_\infty$ is the maximum among sums of independent random variables, which is $O\bigl(\stdE(\sqrt{\log N} + \upperE\|U\|_\infty\|V\|_\infty\log N)\bigr)$ by the Bernstein bound \cite{hoeffding1963,chernoff1952}.The term $y$ has been analyzed in \cite{phuctranVu2023} and mentioned above.
We use the trivial upper bound $y \le \|E\|^2 = O(\stdE^2N)$, which is enough to prove the next theorem.

Regarding $\myD_1$ and $\myD_2$,
our analysis later will give the estimates
\begin{equation*}
    \myD_1 = O\left(
        \log N\sqrt{\frac{\mu(U)}{m}}
        + \frac{\log^{3/2}N}{\sqrt{N}}
        + \frac{\upperE \log^3N}{\sqrt{N}}\cdot \sqrt{\frac{\mu(V)}{n}}
    \right),
\end{equation*}
and symmetrically for $\myD_2$ by swapping $U$ and $V$.
In the simple case where $\mu(U), \mu(V) = O(1)$, $m = \Theta(n)$, this estimate reduces to
\begin{equation*}
    \myD_1 = O\left(\log N\sqrt{\frac{\mu(U)}{m}}\right)
\end{equation*}
whenever $M = O(N^{1/2}\log^{-2}N)$, which is off by a factor $\log N$ from the ideal case.

Combining the observations above, we get the following ``random'' version of Theorem \ref{thm:main-deterministic}:

\begin{theorem}  \label{thm:main-random}
    Consider the objects in Setting \ref{set:main}.
    Let $\eps \in (0, 1)$ be arbitrary
    Suppose $E$ is a random $m\times n$ matrix with independent entries satisfying:
    \begin{equation}  \label{eq:noise-model}
        \E{E_{ij}} = 0,
        \quad \E{|E_{ij}|^2} \le \stdE^2,
        \quad \E{|E_{ij}|^l} \le \upperE^{l - 2}\stdE^l \quad \text{ for all } \ p\in \N_{\ge 2},
    \end{equation}
    where $\upperE$ and $\stdE$ are parameters.
    Let $N = m + n$.
    Define
    \begin{equation}  \label{eq:main-random-myD12}
        \myD_1
        \defeq \frac{\|U\|_{2, \infty}\log N}{\sqrt{r}}
        + \frac{\upperE\|V\|_{2, \infty}\log^3N}{\sqrt{rN}}
        + \frac{\log^{3/2}N}{\sqrt{N}},
    \end{equation}
    and define $\myD_2$ symmetrically by swapping $U$ and $V$.
    For an arbitrary subset $S\subset [r]$, suppose
    \begin{equation}  \label{eq:main-random-S-cond}
    \begin{aligned}
        \frac{\stdE\sqrt{N}}{\sigma_S}
        \ \vee \
        \frac{
            r\stdE(\sqrt{\log N}
            + \upperE\|U\|_\infty \|V\|_\infty \log N)
        }{\Delta_S}
        \ \vee \
        \frac{\stdE\sqrt{rN}}{\sqrt{\Delta_S\sigma_S}}
        \le \frac{1}{16}.
    \end{aligned}
    \end{equation}
    Let
    \begin{equation*}
        R_S \defeq \frac{\stdE\sqrt{N}}{\sigma_S}
        + \frac{
            r\stdE(\sqrt{\log N}
            + \upperE\|U\|_\infty \|V\|_\infty \log N)
        }{\Delta_S}
        + \frac{2r\stdE^2 N}{\Delta_S\sigma_S}.
    \end{equation*}
    There are universal constants $c$ and $C$ such that:
    If $\upperE \le cN^{1/2}\log^{-5}N$
    ,
    then with probability at least $1 - O(N^{-1})$,
    \begin{align}
        \label{eq:s-sing-vec-entry-random}
        \left\|\tilde{V}_S\tilde{V}_S^T - V_SV_S^T\right\|_\infty
        & \le C \myD_2^2
        r R_S.
        \\
        \label{eq:s-sing-vec-row-random}
        \left\|\tilde{V}_S\tilde{V}_S^T - V_SV_S^T\right\|_{2, \infty}
        & \le C \myD_2
        r R_S.
    \end{align}
    Analogous bounds
    for $U$ and $\tilde{U}$ hold, with $\myD_2$ replacing $\myD_1$.
    
    When $S = [s]$ for some $s\in [r]$, we slightly abuse the notation to let
    \begin{equation*}
        R_s \defeq R_{[s]}
        = \frac{\stdE\sqrt{N}}{\sigma_s}
        + \frac{r\stdE(\sqrt{\log N} + \upperE\|U\|_\infty \|V\|_\infty\log N)}{\delta_s}
        + \frac{2r\stdE^2 N}{\delta_s\sigma_s}.
    \end{equation*}
    Then with probability $1 - O(N^{-1})$,
    \begin{equation}  \label{eq:s-rank-approx-entry-random}
        \|\tilde{A}_s - A_s\|_\infty
        \le C \myD_1 \myD_2
        r\sigma_s R_s.
    \end{equation}
    Furthermore, for each $\eps > 0$, if the term $\frac{2r\stdE^2 N}{\Delta_S\sigma_S}$ in $R_S$ is replaced with
    \begin{equation*}
        \frac{r}{\Delta_S\sigma_S} \inf \left\{
            t: \Pr\left(
                \max_{i\neq j} (|v_iE^TEv_j| + |u_iEE^Tu_j|) \le 2t
            \right) \ge 1 - \eps
        \right\},
    \end{equation*}
    then all three bounds above hold with probability at least $1 - \eps - O(N^{-1})$.
\end{theorem}

Going back to the matrix completion setting, we can use this theorem to prove Theorem \ref{thm:dk-matcom}.

\begin{proof}[Proof of Theorem \ref{thm:dk-matcom}]
    Let $\stdE = \mcK/\sqrt{p}$ and $\upperE = 1/\sqrt{p}$.
    Then for $C$ sufficiently large, $p \ge C(m^{-1} + n^{-1})\log^{10} N$ implies $\upperE \le c\sqrt{N}\log^{-5}N$, meaning we can apply Theorem \ref{thm:main-random}, specifically Eq. \eqref{eq:s-rank-approx-entry-random} for this choice of $\stdE$ and $\upperE$ if the condition \eqref{eq:main-random-S-cond} holds.
    We check it for $S = [s]$.
    Given that $\sigma_s \ge \delta_s \ge 40\mcK\sqrt{rN/p}$, we have
    \begin{equation*}
        \frac{\stdE\sqrt{N}}{\sigma_S}
        = \frac{\mcK}{\sigma_S} \sqrt{\frac{rN}{p}}
        \le \frac{1}{40\sqrt{r}} < \frac{1}{16},
        \quad \
        \frac{\stdE\sqrt{rN}}{\sqrt{\delta_s\sigma_s}}
        \le \frac{\mcK\sqrt{rN}}{\sqrt{p}\cdot 40r\mcK\sqrt{rN/p}}
        \le \frac{1}{40} < \frac{1}{16},
    \end{equation*}
    and, using the fact $\mu_0 \le N$ and the assumption $r \le \log^2 N$,
    \begin{equation*}
    \begin{aligned}
        & \frac{r\stdE(\sqrt{\log N} + \upperE\|U\|_\infty\|V\|_\infty \log N}{\delta_S}
        \le \frac{r\mcK\sqrt{\log N}}{\delta_s\sqrt{p}}
        + \frac{r^2\mcK\mu_0\log N}{\delta_s p\sqrt{mn}}
        \\
        & \le \frac{\sqrt{r}\log N}{40\sqrt{N}}
        + \frac{r^{3/2}\mu_0\log N}{\sqrt{pmnN}}
        \le \frac{\sqrt{r}\log N}{40\sqrt{N}}
        + \frac{r^{3/2}\mu_0\log N}{\sqrt{C}N\log^5 N}
        \le \frac{1}{\log N} < \frac{1}{16}.
    \end{aligned}
    \end{equation*}
    It remains to transform the right-hand side of Eq. \eqref{eq:s-rank-approx-entry-random} to the right-hand side of Eq. \eqref{eq:dk-matcom}.
    We have
    \begin{equation*}
        \myD_1
        \le \frac{\sqrt{\mu_0}\log^3N}{\sqrt{pmN}}
        + \frac{\log^{3/2}N}{\sqrt{N}} + \frac{\sqrt{\mu_0}\log N}{\sqrt{n}}
        \le \frac{\log^{3/2}N}{\sqrt{N}}
        + \frac{\sqrt{2\mu_0}\log N}{\sqrt{n}}.
    \end{equation*}
    Combining with the symmetric bound for $\myD_2$, we get
    \begin{equation*}
        \myD_1\myD_2
        \le \frac{\log^3N}{N}
        + \frac{\sqrt{\mu_0}\log^{5/2}N}{\sqrt{N}}
            \cdot \frac{\sqrt{2m} + \sqrt{2n}}{\sqrt{mn}}
        + \frac{2\mu_0\log^2N}{\sqrt{mn}}
        \le 4\log^2N
        \frac{\log N + \mu_0}{\sqrt{mn}},
    \end{equation*}
    which is the first factor on the right-hand side of Eq. \eqref{eq:dk-matcom}.
    
    Consider the term $R_s$.
    From the above, we have
    \begin{equation*}
        R_s \le \frac{\mcK}{\sigma_S} \sqrt{\frac{rN}{p}}
        + \frac{r\mcK\sqrt{\log N}}{\delta_s\sqrt{p}}
        + \frac{r^2\mcK\mu_0\log N}{\delta_s p\sqrt{mn}}
        + \frac{\mcK^2 rN}{p\delta_s\sigma_s}.
    \end{equation*}
    Since $\delta_s \ge 40K\sqrt{rN/p}$, the fourth term is at most $1/40$ of the first term.
    Removing it recovers exactly the second factor on the right-hand side of Eq. \eqref{eq:dk-matcom}.
    The proof is complete.
\end{proof}

In the next two section, we will prove the main theorems.
The technical bounds whose proof do not fit in the main body will be in Section \ref{sec:technical-lemmas}.

\section{Proof of main results}

{\oldcomment{}
As mentioned, Theorem \ref{thm:main-random} is a corollary of \ref{thm:main-deterministic} when the noise matrix is random.
In actuality, Theorem \ref{thm:main-deterministic} is a slightly simplified version of the full argument for the deterministic case and does not directly lead to the random case.
However, the reader can be assured that the changes needed to make Theorem \ref{thm:main-deterministic} imply Theorem \ref{thm:main-random} are trivial, and will be discussed when we prove the latter.

\bigskip

\textbf{Proof structure.} First, we will assume Theorem \ref{thm:main-deterministic} and use it to prove Theorem \ref{thm:main-random}, which directly implies Theorem \ref{thm:dk-matcom}.
The proof contains a novel high-probability \emph{semi-isotropic} bound for powers of a random matrix, which can be of further independent interest.

We will then discard the random noise context and prove Theorem \ref{thm:main-deterministic}.
The proof adapts the contour integral technique in \cite{phuctranVu2023}, but with highly non-trivial adjustments to handle the inifnity norm, ínstead of spectral norm as in \cite{phuctranVu2023}. The proof roughly has two steps:
\begin{enumerate}
    \item Rewrite the quantities on the left-hand sides of the bounds in Theorem \ref{thm:main-deterministic} as a power series in terms of $E$, similar to a Taylor expansion.

    \item Devise a bound that decays exponentially for each power term, and sum them up as a geometric series to obtain a bound on the quantities of interest.
    The final bound, Lemma \ref{lem:T-bound}, will be general enough to imply all three of bounds of Theorem \ref{thm:main-deterministic}.
\end{enumerate}

The structure for this section will be:
\begin{center}
    Theorem \ref{thm:main-random}
    \ $\xleftarrow[]{\quad \text{implied by} \quad}$ \
    Theorem \ref{thm:main-deterministic}
    \ $\xleftarrow[]{\quad \text{implied by} \quad}$ \
    Lemma \ref{lem:T-bound}.
\end{center}
}

\subsection{The random version: Proof of Theorem \ref{thm:main-random}} \label{sec:main-proof-random}

In this section, we prove Theorem \ref{thm:main-random}, assuming Theorem \ref{thm:main-deterministic}.
First, consider the term
\begin{equation*}
    \frac{\|E\|}{\sigma_S} \vee \frac{2r\|U^TEV\|_\infty}{\Delta_S}
    \vee \frac{\sqrt{2r}\|E\|}{\sqrt{\sigma_S\Delta_S}}
\end{equation*}
from the condition \eqref{eq:main-deterministic-R123}.
Let us replace the terms related to $E$ in the above with their respective high-probability bounds.
\begin{itemize}
    \item $\|E\|$.
    There are tight bounds in the literature.
    For $E$ following the Model \eqref{eq:noise-model}, with the assumption $\upperE \le (m + n)^{1/2}\log^{-5}(m + n)$, the moment argument in \cite{vu2005} can be used.

    \item $\|U^TEV\|_\infty = \max_{i, j} |u_i^TEv_j|$.
    These terms can be bounded with a simple Bernstein bound.
    
    \item $y = \frac{1}{2}\max_{i\neq j}(|u_iEE^Tu_j| + |v_iE^TEv_j|)$.
    The terms inside the maximum function can be bounded with the moment method.
    The most saving occurs when $E$ is a stochastic matrix, meaning its row norms and column norms have the same second moment.
    For the purpose of proving Theorem \ref{thm:main-random}, the naive bound $\|E\|^2$ suffices.
    
    
\end{itemize}

{\oldcomment{}
Upper-bounding these three is routine, which we summarize in the lemma below.
}

\begin{lemma}  \label{lem:E-bound-basic}
    Consider the objects in Setting \ref{set:main}.
    Let $E\in \R^{m\times n}$ be a random matrix satisfying Model \eqref{eq:noise-model} with parameters $\upperE$ and $\stdE$.
    Suppose $\upperE \le (m + n)^{1/2}\log^{-3}(m + n)$.
    Then with probability $1 - O((m + n)^{-2})$, all of the following hold:
    \begin{align}  \label{eq:E-norm-bound}
        & \|E\| \le 1.9\stdE\sqrt{m + n} \le 2\stdE\sqrt{m + n},
        \\
        \label{eq:E^2-isotropic-bound}
        & \max_{i\neq j}(|u_iEE^Tu_j| + |v_iE^TEv_j|)
        \le 2\|E\|^2 \le 8\stdE^2(m + n).
        \\
        \label{eq:E-isotropic-bound}
        & \max_{i, j} |u_i^TEv_j|
        \le 2\stdE(\sqrt{\log (m + n)} + \upperE\|U\|_\infty\|V\|_\infty\log (m + n)).
    \end{align}
\end{lemma}

\begin{proof}
    Eq. \eqref{eq:E-norm-bound} follows from the moment argument in \cite{vu2005}.
    Eq. \eqref{eq:E^2-isotropic-bound} follows from Eq. \eqref{eq:E-norm-bound}.
    It remains to check
    Eq. \eqref{eq:E-isotropic-bound}.
    Fix $i, j\in [r]$.
    Write
    \begin{equation*}
        u_i^TEv_j = \sum_{k\in [m], h\in [n]} u_{ik}v_{jh}E_{kh}
        = \sum_{(k, h)\in [m]\times [n]} Y_{kh},
    \end{equation*}
    where we temporarily let $Y_{kh} \defeq u_{ik}v_{jh}E_{kh}$ for convenience.
    We have $|Y_{kh}| \le \|U\|_\infty\|V\|_\infty |E_{kh}|$.
    Let $X_{kh} \defeq Y_{kh} / (\stdE \|U\|_\infty\|V\|_\infty)$, then $\{X_{kh}: (k, h)\in [m]\times [n]\}$ are independent random variables and for each $(k, h)\in [m]\times [n]$,
    \begin{equation*}
        \E{X_{kh}} = 0,
        \ \quad
        \E{|X_{kh}|^2} \le 1,
        \ \quad
        \E{|X_{kh}|^l} \le \upperE^{l - 2}
        \ \text{ for all } l\in \N.
    \end{equation*}
    We also have
    \begin{equation*}
        \sum_{k, h}\E{|X_{kh}|^2}
        = \frac{
            \sum_{k, h}u_{ik}^2v_{jh}^2\E{|E_{kh}|^2}
        }{
            \stdE^2\|U\|_\infty^2\|V\|_\infty^2
        }
        \le \frac{
            \stdE^2 \sum_{k, h}u_{ik}^2v_{jh}^2
        }{\|U\|_\infty^2\|V\|_\infty^2}
        = \frac{1}{\|U\|_\infty^2\|V\|_\infty^2}
    \end{equation*}
    By Bernstein's inequality \cite{chernoff1952},
    we have for all $t > 0$
    \begin{equation*}
        \Pr\left(
            \Bigl| \sum_{k, h} X_{kh} \Bigr| \ge t
        \right)
        \le \exp \left(
            \frac{-t^2}{\sum_{k, h}\E{|X_{kh}|^2} + \frac{2}{3}\upperE t}
        \right)
        \le \exp \left(
            \frac{-t^2}{\|U\|_\infty^{-2}\|V\|_\infty^{-2} + \frac{2}{3}\upperE t}
        \right).
    \end{equation*}
    We rescale $Y_{kh} = \stdE\|U\|_\infty\|V\|_\infty X_{kh}$ and replace $t$ with $t/(\stdE\|U\|_\infty\|V\|_\infty)$, the above becomes
    \begin{equation*}
        \Pr\left(
            \Bigl| \sum_{k, h} Y_{kh} \Bigr| \ge t
        \right)
        \le \exp \left(
            \frac{-t^2}{\stdE^2 + \frac{2}{3}\upperE\|U\|_\infty\|V\|_\infty t}
        \right).
    \end{equation*}
    Let $N = m + n$ and $t = 2\stdE(\sqrt{\log N} + \upperE\|U\|_\infty\|V\|_\infty\log N)$, we have
    \begin{equation*}
        t^2 \ge 4\stdE^2\log N,
        \ \quad
        t^2 \ge 2\upperE\|U\|_\infty\|V\|_\infty t\log N,
    \end{equation*}
    thus
    \begin{equation*}
        t^2 \ge \frac{12}{7}\Bigl(
            \stdE^2 + \frac{2}{3}\upperE\|U\|_\infty\|V\|_\infty t
        \Bigr) \log N.
    \end{equation*}
    Combining everything above, we get
    \begin{equation*}
        \Pr\left(
            |u_i^TEv_j|
            \ge 2\stdE(\sqrt{\log N} + \upperE\|U\|_\infty\|V\|_\infty\log N)
        \right)
        \le N^{-12/7}.
    \end{equation*}
    By a union bound over $(i, j)\in [r]\times [r]$,
    the proof of Eq. \eqref{eq:E-isotropic-bound} and the lemma is complete.
\end{proof}


{\oldcomment{}
Now all that remains is computing $\myD_1$ and $\myD_2$.
More precisely, since both are random, we compute a good choice of high-probability upper bounds for them.
This, however, is likely intractable since the appearance of powers of $\|E\|$ in the denominator makes it hard to analyze the right-hand sides of Eq. \eqref{eq:main-determistic-myD12}.
To overcome this, notice that the argument in Theorem \ref{thm:main-deterministic} works in the same way if, instead of being rigidly refined by Eq. \eqref{eq:main-determistic-myD12}, $\myD_1$ and $\myD_2$ are any real numbers satisfying
\begin{equation}  \label{eq:main-determistic-E-power}
\begin{aligned}
    & \myD_1 \ge \max_{a\in [[0, 10\log(m + n)]]} \
    \frac{1}{\sqrt{r}} \max\left\{
        \frac{\left\|(EE^T)^aU\right\|_{2, \infty}}{\uppnormE^{2a}}, \
        \frac{\left\|(EE^T)^aEV\right\|_{2, \infty}}{\uppnormE^{2a + 1}}
    \right\},
    \\
    & \myD_2 \ge \max_{a\in [[0, 10\log(m + n)]]} \
    \frac{1}{\sqrt{r}} \max\left\{
        \frac{\left\|(E^TE)^aV\right\|_{2, \infty}}{\uppnormE^{2a}}, \
        \frac{\left\|(E^TE)^aE^TU\right\|_{2, \infty}}{\uppnormE^{2a + 1}}
    \right\},
\end{aligned}
\end{equation}
for some upper bound $\uppnormE \ge \|E\|$.

From this point, we will discard Eq. \eqref{eq:main-determistic-myD12} and treat $(\myD_1, \myD_2, \uppnormE)$ as any tuple that satisfies Eq. \eqref{eq:main-determistic-E-power}.
Specifically, we will choose $\myD_0(U)$, $\myD_1(U)$, $\myD_0(V)$, $\myD_1(V)$ such that
\begin{equation*}
    \forall a\in [[0, 10\log(m + n)]]: \
    \myD_0(U) \ge
    \frac{1}{\sqrt{r}} \frac{\left\|(EE^T)^aU\right\|_{2, \infty}}{\uppnormE^{2a}},
    \quad
    \myD_1(U) \ge
    \frac{1}{\sqrt{r}} \frac{\left\|(E^TE)^aE^TU\right\|_{2, \infty}}{\uppnormE^{2a + 1}}
\end{equation*}
and symmetrically for $\myD_0(V)$ and $\myD_1(V)$, with $E$ and $E^T$ swapped.
We can then simply let $\myD_1 = \myD_0(U) + \myD_1(V)$ and $\myD_2 = \myD_1(U) + \myD_0(V)$.

This is equivalent to bounding terms of the form
\begin{equation*}
    \|e_{m, k}^T(EE^T)^aU\|,
    \quad \|e_{m, k}^T(EE^T)^aEV\|,
    \quad \|e_{n, l}^T(E^TE)^aV\|,
    \quad \|e_{n, l}^T(E^TE)^aE^TU\|,
\end{equation*}
uniformly over all choices for $k\in [m]$, $l\in [n]$ and $0 \le a \le 10\log(m + n)$.
We call them \textbf{semi-isotropic bounds} of powers of $E$, due to one side of them involving generic unit vectors (isotropic part) and the other side involving standard basis vectors.

To the best of our knowledge, there has been no well-known semi-isotropic bounds, entry bounds, or isotropic bounds of powers of a random matrix in the literature.
The lemma below, which establishes such semi-isotropic bounds and gives values to $\myD_0(U)$, $\myD_1(U)$, $\myD_0(V)$, $\myD_1(V)$ and $\uppnormE$, is thus another noteworthy contribution of this paper and has potentials for further applications.

\begin{lemma}  \label{lem:E-power-union-bound}
    Let $\upperE$ and $\stdE$ be positive real numbers and $E$ be a $m\times n$ random matrix with independent entries following Model \eqref{eq:noise-model} with parameters $\upperE$ and $\stdE$.
    Let
    \begin{equation*}
        \uppnormE \defeq 1.9\stdE\sqrt{m + n}.
    \end{equation*}
    For each $p > 0$, define
    \begin{equation}  \label{eq:myD(U,p)-defn}
        \myD_0(U, p)
        \defeq \frac{p\|U\|_{2, \infty}}{\sqrt{r}},
        \quad
        \myD_1(U, p)
        \defeq \frac{\upperE p^3\|U\|_{2, \infty}}{\sqrt{r(m + n)}}
            + \frac{p^{3/2}}{\sqrt{m + n}}.
    \end{equation}
    There are universal constants $C$ and $c$ such that, for any $t > 0$, if $\upperE \le \frac{c\sqrt{m + n}}{t^2\log^2 (m + n)}$, then for each fixed $k\in [m]$, with probability $1 - O(\log^{-C}(m + n))$,
    \begin{equation}  \label{eq:E-power-bound-myD0(U)-1}
    \begin{aligned}
        \max_{0\le \alpha \le t\log (m + n)}
        \frac{\|e_{m, k}^T (EE^T)^a U\|}{\uppnormE^{2a}\sqrt{r}}
        \le \myD_0(U, \log\log (m + n))
    \end{aligned}
    \end{equation}
    for each fixed $k\in [n]$, with probability $1 - O(\log^{-C}(m + n))$,
    \begin{equation}  \label{eq:E-power-bound-myD1(U)-1}
    \begin{aligned}
        \max_{0\le \alpha \le t\log (m + n)}
        \frac{\|e_{n, k}^T (E^TE)^aE^T U\|}{\uppnormE^{2a + 1}\sqrt{r}}
        \le \myD_1(U, \log\log (m + n))
    \end{aligned}
    \end{equation}
    If the stronger bound $\upperE\le \frac{c\sqrt{m + n}}{t^2\log^5(m + n)}$ holds, then with probability $1 - O((m + n)^{-2})$,
    \begin{align}
        \label{eq:E-power-bound-myD0(U)-2}
        \max_{0\le \alpha \le t\log (m + n)} \max_{k\in [m]}
        \frac{\|e_{m, k}^T (EE^T)^a U\|}{\uppnormE^{2a}\sqrt{r}}
        & \le \myD_0(U, \log(m + n)),
        \\
        \label{eq:E-power-bound-myD1(U)-2}
        \max_{0\le \alpha \le t\log (m + n)} \max_{k\in [n]}
        \frac{\|e_{n, k}^T (E^TE)^aE^T U\|}{\uppnormE^{2a + 1}\sqrt{r}}
        & \le \myD_1(U, \log(m + n))
    \end{align}
    Analogous bounds hold for $V$, with $E$ and $E^T$ swapped.
\end{lemma}

We only use Eq. \eqref{eq:E-power-bound-myD0(U)-2} and Eq. \eqref{eq:E-power-bound-myD1(U)-2} to prove Theorem \ref{thm:main-random}, but for completeness, we still include Eqs. \eqref{eq:E-power-bound-myD0(U)-1} and \eqref{eq:E-power-bound-myD1(U)-1}, which have a better bound at the cost of being non-uniform in $k$.
They may have potential for other applications.

To prove this lemma, we will use the moment method, with a walk-counting argument inspired by the coding scheme in \cite{vu2005}, to bound these terms.
We put the full proof in Section \ref{sec:E-power-bound-proof}.
}

Let us prove Theorem \ref{thm:main-random} using these lemmas.

{\oldcomment{}
\begin{proof}[Proof of Theorem \ref{thm:main-random}]
    Consider the objects from Setting \ref{set:main}.
    We aim to apply Theorem \ref{thm:main-deterministic}.

    \noindent By Lemma \ref{lem:E-bound-basic}, with probability $1 - O((m + n)^{-1})$, we can replace condition \eqref{eq:main-deterministic-R123} in Theorem \ref{thm:main-deterministic}
    \begin{equation*}
        \frac{\|E\|}{\sigma_S} \vee \frac{2r\|U^TEV\|_\infty}{\Delta_S}
        \vee \frac{\sqrt{2r}\|E\|}{\sqrt{\sigma_S\Delta_S}}
        \le \frac{1}{8}
    \end{equation*}
    with condition \eqref{eq:main-random-S-cond} in Theorem \ref{thm:main-random}
    \begin{equation*}
        \frac{\stdE\sqrt{N}}{\sigma_S}
        \ \vee \
        \frac{
            r\stdE(\sqrt{\log N}
            + \upperE\|U\|_\infty \|V\|_\infty \log N)
        }{\Delta_S}
        \ \vee \
        \frac{\stdE\sqrt{rN}}{\sqrt{\Delta_S\sigma_S}}
        \le \frac{1}{16}.
    \end{equation*}
    Assume \eqref{eq:main-random-S-cond} holds, then \eqref{eq:main-deterministic-R123} also hold and we can now apply Theorem \ref{thm:main-deterministic}.
    Define
    \begin{equation*}
        \myD_1 = \myD_0(U, \log(m + n)) + \myD_1(V, \log(m + n)),
        \quad
        \myD_2 = \myD_0(V, \log(m + n)) + \myD_1(U, \log(m + n)),
    \end{equation*}
    where $\myD_0(U, \cdot), \myD_1(U, \cdot)$ and $\myD_0(V, \cdot), \myD_1(V, \cdot)$ are from Lemma \ref{lem:E-power-union-bound}.
    These terms match exactly with $\myD_1$ and $\myD_2$ from the statement of Theorem \ref{thm:main-random}.
    If they also matched $\myD_1$ and $\myD_2$ in Theorem \ref{thm:main-deterministic}, the proof would be complete.
    However, they do not.
    
    Let $\uppnormE \defeq 2\stdE\sqrt{m + n}$, then $\uppnormE \ge \|E\|$ by Lemma \ref{lem:E-bound-basic}.
    Per the discussion around the condition \eqref{eq:main-determistic-E-power} above, if we can show that $\myD_1$, $\myD_2$ and $\uppnormE$ satisfy this condition, then the argument in Theorem \ref{thm:main-deterministic} still works.
    By Lemma \ref{lem:E-power-union-bound} for $t = 10$, \eqref{eq:main-determistic-E-power} holds with probability $1 - O((m + n)^{-2})$, so the proof is complete.
\end{proof}
}

In the next section, we prove Theorem \ref{thm:main-deterministic}.
The proof is an adaptation of the main argument in \cite{phuctranVu2023} for the SVD.
While this adaptation is easy, it has several important adjustments, sufficient to make Theorem \ref{thm:main-deterministic} independent result rather than a simple corollary.
For instance, the adjustment to adapt the argument for the infinity and $2$-to-infinity norms necessitates the semi-isotropic bounds, a feature not required in the original results for the operator norm.
For this reason, we present the proof in its entirety.



\subsection{The deterministic version: Proof of Theorem \ref{thm:main-deterministic}}  \label{sec:main-proof-deterministic}

In this section, we provide the proof of Theorem \ref{thm:main-deterministic}.

Recall that we are  interested in the differences $\tilde{V}_S\tilde{V}_S^T - V_SV_S^T$ and $\tilde{A}_s - A_s$.
As we shall see in the next part, both can be expressed as almost identical power series of the noise matrix $E$, provided the conditions of the theorem hold.
We will establish a common procedure to bound both, consisting of three steps:
\begin{enumerate}
    \item Expand the matrices above as instances of the same generic power series.

    \item Bound each individual term in the generic power series with terms that shrink geometrically with each power of $E$.

    \item Sum all bounding terms as a geometric series to obtain the final bound.
\end{enumerate}


{\oldcomment{}
\begin{setting}  \label{set:svd-to-eigdecomp}
    Let us introduce some new objects related to the objects in Theorem \ref{thm:main-deterministic}.
    \begin{itemize}
        \item $\sym{B} \defeq \begin{pmatrix} 0 & B \\ B^T & 0\end{pmatrix}$.
        This definition is for all matrices $B$.
    
        \item For each $i\in [r]$, let $\lambda_i = \sigma_i$ and $w_i = \frac{1}{\sqrt{2}} \begin{bmatrix} u_i^T, & v_i^T \end{bmatrix}^T$.
        For each $i\in [[r + 1, 2r]]$, let $\lambda_i = -\sigma_{i - r}$ and $w_i = \frac{1}{\sqrt{2}} \begin{bmatrix} u_{i - r}^T, & v_{i - r}^T \end{bmatrix}^T$.
        Define $\boldsymbol{\Lambda} \defeq \{\lambda_i\}_{i\in [2r]}$ and $W \defeq \begin{bmatrix} w_i \end{bmatrix}_{i = 1}^{2r}$.
    
        \item Let $\{w_i: i\in [[2r + 1, m + n]]\}$ be any orthonormal basis for the column space of $I - WW^T$.
    
        \item Define $\tilde{w}_i$ similarly, with $\rank \tilde{A}$ instead of $r$.

        \item For future use, let $e_{N, k}$ be the $k^{th}$ vector of the standard basis in $\R^N$.
    \end{itemize}
\end{setting}
}

We first introduce the symmetrization trick, which translates from the SVD to an eigendecomposition.
If $A$ has the SVD:
$
A = \sum_{i\in [r]} \sigma_iu_i^Tv_i^T,
$
then $\sym{A}$ has the eigendecomposition:
\begin{equation*}
    \sym{A}
    = W\Lambda W^T = \sum_{i = 1}^{m + n} \lambda_i w_i w_i^T.
\end{equation*}
Note that the singular values of $\sym{A}$ are again $\sigma_1, \ldots, \sigma_r$, but each with multiplicity $2$, thus the matrices
\begin{equation*}
    W_s \defeq \begin{bmatrix}
        w_1, & w_2, & \ldots & w_s, & w_{r + 1}, & \ldots & w_{r + s}
    \end{bmatrix},
    \quad
    (\sym{A})_s \defeq \sum_{i = 1}^s \lambda_i w_iw_i^T + \sum_{i = r + 1}^{r + s} \lambda_i w_iw_i^T
\end{equation*}
are respectively, the singular basis of the most significant $2s$ vectors and the best rank-$2s$ approximation of $\sym{A}$.
However, we still use the subscript $s$ instead of $2s$ to emphasize their relation to the quantities $U_s$, $V_s$ and $A_s$.
For an arbitrary subset $S\subset [r]$, we analogously denote
\begin{equation*}
    W_S \defeq \begin{bmatrix}
        w_i, \ w_{i + r}
    \end{bmatrix}_{i\in S},
    \ \text{ and } \
    (\sym{A})_S \defeq \sum_{i \in S} \lambda_i w_iw_i^T + \sum_{i - r \in S} \lambda_i w_iw_i^T.
\end{equation*}
We define $\tilde{\lambda}_i$, $\tilde{w}_i$ and $\tilde{W}_s$, $\tilde{W}_S$, $\tilde{A}_s$, $\tilde{A}_S$ similarly for $\tilde{A} = A + E$.
To ease the notation,
we denote
\begin{equation*}
    P_i \defeq w_iw_i^T,
    \quad \text{ for } i = 1, 2, \ldots, 2r,
    \quad \text{ and }
    \quad
    Q \defeq I - \sum_{i\le 2r} P_i = I - WW^T.
\end{equation*}
The resolvent of $\sym{A}$, which is a function of a complex variable $z$, can now be written:
\begin{equation*}
    (zI - \sym{A})^{-1}
    = \sum_{i = 1}^{m + n} \frac{P_i}{z - \lambda_i}
    = \sum_{i = 1}^{2r} \frac{P_i}{z - \lambda_i}
    + \frac{Q}{z}.
\end{equation*}
The strategy for proving Theorem \ref{thm:main-deterministic} is a Taylor-like expansion of the difference of resolvents
\begin{equation*}
    (zI - \sym{\tilde{A}})^{-1}
    - (zI - \sym{A})^{-1}
    = \sum_{\gamma = 1}^\infty
    \left[
        (zI - \sym{A})^{-1}\sym{E}
    \right]^\gamma
    (zI - \sym{A})^{-1}.
\end{equation*}
 It is easy to show that this 
 identity  hold whenever the right-hand side converges. 
{
Conveniently, the convergence is also guaranteed by the condition \eqref{eq:main-deterministic-R123} of Theorem \ref{thm:main-deterministic}, as we will see later.
}

Assuming this is true, we can then extract out the differences of the singular vector projections.
The above can be rewritten as
\begin{equation}  \label{eq:resolvent-diff-1}
    \sum_{i = 1}^{m + n}\frac{\tilde{w}_i\tilde{w}_i^T}{z - \tilde{\lambda}_i}
    - \sum_{i = 1}^{m + n}\frac{w_iw_i^T}{z - \lambda_i}
    = \sum_{\gamma = 1}^\infty
    \biggl[
        \Bigl( \sum_{\lambda_i \neq 0}\frac{P_i}{z - \lambda_i} + \frac{Q}{z} \Bigr)
        \sym{E}
    \biggr]^\gamma
    \Bigl( \sum_{\lambda_i \neq 0}\frac{P_i}{z - \lambda_i} + \frac{Q}{z} \Bigr).
\end{equation}
Let $\Gamma_S$ denote an arbitrary contour in $\C$ that encircles $\{\pm\sigma_i, \pm\tilde{\sigma}_i\}_{i \in S}$ and none of the other eigenvalues of $\tilde{W}$ and $W$.
Integrating over $\Gamma_S$ of both sides and dividing by $2\pi i$, we have
\begin{equation}  \label{eq:proj-s-diff-taylor}
\begin{aligned}
    & \begin{bmatrix}
        \tilde{U}_S\tilde{U}_S^T - U_SU_S^T & 0 \\
        0 & \tilde{V}_S\tilde{V}_S^T - V_SV_S^T
    \end{bmatrix}
    = \tilde{W}_S\tilde{W}_S
    - W_SW_S^T
    \\
    & = \sum_{\gamma = 1}^\infty
    \oint_{\Gamma_S} \frac{\td z}{2\pi i}
    \biggl[
        \Bigl( \sum_{\lambda_i \neq 0}\frac{P_i}{z - \lambda_i} + \frac{Q}{z} \Bigr)
        \sym{E}
    \biggr]^\gamma
    \Bigl( \sum_{\lambda_i \neq 0}\frac{P_i}{z - \lambda_i} + \frac{Q}{z} \Bigr).
\end{aligned}
\end{equation}
We quickly note that the following identity can be obtained by multiplying both sides of Eq. \eqref{eq:resolvent-diff-1} with $z$, dividing by $2\pi i$ and integrating over $\Gamma_S$:
\begin{equation}  \label{eq:rank-s-approx-diff-taylor}
\begin{aligned}
    & \sym{(\tilde{A}_S - A_S)}
    = (\sym{\tilde{A}})_S - (\sym{A})_S
    = \sum_{i\in S \vee i - r\in S}
    \left(
        \frac{z\tilde{w}_i\tilde{w}_i^T}{z - \tilde{\lambda}_i}
        - \frac{zw_iw_i^T}{z - \lambda_i}
    \right)
    \\
    & = \sum_{\gamma = 1}^\infty
    \oint_{\Gamma_S} \frac{z\td z}{2\pi i}
    \biggl[
        \Bigl( \sum_{\lambda_i \neq 0}\frac{P_i}{z - \lambda_i} + \frac{Q}{z} \Bigr)
        \sym{E}
    \biggr]^\gamma
    \Bigl( \sum_{\lambda_i \neq 0}\frac{P_i}{z - \lambda_i} + \frac{Q}{z} \Bigr).
\end{aligned}
\end{equation}

It is thus beneficial to find a \emph{common strategy} that can bound both these expressions at the same time.
We consider the following general expression for $\nu\in \N$:
\begin{equation}  \label{eq:generic-diff-taylor}
    \myT_\nu
    \defeq \oint_{\Gamma_S} \frac{z^\nu\td z}{2\pi i}
    \biggl[
        \Bigl( \sum_{\lambda_i \neq 0}\frac{P_i}{z - \lambda_i} + \frac{Q}{z} \Bigr)
        \sym{E}
    \biggr]^\gamma
    \Bigl( \sum_{\lambda_i \neq 0}\frac{P_i}{z - \lambda_i} + \frac{Q}{z} \Bigr).
\end{equation}
Our common strategy will establish a bound on this generic form, which implies Eqs. \eqref{eq:s-sing-vec-entry} and \eqref{eq:s-sing-vec-row} for $\nu = 0$ and Eq. \eqref{eq:s-rank-approx-entry} for $\nu = 1$. To be more precise, we are going to focus on one entry of 
\eqref{eq:generic-diff-taylor}.
However, what we follow, we first describe the critical step of  breaking up the RHS of \eqref{eq:generic-diff-taylor}
into many more treatable terms.
This steps requries a significant amount of foresight and planning. 
The focus on one single entry follows next.


We are going to use  the following (easy to verify) identity:
\begin{equation*}  \label{eq:main-proof-temp1}
    \sum_{\lambda_i \neq 0} \frac{P_i}{z - \lambda_i} + \frac{Q}{z}
    = \sum_{\lambda_i \neq 0} \frac{\lambda_i P_i}{z(z - \lambda_i)} + \frac{I}{z}
\end{equation*}
Plugging this identity into Eq. \eqref{eq:generic-diff-taylor}, we get
\begin{equation}  \label{eq:main-proof-terms1}
    \myT_\nu
    = \sum_{\gamma = 1}^\infty
    \oint_{\Gamma_S} \frac{z^\nu \td z}{2\pi i}
    \biggl[
        \Bigl( \sum_{\lambda_i \neq 0} \frac{\lambda_i P_i}{z(z - \lambda_i)} + \frac{I}{z} \Bigr)
        \sym{E}
    \biggr]^\gamma
    \Bigl( \sum_{\lambda_i \neq 0} \frac{\lambda_i P_i}{z(z - \lambda_i)} + \frac{I}{z} \Bigr).
\end{equation}
Fix $\gamma \in \N$, $\gamma \ge 1$ and consider the $\gamma$-power term in the series.
Expanding the power yields a sum of terms of the form
\begin{equation*}
\begin{aligned}
    \oint_{\Gamma_S} \frac{z^\nu \td z}{2\pi i}
    \left(\frac{I}{z}\sym{E}\right)^{\alpha_0}
    \frac{\lambda_{i_{11}}P_{i_{11}}}{z(z - \lambda_{i_{11}})}
    \sym{E}
    \frac{\lambda_{i_{12}}P_{i_{12}}}{z(z - \lambda_{i_{12}})}
    \ldots
    \sym{E}
    \frac{\lambda_{i_{1\beta_1}}P_{i_{1\beta_1}}}{z(z - \lambda_{i_{1\beta_1}})}
    \sym{E}
    \left(\frac{I}{z}\sym{E}\right)^{\alpha_1}
    \\
    \frac{\lambda_{i_{21}}P_{i_{21}}}{z(z - \lambda_{i_{21}})}
    \sym{E}
    \ldots
    \sym{E}
    \frac{\lambda_{i_{2\beta_2}}P_{i_{2\beta_2}}}{z(z - \lambda_{i_{2\beta_2}})}
    \sym{E}
    \left(\frac{I}{z}\sym{E}\right)^{\alpha_2}
    \ldots
    \sym{E}
    \frac{\lambda_{i_{h\beta_h}}P_{i_{h\beta_h}}}{z(z - \lambda_{i_{h\beta_h}})}
    \left(\sym{E}\frac{I}{z}\right)^{\alpha_h},
\end{aligned}
\end{equation*}
which can be rewritten as
\begin{equation}  \label{eq:main-proof-terms}
    \myC_\nu(\mathbf{I})
    \sym{E}^{\alpha_0}
    \left[
        \prod_{k = 1}^{h - 1}
        \myM\left(\mathbf{i}_k\right)
        \sym{E}^{\alpha_k + 1}
    \right]
    \myM\left(\mathbf{i}_h\right)
    \sym{E}^{\alpha_h},
\end{equation}
where we denote
\begin{equation*}
    \mathbf{I} \defeq [\mathbf{i}_1, \mathbf{i}_2, \ldots, \mathbf{i}_h],
    \quad \quad
    \mathbf{i}_k \defeq [i_{k1}, i_{k2}, \ldots, i_{k\beta_k}],
\end{equation*}
and for a non-empty sequence $\mathbf{i} = [i_1, i_2, \ldots, i_\beta]$ we denote the \emph{monomial matrix}
\begin{equation}  \label{eq:monomial-matrix}
    \myM\left(\mathbf{i}\right)
    \defeq
    P_{i_1}\prod_{j = 2}^{\beta}\sym{E}P_{i_j},
\end{equation}
and the scalar \emph{integral coefficient} for the non-empty sequence $\mathbf{I} = [i_{11}, i_{12}, \ldots, i_{h\beta_h}]$
\begin{equation}  \label{eq:integral-coef}
    \myC_\nu(\mathbf{I})
    \defeq
    \oint_{\Gamma_S} \frac{{\oldcomment{}z^{\nu}} \td z}{2\pi i}
    \frac{1}{z^{\gamma + 1}}
    \prod_{k = 1}^h \prod_{j = 1}^{\beta_k}
    \frac{\lambda_{i_{kj}}}{z - \lambda_{i_{kj}}}.
\end{equation}
Let $\Pi_h(\gamma)$ be the set of all tuples of $\boldsymbol{\alpha} = [\alpha_k]_{k = 0}^h$ and $\boldsymbol{\beta} = [\beta_k]_{k = 1}^h$ such that:
\begin{equation}  \label{eq:alpha-beta-defn}
\begin{aligned}
    \bullet\quad &
    \alpha_0, \alpha_h \ge 0,
    \quad \text{ and }
    \alpha_k \ge 1 \text{ for } 1\le k\le h - 1,
    \\
    \bullet\quad &
    \beta_k \ge 1 \text{ for } 1\le k\le h,
    \\
    \bullet\quad &
    \alpha + \beta = \gamma + 1,
    \quad \text{ where } \alpha \defeq \textstyle\sum_{k = 0}^h \alpha_k,
    \quad \text{ and } \beta \defeq \sum_{k = 1}^h \beta_k.
\end{aligned}
\end{equation}
Note that the conditions above imply $2h - 1 \le \gamma + 1$, so the maximum value for $h$ is $\lfloor \gamma/2 \rfloor + 1$.

Combining Eqs. \eqref{eq:main-proof-terms1}, \eqref{eq:integral-coef}, and \eqref{eq:monomial-matrix}, 
we get the expansion
\begin{equation}  \label{eq:generic-T-defn}
    \myT_\nu
    = \sum_{\gamma = 1}^\infty \myT_\nu^{(\gamma)},
    \quad \text{ where }
    \myT_\nu^{(\gamma)} = \sum_{h = 0}^{\lfloor \gamma/2 \rfloor + 1}
    \myT_\nu^{(\gamma, h)},
    \quad \text{ where }
    \myT_\nu^{(\gamma, h)} = \sum_{(\boldsymbol{\alpha}, \boldsymbol{\beta}) \in \Pi_h(\gamma)}
    \myT_\nu(\boldsymbol{\alpha}, \boldsymbol{\beta}).
\end{equation}
where
\begin{equation}  \label{eq:T(alpha,beta)-defn}
    \myT_\nu(\boldsymbol{\alpha}, \boldsymbol{\beta})
    \defeq
    \sum_{\mathbf{I} \in [2r]^{\beta_1 + \ldots + \beta_h}}
    \myC_\nu(\mathbf{I})
    \sym{E}^{\alpha_0}
    \left[
        \prod_{k = 1}^{h - 1}
        \myM\left(\mathbf{i}_k\right)
        \sym{E}^{\alpha_k + 1}
    \right]
    \myM\left(\mathbf{i}_h\right)
    \sym{E}^{\alpha_h}.
\end{equation}

The sum looks complicated, but the main advantage is that each term on the RHS of \eqref{eq:T(alpha,beta)-defn}
is easier to bound.

Now, we come back to address the main issue of bound the infinity norm. 
Consider Eqs. \eqref{eq:s-sing-vec-entry} and \eqref{eq:s-rank-approx-entry}.
The strategy here is natural. We are going to consider an arbitrary entry of  
$ \tilde{V}_S\tilde{V}_S^T - V_SV_S^T $ (and $\tilde{A}_s - A_s$), and then bounding it with overwhelming probablity, and use the union bound to finish the task.

For Eq. \eqref{eq:s-sing-vec-row}, since the $2$-to-$\infty$ norm of a matrix is simply the norm of the largest row, we can use the aforementioned strategy by replacing the fixed entry with a fixed row in the first step.


Now we turn to the details. Consider an arbitrary $jl$-entry of $\tilde{V}_S\tilde{V}_S^T - V_SV_S^T$.
It corresponds to the $(m + j)(m + l)$-entry of $\tilde{W}_S\tilde{W}_S^T - W_SW_S^T$.
Using Eqs. \eqref{eq:generic-T-defn} and \eqref{eq:T(alpha,beta)-defn}, 
we can write this entry as
\begin{equation}   \label{eq:main-proof-terms-entry}
    e_{m + n, m + j}^T \myT_0 e_{m + n, m + k}
    = \sum_{\gamma = 1}^\infty \sum_{h = 0}^{\lfloor \gamma/2 \rfloor + 1}
    \sum_{(\boldsymbol{\alpha}, \boldsymbol{\beta})\in \Pi_h(\gamma)}
    e_{m + n, m + j}^T\myT_0(\boldsymbol{\alpha}, \boldsymbol{\beta})e_{m + n, m + k}.
\end{equation}

This is the key quantity that we need to bound. 
Similarly, the expansions for a single row of $\tilde{V}_S\tilde{V}_S^T - V_SV_S^T$ or a single entry of $\tilde{A}_s - A_s$ also have the form $M^T \myT_\nu M'$, where $M$ and $M'$ is either a standard basis vector $e_{m + n, j}$ or $I_{m + n}$.



It is thus beneficial to establish a general bound for 
arbitrary choices of $M$ and $M'$ in the operator norm, which covers both the Euclidean norm of a vector and the absolute value of a number.
In fact, our proof works for any sub-multiplicative norm, which we use the generic $\|\cdot\|$ to denote from this point to the end of this section.

\bigskip

\noindent{\bf The general task.}
We aim to upper bound $\|M^T \myT_\nu M'\|$ for a sub-multiplicative norm $\|\cdot\|$.
To prove Theorem \ref{thm:main-deterministic}, we only apply this general bound for the operator norm.
Each specific bound in the theorem will be a consequence of this general bound with appropriate choices of $M$ and $M'$.

\bigskip

{\oldcomment{}
Let us introduce this general bound and finish the proof of Theorem \ref{thm:main-deterministic} here.
The proof of the bound will be in the next section, where we only need to care about the context in Setting \ref{set:svd-to-eigdecomp}.
}

\begin{lemma}  \label{lem:T-bound}
     Consider the objects in Setting \ref{set:svd-to-eigdecomp} and define the following terms
     \begin{equation}  \label{eq:T-bound-R123-defn}
     \begin{aligned}
         R_1 \defeq \frac{\|E\|}{\lambda_S} \vee \frac{2r\|W^T\sym{E}W\|_\infty}{\Delta_S},
         \quad
         R_2 \defeq \frac{\sqrt{2r}\|E\|}{\sqrt{\lambda_S\Delta_S}},
         \quad
         R_3 \defeq \frac{2r}{\lambda_S\Delta_S} \max_{|i - j|\notin\{0, r\}} |w_i\sym{E}^2w_j|.
     \end{aligned}
     \end{equation}
     Additionally, define $\myL_0 = 2$ and $\myL_1 = \lambda_S$ and
     \begin{equation}  \label{eq:T-bound-myD}
         \myD \defeq \max_{\alpha \in [\lceil 10\log(m + n) \rceil]}
         \frac{1}{2r} \sum_{i = 1}^{2r} \frac{\|w_i^T\sym{E}^\alpha M\|}{\|E\|^\alpha},
     \end{equation}
     and analogously for $\myD'$ and $M'$.
     Suppose $R_1 \vee R_2 \le 1/4$.
     Then the $\myT_\nu$ defined in Eq. \eqref{eq:generic-T-defn} converges in the metric $\|\cdot\|$ and satisfies, for a universal constant $C$,
     \begin{equation*}
         \left\| {\oldcomment{}M^T\myT_\nu M'} \right\| \le
         Cr\myL_\nu(R_1 + R_3)\Bigl[
            \myD\myD'  + \|M\|\|M'\|(m + n)^{-2.5}
         \Bigr].
     \end{equation*}
\end{lemma}

Let us remark on the meanings of the new terms, which are simply translation of terms from Theorem \ref{thm:main-deterministic} into the language of Setting \ref{set:svd-to-eigdecomp}.
\begin{itemize}
    \item The term $\|M\|\|M'\|(m + n)^{-2.5}$ is small, and will be absorbed into the term $\myD\myD'$ for our applications.

    \item Recall the definitions of $\lambda_i$ and $W$ from Setting \ref{set:svd-to-eigdecomp}.
    Plugging them into Eq. \eqref{eq:T-bound-R123-defn} recovers
    \begin{equation*}
        R_1 \vee R_2
        = \frac{\|E\|}{\sigma_S} \vee \frac{2r\|U^TEV\|_\infty}{\Delta_S}
        \vee \frac{\sqrt{2r}\|E\|}{\sqrt{\sigma_S\Delta_S}},
    \end{equation*}
    and
    \begin{equation*}
        R_1 + R_3 \le 2\left(
            \frac{\|E\|}{\sigma_S} + \frac{r\|U^TEV\|_\infty}{\Delta_S}
            + \frac{ry}{\Delta_S\sigma_S}
        \right).
    \end{equation*}

    \item Similarly, recall the definitions of $\myD_1$ and $\myD_2$ in Eq. \eqref{eq:main-determistic-myD12}.
    As a function of $M$, $\myD$ satisfies
    \begin{equation*}
        \myD(e_{m + n, k}) = \myD_1 \text{ for } k\le m,
        \quad \myD(e_{m + n, k}) = \myD_2 \text{ for } m + 1\le k\le m + n,
        \quad \text{ and } \ \myD(I) \le 1,
    \end{equation*}

    \item The term $\myL_\nu$ plays the role of the factor $\sigma_s$ on the right-hand side of Eq. \eqref{eq:s-rank-approx-entry} when we prove it by applying this lemma for $\nu = 1$.
    We set $\nu = 0$ when proving Eqs. \eqref{eq:s-sing-vec-entry} and \eqref{eq:s-sing-vec-row}, in which cases $\myL_\nu$ has no effects.
\end{itemize}

\begin{proof}[Proof of Theorem \ref{thm:main-deterministic}]
    Consider the objects defined in Theorem \ref{thm:main-deterministic} and the additional objects in Setting \ref{set:svd-to-eigdecomp}.
    Note that $\lambda_i = \sigma_i$ for $i\in [r]$ and $-\sigma_{i - r}$ for $i\in [[r + 1, 2r]]$, and $\lambda_S = \sigma_S$ for every subset $S$ in this context.
    We choose $\|\cdot\|$ to be the operator norm.
    By the remark above, the condition \eqref{eq:main-deterministic-R123} in Theorem \ref{thm:main-deterministic} is equivalent to $R_1\vee R_2 \le 1/4$, so we can apply Lemma \ref{lem:T-bound}.
    
    Let us prove Eq. \eqref{eq:s-sing-vec-entry}.
    Consider arbitrary $j, k\in [n]$.
    Combining Eq. \eqref{eq:proj-s-diff-taylor} and the equations leading up to Eq. \eqref{eq:generic-T-defn}, we have
    \begin{equation*}
        \left(
            \tilde{V}_S\tilde{V}_S^T - V_SV_S^T
        \right)_{jk}
        = \left(
            \tilde{W}_S\tilde{W}_S^T - W_SW_S^T
        \right)_{(j + m)(k + m)}
        {\oldcomment{}
        = e_{m + n, j + m}^T\myT_0 e_{m + n, m + k}.
        }
    \end{equation*}
    We apply Lemma \ref{lem:T-bound} for the choices $M = e_{m + n, j + m}$, $M' = e_{m + n, k + m}$ and $\nu = 0$.

    
    Note that both $M$ and $M'$ are basis vectors in $\{e_{m + n, i}: i\in [[m + 1, m + n]]\}$, so by the remark, we get $\myD = \myD' = \myD_2$ (defined in Eq. \eqref{eq:main-determistic-myD12}).
    Lastly, $\myL_0 = 2$.
    We have, by Lemma \ref{lem:T-bound},
    %
    %
    \begin{equation*}
        \left|
            \left(
                \tilde{V}_S\tilde{V}_S^T - V_SV_S^T
            \right)_{jk}
        \right|
        \le Cr\myL_0(R_1 + R_3)
        \left(
            \myD_2^2 + \frac{\|M\|\|M'\|}{(m + n)^{2.5}} 
        \right)
        \le 3Cr\myD_2^2(R_1 + R_3),
    \end{equation*}
    where the last inequality is due to the facts $\|M\| = \|M'\| = 1$ and $\myD_1, \myD_2 \ge (m + n)^{-1/2}$.
    This holds over all $j, k\in [n]$, so it extends to the infinity norm, proving Eq. \eqref{eq:s-sing-vec-entry}.

    \bigskip

    Let us prove Eq. \eqref{eq:s-sing-vec-row}.
    Consider an arbitrary $j\in [n]$.
    The same argument in the previous step gives
    \begin{equation*}
        \left(
            \tilde{V}_S\tilde{V}_S^T - V_SV_S^T
        \right)_{j, \cdot}
        = \left(
            \tilde{W}_S\tilde{W}_S^T - W_SW_S^T
        \right)_{(j + m), \cdot}
        {\oldcomment{}
        = e_{m + n, m + j}^T\myT_0 = M^T\myT_\nu M'},
    \end{equation*}
    for the choices $\nu = 0$, $M = e_{m + n, j + m}$ and $M' = I_{m + n}$.
    We repeat the previous argument, but this time $\myD = \myD_2$ while $\myD' = 1$, so the final bound has only one instance of $\myD_2$, namely
    \begin{equation*}
        \left\|
            \left(
                \tilde{V}_S\tilde{V}_S^T - V_SV_S^T
            \right)_{j, \cdot}
        \right\|
        \le 3Cr\myD_2(R_1 + R_3),
    \end{equation*}
    which holds uniformly over $j\in [n]$, proving Eq. \eqref{eq:s-sing-vec-row}.

    \bigskip
    
    Let us prove Eq. \eqref{eq:s-rank-approx-entry}.
    Consider arbitrary $j\in [m]$ and $k\in [n]$.
    Combining Eq. \eqref{eq:rank-s-approx-diff-taylor} and the equations leading up to Eq. \eqref{eq:generic-T-defn}, we have
    \begin{equation*}
        \bigl(
            \tilde{A}_s - A_s
        \bigr)_{jk}
        = \left(
            \tilde{W}_S\tilde{W}_S^T - W_SW_S^T
        \right)_{j(k + m)}
        {\oldcomment{}
        = e_{m + n, j}^T\myT_1e_{m + n, m + k} = M^T\myT_\nu M'
        }
    \end{equation*}
    for the choices $M = e_{m + n, j}$, $M' = e_{m + n, k + m}$ and $\nu = 1$.
    This time $\myL_\nu = \lambda_s = \sigma_s$, and $\myD = \myD_1$ while $\myD' = \myD_2$.
    Adapting them into the previous arguments give
    \begin{equation*}
        \left|
            \left(
                \tilde{A}_s\tilde{A}_s^T - A_sA_s^T
            \right)_{jk}
        \right|
        \le 3Cr\myD_1\myD_2\sigma_s(R_1 + R_3)
        ,
    \end{equation*}
    for a constant $C_0$.
    This bound holds uniformly over $j\in [m]$ and $k\in [n]$, so it holds in the infinity norm.
    The proof is complete.
\end{proof}

\subsection{Proof of the bound on generic series}
\label{sec:proof-T-bound}

In this section, we prove Lemma \ref{lem:T-bound}, which is the backbone of Theorem \ref{thm:main-deterministic}'s proof.
We summarize the objects involved in the bound below.
Note that at this point we do not need to care about $A$.

\begin{setting}  \label{set:generic-series}
Let the following objects and properties be given:
\begin{itemize}
    \item $\boldsymbol{\Lambda} = \{\lambda_i\}_{i\in [2r]}$: a set of real numbers such that $\delta_i \defeq \lambda_i - \lambda_{i + 1} > 0$ for each $i \le r - 1$, $\delta_r = \lambda_r > 0$, and $\lambda_i = -\lambda_{i - r}$ for $i \ge r + 1$.

    \item $W = \{w_i\}_{i\in [2r]}$ for $i\in [2r]$: a set of orthonormal vectors in $\R^{m + n}$.
    We slightly abuse the notation here and use $W$ as an orthogonal matrix when necessary.

    \item $S$: an arbitrary subset of $\{\lambda_i\}_{i \in [r]}$.


    \item $\gamma$ and $h$: positive integers such that $2h - 1 \le \gamma + 1$.

    \item $\Pi_h(\gamma)$: the set of pairs of index sequences $(\boldsymbol{\alpha}, \boldsymbol{\beta})$ satisfying Eq. \eqref{eq:alpha-beta-defn}.
    

    \item $\myC_\nu = \myC_{\nu, \boldsymbol{\Lambda}, S}: \bigcup_{\beta = 0}^{\gamma + 1} [2r]^\beta \to \R$: the mapping from an index sequence $\mathbf{I}$ to its integral coefficient defined in Eq. \eqref{eq:integral-coef}.
    We hide the dependencies on $\boldsymbol{\Lambda}$ and $S$ because they are fixed throughout the proof.

    \item $E$: an arbitrary matrix in $\R^{m\times n}$, and let $\sym{E}$ be its symmetrization.

    \item $\myM = \myM_{W, E}: \bigcup_{\beta = 0}^{\gamma + 1} [2r]^\beta \to \R^{(m + n)\times (m + n)}$: the mapping from an index sequence $\mathbf{i}$ to its monomial matrix defined in Eq. \eqref{eq:monomial-matrix}.

    \item $M$ and $M'$: arbitrary matrices with $m + n$ rows.

    \item $\myT_\nu$: the target sum defined in Eq. \eqref{eq:generic-T-defn}.
    The sub-terms $\myT_\nu^{(\gamma)}$, $\myT_\nu^{(\gamma, h)}$ are defined in the same equation, while $\myT_\nu(\boldsymbol{\alpha}, \boldsymbol{\beta})$ is defined in Eq. \eqref{eq:T(alpha,beta)-defn}.



\end{itemize}
\end{setting}

The quantities we need to bound is $\|M^T\myT_\nu M'\|$ where $\|\cdot\|$ is a sub-multiplicative matrix norm, $M$ and $M'$ are arbitrary matrices (with dimensions compatible to $\myT_\nu$'s), and $\myT_\nu$ is defined in Eq. \eqref{eq:generic-T-defn}.
We reproduce the bound below, for the reader's convenience.
\begin{equation}  \label{eq:generic-T-defn-1}
    \myT_\nu
    = \sum_{\gamma = 1}^\infty \myT_\nu^{(\gamma)},
    \quad \text{ where }
    \myT_\nu^{(\gamma)} = \sum_{h = 0}^{\lfloor \gamma/2 \rfloor + 1}
    \myT_\nu^{(\gamma, h)},
    \quad \text{ where }
    \myT_\nu^{(\gamma, h)} = \sum_{(\boldsymbol{\alpha}, \boldsymbol{\beta}) \in \Pi_h(\gamma)}
    \myT_\nu(\boldsymbol{\alpha}, \boldsymbol{\beta}).
\end{equation}
For the definitions of $\myT_\nu(\boldsymbol{\alpha}, \boldsymbol{\beta})$ and $\Pi_h(\gamma)$, we refer to Section \ref{sec:main-proof-deterministic}.

\bigskip

\noindent\textbf{Bounding strategy.}
Our strategy to bound $\|M^T \myT_\nu M'\|$ is straightforward, with three steps:
\begin{enumerate}
    \item Bounding $\|M^T\myT_\nu(\boldsymbol{\alpha}, \boldsymbol{\beta})M'\|$ for each pair $(\boldsymbol{\alpha}, \boldsymbol{\beta}) \in \Pi_h(\gamma)$ with an exponentially small bound.

    \item Summing up the bounds over $\Pi_h(\gamma)$ and over all $0 \le h \le \gamma/2 + 1$ to get a bound for $\|M^T\myT_\nu^{(\gamma)}M'\|$.
    The bounds in the previous steps should be good enough to still make this bound exponentially small.

    \item Summing up the geometric series of bounds
    over all $\gamma \ge 1$ to get a bound for $\|M^T \myT_\nu M'\|$.
\end{enumerate}

We introduce the following two lemmas corresponding to the first two steps (the third will result in Lemma \ref{lem:T-bound}).

\begin{lemma}  \label{lem:T(alpha,beta)-bound-1}
    Consider objects in Setting \ref{set:generic-series} and Lemma \ref{lem:T-bound}.
    For each $\gamma \ge 1$, $0\le h\le \gamma/2 + 1$, and $\boldsymbol{\alpha}, \boldsymbol{\beta} \in \Pi_h(\gamma)$,
    we have
    \begin{equation}  \label{eq:T(alpha,beta)-bound-1}
    \begin{aligned}
        \|{\oldcomment{}M^T\myT_\nu(\boldsymbol{\alpha}, \boldsymbol{\beta})M'}\|
        \le
        & \ \myL_\nu
        \binom{\gamma + \beta - 2}{\beta - 1}
        \|W^T\sym{E}W\|_\infty^{\beta - h}
        \|E\|^{\alpha - \alpha_0 - \alpha_h + h - 1}
        \frac{
            \left(
                2|S| + 2\rho|S^c|
            \right)^{\beta - 2}
        }{\lambda_S^{\gamma + 1 - \beta}\Delta_S^{\beta - 1}}
        \\
        & \quad \cdot \sum_{i = 1}^{2r} \rho^{\one\{i_k\in S^c\}}
        \left\| w_i^T \sym{E}^{\alpha_0} M\right\| 
        \cdot \sum_{i = 1}^{2r} \rho^{\one\{i_k\in S^c\}}
        \left\| w_i^T \sym{E}^{\alpha_h} M'\right\|,
    \end{aligned}
    \end{equation}
    where $\rho \defeq (\lambda_S + \Delta_S) / (2\lambda_S)$.
    Consequently,
    \begin{equation}  \label{eq:T(alpha,beta)-bound-2}
    \begin{aligned}
        \|{\oldcomment{}M^T\myT_\nu(\boldsymbol{\alpha}, \boldsymbol{\beta})M'}\|
        \le
        & \ \myL_\nu
        \binom{\gamma + \beta - 2}{\beta - 1}
        \myD \myD'
        \frac{
            (2r)^\beta
            \|W^T\sym{E}W\|_\infty^{\beta - h}
            \|E\|^{\alpha + h - 1}
        }{\lambda_S^{\gamma + 1 - \beta}\Delta_S^{\beta - 1}}.
    \end{aligned}
    \end{equation}
\end{lemma}

Since the bound in Eq. \eqref{eq:T(alpha,beta)-bound-1} is rather cumbersome, we will mostly use Eq. \eqref{eq:T(alpha,beta)-bound-2}.
Eq. \eqref{eq:T(alpha,beta)-bound-1} is still useful for future potential developments in this direction.

\begin{lemma}  \label{lem:T-gamma-bound}
    Consider the objects in Setting \ref{set:generic-series}, the settings of Lemma \ref{lem:T-bound}, and $\myT_\nu^{(\gamma)}$ defined in Eq. \eqref{eq:generic-T-defn-1}.
    For each $1 \le \gamma \le 10\log (m + n)$, we have
    \begin{equation*}
    \begin{aligned}
        \left\| {\oldcomment{}M^T\myT_\nu^{(\gamma)}M'}
        \right\|
        \le
        r\myL_\nu 
            \myD\myD'
        \left[
            9 R_1 (6R)^{\gamma - 1}
            + \one\{\gamma \text{ even}\} \Bigl(
                4\left(
                    \tfrac{3\sqrt{3}}{\sqrt{2}}R_1
                \right)^\gamma
                + 27R_3\left(
                    \tfrac{3\sqrt{3}}{2}R_2
                \right)^{\gamma - 2}
            \Bigr)
        \right].
    \end{aligned}
    \end{equation*}
    For each $\gamma > 10\log (m + n)$, we have
    \begin{equation*}
        \left\| {\oldcomment{}M^T\myT_\nu^{(\gamma)}M'}
        \right\|
        \le
        r\myL_\nu \|M\|\|M'\|
        \left[
            9 R_1 (6R)^{\gamma - 1}
            + \one\{\gamma \text{ even}\} \Bigl(
                4\left(
                    \tfrac{3\sqrt{3}}{\sqrt{2}}R_1
                \right)^\gamma
                + 27R_3\left(
                    \tfrac{3\sqrt{3}}{2}R_2
                \right)^{\gamma - 2}
            \Bigr)
        \right].
    \end{equation*}
\end{lemma}

The structure is simple:
Lemma \ref{lem:T-bound} $\xleftarrow{\text{\ implied by \ }}$
Lemma \ref{lem:T-gamma-bound}  $\xleftarrow{\text{\ implied by \ }}$
Lemma \ref{lem:T(alpha,beta)-bound-1}.
Assuming the two lemmas above, we can prove Lemma \ref{lem:T-bound}, finishing the third step in the strategy.

\begin{proof}[Proof of Lemma \ref{lem:T-bound}]
    For convenience, let $k = \lfloor 10\log (m + n)\rfloor$.
    Applying Lemma \ref{lem:T-gamma-bound}, we have
    \begin{equation*}
    \begin{aligned}
        \sum_{\gamma = 1}^k
        \left\| {\oldcomment{}M^T\myT_\nu^{(\gamma)}M'}
        \right\|
        & \le r\myL_\nu\myD\myD'
        \left[
            9 \sum_{\gamma = 1}^\infty
            R_1 (6R)^{\gamma - 1}
            + \sum_{\gamma = 1}^\infty
            \Bigl(
                4\left(
                    \tfrac{3\sqrt{3}}{\sqrt{2}}R_1
                \right)^{2\gamma}
                + 27R_3\left(
                    \tfrac{3\sqrt{3}}{2}R_2
                \right)^{2\gamma - 2}
            \Bigr)
        \right]
        \\
        & \le r\myL_\nu\myD\myD' \left[ \frac{9R_1}{1 - 6R}
        + \frac{108R_1^2}{2 - 27R_1^2}
        + \frac{108R_3}{4 - 27R_2^2} \right]
        \le 4r\myL_\nu\myD\myD'(18 R_1 + 27 R_3),
    \end{aligned}
    \end{equation*}
    and
        \begin{equation*}
    \begin{aligned}
        & \sum_{\gamma = {k + 1}}^\infty
        \left\| {\oldcomment{}M^T\myT_\nu^{(\gamma)}M'}
        \right\|
        \\
        & \le r\myL_\nu\|M\|\|M'\|
        \left[
            9 \sum_{\gamma = k}^\infty
            R_1 (6R)^{\gamma - 1}
            + \sum_{\gamma = \lceil k/2 \rceil}^\infty
            \left(
                4\left(
                    \tfrac{3\sqrt{3}}{\sqrt{2}}R_1
                \right)^{2\gamma}
                + 27R_3\left(
                    \tfrac{3\sqrt{3}}{2}R_2
                \right)^{2\gamma - 2}
            \right)
        \right]
        \\
        & \le r\myL_\nu\|M\|\|M'\| \left[ \frac{9R_1(6R)^{k - 1}}{1 - 6R}
        + \frac{4(27R_1^2/2)^{\lceil k/2 \rceil}}{1 - 27R_1^2/2}
        + \frac{27R_3(27R_2^2/4)^{\lceil k/2 \rceil - 1}}{1 - 27R_2^2/4} \right]
        \\
        & \le
        \frac{r\myL_\nu\|M\|\|M'\|}{1 - 6R}
        \left[
            9R_1(6R)^k + 4(4R_1)^k + 27R_3(3R)^{k - 2}
        \right]
        \le \frac{9r\myL_\nu\|M\|\|M'\|(R_1 + 3R_3)}{(m + n)^{2.5}}.
    \end{aligned}
    \end{equation*}
    The convergence is guaranteed by the geometrically vanishing bounds on the $\|\cdot\|$-norms of the terms.
    Summing up the two parts, we obtain, by the triangle inequality
    \begin{equation*}
        \left\| M^T\myT_\nu M'\right\|
        \le 36r\myL_\nu (2R_1 + 3R_3)
        \left(
            \myD\myD' + \frac{\|M\|\|M'\|}{4(m + n)^{2.5}}
        \right).
    \end{equation*}
    The proof is complete.
\end{proof}

We can now proceed with the first two steps, namely proving Lemmas \ref{lem:T(alpha,beta)-bound-1} and \ref{lem:T-gamma-bound}.

\subsection{Bounding each term in the generic series}
\label{sec:main-proof-deterministic-bound-series}

Let us prove Lemma \ref{lem:T(alpha,beta)-bound-1}.
We aim to bound $\|M^T\myT_\nu(\boldsymbol{\alpha}, \boldsymbol{\beta})M'\|$.
Eq. \eqref{eq:T(alpha,beta)-defn} gives
\begin{equation}  \label{eq:T(alpha,beta)-defn-1}
    M^T \myT_\nu(\boldsymbol{\alpha}, \boldsymbol{\beta}) M'
    \defeq
    \sum_{\mathbf{I} \in [2r]^{\beta_1 + \ldots + \beta_h}}
    \myC_\nu(\mathbf{I})
    \Bigl(
        M^T D(\boldsymbol{\alpha}, \boldsymbol{\beta}, \mathbf{I}) M'
    \Bigr)
    ,
\end{equation}
where
\begin{equation}  \label{eq:D(alpha,beta,I)-defn}
    D(\boldsymbol{\alpha}, \boldsymbol{\beta}, \mathbf{I})
    \defeq \sym{E}^{\alpha_0}
    \left[
        \prod_{k = 1}^{h - 1}
        \myM\left(\mathbf{i}_k\right)
        \sym{E}^{\alpha_k + 1}
    \right]
    \myM\left(\mathbf{i}_h\right)
    \sym{E}^{\alpha_h}.
\end{equation}
The real number $\myC_\nu(\mathbf{I})$ is the integral coefficient defined in Eq. \eqref{eq:integral-coef} and each $\myM(\mathbf{i}_k)$ is a monomial matrix defined in Eq. \eqref{eq:monomial-matrix}.
Lemma \ref{lem:T(alpha,beta)-bound-1} follows from the two bounds below.
We emphasize that we are only interested in the two cases: (1) $\nu = 0$, and (2) $\nu = 1$ and $S = [s]$ for some $s\in [r]$.

\begin{lemma}  \label{lem:integral-coef-bound}
    Consider the objects defined in Setting \ref{set:generic-series} and an arbitrary tuple $\mathbf{I} \defeq \{i_k\}_{k\in [\beta]} \in [2r]^\beta$
    and denote the following:
    \begin{equation*}
    \begin{aligned}
        \beta_S(\mathbf{I}) & \defeq \left|\{1\le k\le \beta: \ i_k \in S\}\right|,
        \\
        \lambda_S(\mathbf{I}) & \defeq \min\{ |\lambda_{i_k}| : k\in S \},
        \\
        \Delta_S(\mathbf{I}) & \defeq \min\{ |\lambda_{i_k} - \lambda_{i_l}| : i_k\in S, i_l\notin S \}.
    \end{aligned}
    \end{equation*}
    Let $\myL_0(\mathbf{I}) = 2$ and $\myL_1(\mathbf{I}) = \lambda_S(\mathbf{I})$.
    We have, for $\nu\in \{0, 1\}$,
    \begin{equation}  \label{eq:integral-coef-bound-1}
        |\myC_\nu(\mathbf{I})|
        \le \myL_\nu(\mathbf{I}) \left(
            1 + \frac{\Delta_S(\mathbf{I})}{\lambda_S(\mathbf{I})}
        \right)^{\beta_{S^c}(\mathbf{I})}
        \binom{
            \gamma + \beta_S(\mathbf{I}) - 2
        }{
            \beta_S(\mathbf{I}) - 1
        }
        \frac{1}{
            \lambda_S(\mathbf{I})^{\gamma + 1 - \beta}
            \Delta_S(\mathbf{I})^{\beta - 1}
        }.
    \end{equation}
    Consequently, in either of the two cases: (1) $\nu = 0$, and (2) $\nu = 1$ and $S = [s]$ for some $s\in [r]$, we obtain the following version where $\myL_\nu$ is free from $\mathbf{I}$:
    \begin{equation}  \label{eq:integral-coef-bound-2}
        |\myC_\nu(\mathbf{I})|
        \le \myL_\nu
        \left(
            1 + \frac{\Delta_S}{\lambda_S}
        \right)^{\beta_{S^c}(\mathbf{I})}
        \binom{\gamma + \beta_S(\mathbf{I}) - 2}{\beta_S(\mathbf{I}) - 1}
        \frac{1}{\lambda_S^{\gamma + 1 - \beta}\Delta_S^{\beta - 1}},
    \end{equation}
    where $\myL_0 = 2$ and $\myL_1 = \lambda_s$.
\end{lemma}

Note that in order for Eq. \eqref{eq:integral-coef-bound-2} to hold for $\nu = 1$, $S$ needs to contain exactly the first $s$ indices for some $s$.
In the steps that follow, we will mainly use Eq. \eqref{eq:integral-coef-bound-2}, with one exception where the more precise Eq. \eqref{eq:integral-coef-bound-1} is needed.
It thus makes sense to keep both.

\begin{lemma}  \label{lem:monomial-matrix-bound}
    Consider the objects defined in Setting \ref{set:generic-series} and an arbitrary tuple $\mathbf{I} \defeq (\mathbf{i}_k)_{k\in [h]}$, where $\mathbf{i}_k = (i_{kl})_{l\in [\beta_k]} \in [2r]^{\beta_k}$.
    Let $D(\boldsymbol{\alpha}, \boldsymbol{\beta}, \mathbf{I})$ be defined in Eq. \eqref{eq:D(alpha,beta,I)-defn}.
    We have
    \begin{equation}  \label{eq:monomial-matrix-bound}
    \begin{aligned}
        & \left\|
            M^T D(\boldsymbol{\alpha}, \boldsymbol{\beta}, \mathbf{I}) M'
        \right\|
        \le \|W^T\sym{E}W\|_\infty^{\beta - h} \|E\|^{\alpha - \alpha_0 - \alpha_h + h - 1}
        \left\| M^T \sym{E}^{\alpha_0}w_{i_1}\right\|
        \left\| w_{i_{h\beta_h}}^T \sym{E}^{\alpha_h} M'\right\|.
    \end{aligned}
    \end{equation}
\end{lemma}

The proof of Lemma \ref{lem:integral-coef-bound} is a triple inductive argument, which is not complicated but technically lengthy, and thus will be put in Appendix \ref{sec:technical-lemmas}.
The proof of Lemma \ref{lem:monomial-matrix-bound} is a simple rewriting of the product on the left-hand side of Eq. \eqref{eq:monomial-matrix-bound}, coupled with the sub-multiplicativity of the norm $\|\cdot\|$.
We will first finish the proof of Lemma \ref{lem:T(alpha,beta)-bound-1} assuming the two lemmas above, then prove Lemma \ref{lem:monomial-matrix-bound}.

\begin{proof}[Proof of Lemma \ref{lem:T(alpha,beta)-bound-1}]
    Temporarily let $X \defeq W^T\sym{E}W$.
    Applying Lemma \ref{lem:monomial-matrix-bound} to Eq. \eqref{eq:T(alpha,beta)-defn-1} gives
    \begin{equation*}
        \| {\oldcomment{}M^T\myT_\nu(\boldsymbol{\alpha}, \boldsymbol{\beta})M'} \|
        \le \|X\|_\infty^{\beta - h} \|E\|^{\alpha - \alpha_0 - \alpha_h + h - 1}
        \sum_{\mathbf{I}\in [2r]^\beta}
        | \myC_\nu(\mathbf{I}) |
        \left\| M^T \sym{E}^{\alpha_0}w_{i_1}\right\|
        \left\| w_{i_{h\beta_h}}^T \sym{E}^{\alpha_h} M'\right\|.
    \end{equation*}
    Temporarily let $T$ be the sum on the right-hand side.
    Applying Eq. \eqref{eq:integral-coef-bound-2} from Lemma \ref{lem:integral-coef-bound},
    we have
    \begin{equation*}
    \begin{aligned}
        | \myC_\nu(\mathbf{I}) |
        & \le \myL_\nu
        \left(1 + \frac{\Delta_S}{\lambda_S}\right)^{\beta_{S^c}(\mathbf{I})}
            \binom{\gamma + \beta_S(\mathbf{I}) - 2}{\beta_S(\mathbf{I}) - 1}
        \frac{1}{\lambda_S^{\gamma + 1 - \beta}\Delta_S^{\beta - 1}}
        \\
        & \le \myL_\nu
        \binom{\gamma + \beta - 2}{\beta - 1}
        \left(\frac{\lambda_S + \Delta_S}{2\lambda_S}\right)^{\beta_{S^c}(\mathbf{I})}
        \frac{1}{\lambda_S^{\gamma + 1 - \beta}\Delta_S^{\beta - 1}}.
    \end{aligned}
    \end{equation*}
    Let us first prove Eq. \eqref{eq:T(alpha,beta)-bound-1}.
    For convenience, let $\rho \defeq (\lambda_S + \Delta_S) / (2\lambda_S)$.
    Note that $\rho \le 1$ since $\Delta_S \le \lambda_S$.
    Plugging this bound into $T$, we get
    \begin{equation*}
    \begin{aligned}
        T
        & \le
        \myL_\nu
        \binom{\gamma + \beta - 2}{\beta - 1}
        \frac{1}{\lambda_S^{\gamma + 1 - \beta}\Delta_S^{\beta - 1}}
        \sum_{\mathbf{I}\in [2r]^\beta}
        \left\| M^T \sym{E}^{\alpha_0}w_{i_1}\right\|
        \left\| w_{i_{h\beta_h}}^T \sym{E}^{\alpha_h} M'\right\|
        \prod_{k = 1}^\beta \rho^{\one\{i_k\in S^c\}}
        \\
        & = \myL_\nu
        \binom{\gamma + \beta - 2}{\beta - 1}
        \frac{
            \left(
                2|S| + 2\rho|S^c|
            \right)^{\beta - 2}
        }{\lambda_S^{\gamma + 1 - \beta}\Delta_S^{\beta - 1}}
        \sum_{i = 1}^{2r} \rho^{\one\{i_k\in S^c\}}
        \left\| w_i^T \sym{E}^{\alpha_0} M\right\| 
        \sum_{i = 1}^{2r} \rho^{\one\{i_k\in S^c\}}
        \left\| w_i^T \sym{E}^{\alpha_h} M'\right\|.
    \end{aligned}
    \end{equation*}
    The proof of Eq. \eqref{eq:T(alpha,beta)-bound-1} is complete.
    Eq. \eqref{eq:T(alpha,beta)-bound-2} follows from the fact that $\rho \le 1$, the definitions of $\myD$ and $\myD'$ from Eq. \eqref{eq:T-bound-myD} in Lemma \ref{lem:T-bound}.
    The proof of Lemma \ref{lem:T(alpha,beta)-bound-1} is complete.
\end{proof}

\begin{proof}[Proof of Lemma \ref{lem:monomial-matrix-bound}]
    Consider a monomial matrix $\myM(\mathbf{i}_k)$ for $k\in [h]$.
    It is simply a rank one matrix, and
    we can simplify it by extracting many scalars as below:
    \begin{equation*}
        \myM\left(\mathbf{i}_k\right)
        = w_{i_{k1}}w_{i_{k1}}^T\prod_{j = 2}^{\beta_k}\sym{E}w_{i_{kj}}w_{i_{kj}}^T
        = w_{i_{k1}}
        \Bigl(
            \prod_{j = 2}^{\beta_k}w_{i_{k(j - 1)}}^T\sym{E}w_{i_{kj}}
        \Bigr)
        w_{i_{k\beta_k}}^T
        = \Bigl(
            \prod_{j = 2}^{\beta_k} X_{i_{k(j - 1)}i_{kj}}
        \Bigr)w_{i_{ki}}w_{i_{k\beta_k}}^T,
    \end{equation*}
    where we temporarily denote $X \defeq W^T\sym{E}W$.
    From Eq. \eqref{eq:T(alpha,beta)-defn}, we thus have
    \begin{equation}  \label{eq:main-proof-D(alpha,beta,I)-defn}
        M^T D(\boldsymbol{\alpha}, \boldsymbol{\beta}, \mathbf{I}) M'
        =
        \biggl[
            \prod_{k = 1}^h \prod_{j = 2}^{\beta_k} X_{i_{k(j - 1)}i_{kj}}
        \biggr]
        M^T \sym{E}^{\alpha_0}w_{i_1}
        \biggl[
            \prod_{k = 1}^{h - 1} w_{i_{k\beta_k}}^T \sym{E}^{\alpha_k + 1} w_{i_{(k + 1)1}}
        \biggr]
        w_{i_{h\beta_h}}^T \sym{E}^{\alpha_h} M'
        .
    \end{equation}
    For each $k\in [h - 1]$, we have $\bigl\|w_{i_{k\beta_k}}^T \sym{E}^{\alpha_k + 1} w_{i_{(k + 1)1}}\bigr\| \le \|E\|^{\alpha_k + 1}$.
    For each $k\in [h]$ and $j\in [\beta_k]$, $\bigl|X_{i_{k(j - 1)}i_{kj}}\bigr| \le \|X\|_\infty$.
    Therefore the above leads to
    \begin{equation*}
    \begin{aligned}
        \left\|
            M^T D(\boldsymbol{\alpha}, \boldsymbol{\beta}, \mathbf{I}) M'
        \right\|
        & \le \|X\|_\infty^{(\beta_1 - 1) + \ldots + (\beta_h - 1)}
        \|E\|^{(\alpha_1 + 1) + \ldots (\alpha_{h - 1} + 1)}
        \left\| M^T \sym{E}^{\alpha_0}w_{i_1}\right\|
        \left\| w_{i_{h\beta_h}}^T \sym{E}^{\alpha_h} M'\right\|
        \\
        & = \|X\|_\infty^{\beta - h}
        \|E\|^{\alpha - \alpha_0 - \alpha_h + h - 1}
        \left\| M^T \sym{E}^{\alpha_0}w_{i_1}\right\|
        \left\| w_{i_{h\beta_h}}^T \sym{E}^{\alpha_h} M'\right\|.
    \end{aligned}
    \end{equation*}
    The proof is complete.
\end{proof}

This concludes the first step.
Next, we will sum up the bounds in Eq. \eqref{eq:T(alpha,beta)-bound-2} over all choices of $(\boldsymbol{\alpha}, \boldsymbol{\beta}) \in \Pi_h(\gamma)$, over all $h\ge \gamma/2 + 1$ to get a bound for $M^T\myT_\nu^{(\gamma)}M'$, proving Lemma \ref{lem:T-gamma-bound}.

\begin{proof}[Proof of Lemma \ref{lem:T-gamma-bound}]
    Let us consider the case $\gamma \le 10\log (m + n)$ first.
    Then we have $\alpha_0 \vee \alpha_h \le 10\log (m + n)$, so
    Lemma \ref{lem:T(alpha,beta)-bound-1} gives
    \begin{equation*}
    \begin{aligned}
        \|{\oldcomment{}M^T\myT_\nu(\boldsymbol{\alpha}, \boldsymbol{\beta})M'}\|
        \le \myL_\nu 
            \binom{\gamma + \beta - 2}{\beta - 1}
        \myD\myD' \frac{
            (2r)^{\beta}
            \|W^T\sym{E}W\|_\infty^{\beta - h}
            \|E\|^{\alpha + h - 1}
        }
        {\lambda_S^{\gamma + 1 - \beta}\Delta_S^{\beta - 1}}
        .
    \end{aligned}
    \end{equation*}
    Again, we temporarily define $X \defeq W^T\sym{E}W$ for convenience.
    Rearranging the terms, we have
    \begin{equation}  \label{eq:main-proof-temp3}
    \begin{aligned}
        \|{\oldcomment{}M^T\myT_\nu(\boldsymbol{\alpha}, \boldsymbol{\beta})M'}\|
        \le 2r\myL_\nu\myD\myD'
        \binom{\gamma + \beta - 2}{\beta - 1}
        \left[ \frac{\|E\|}{\lambda_S} \right]^{\gamma_1}
        \left[ \frac{2r\|X\|_\infty}{\Delta_S} \right]^{\gamma_2}
        \left[ \frac{\sqrt{2r}\|E\|}{\sqrt{\Delta_S\lambda_S}} \right]^{\gamma_3}
        ,
    \end{aligned}
    \end{equation}
    where we temporarily define the following:
    \begin{equation*}
    \begin{aligned}
        & \gamma_1 = \gamma_1(\boldsymbol{\alpha}) \defeq \alpha_1 + \ldots + \alpha_{h - 1} - (h - 1),
        \\
        & \gamma_2 = \gamma_2(\boldsymbol{\alpha}) \defeq \alpha_0 + \alpha_h,
        \quad \gamma_2 = \gamma_2(\boldsymbol{\beta}) \defeq \beta - h,
        \quad \gamma_3 = \gamma_3(h) \defeq 2(h - 1)
        .
    \end{aligned}
    \end{equation*}
    Since $R_1$ upper bounds the former four powers and $R_2$ upper bounds the latter,
    we get
    \begin{equation*} 
    \begin{aligned}
        \|{\oldcomment{}M^T\myT_\nu(\boldsymbol{\alpha}, \boldsymbol{\beta})M'}\|
        & \le 2r\myL_\nu\myD\myD' \binom{\gamma + \beta - 2}{\beta - 1}
        R_1^{\gamma_1(\boldsymbol{\alpha}) + \alpha_0 + \alpha_h + \gamma_2(\boldsymbol{\beta})} R_2^{\gamma_3(h)}
        \\
        & = 2r\myL_\nu\myD\myD' \binom{\gamma + \beta - 2}{\beta - 1}
        R_1^{\gamma - 2h + 2} R_2^{2h - 2}.
    \end{aligned}
    \end{equation*}
    Plugging this bound into Eq. \eqref{eq:generic-T-defn-1}, we get
    \begin{equation}  \label{eq:main-proof-temp6}
    \begin{aligned}
        \left\| {\oldcomment{}M^T\myT_\nu^{(\gamma)}M'} \right\|
        & \le 2r\myL_\nu\myD\myD'
        \sum_{h = 1}^{\lfloor \gamma /2 \rfloor + 1}
        \sum_{(\boldsymbol{\alpha}, \boldsymbol{\beta}) \in \Pi_h(\gamma)}
        \binom{\gamma + \beta - 2}{\beta - 1}
        R_1^{\gamma - 2h + 2} R_2^{2h - 2}
        \\
        & \le 2r\myL_\nu\myD\myD'
        \sum_{h = 1}^{\lfloor \gamma /2 \rfloor + 1}
        \sum_{\beta = h}^{\gamma + 2 - h}
        \binom{\gamma + \beta - 2}{\beta - 1}
        R_1^{\gamma - 2h + 2} R_2^{2h - 2}
        \left| \Pi_h(\gamma, \beta) \right|,
    \end{aligned}
    \end{equation}
    where
    \begin{equation*}
        \Pi_h(\gamma, \beta) \defeq \{(\boldsymbol{\alpha}, \boldsymbol{\beta}) \in \Pi_h(\gamma): \beta_1 + \ldots + \beta_h = \beta \}
    \end{equation*}
    An element of this set is just a tuple $(\alpha_0, \ldots, \alpha_h, \beta_1, \ldots, \beta_h)$ such that
    \begin{equation*}
        \beta_1, \ldots, \beta_h \ge 1,
        \quad \sum_{i = 1}^h \beta_i = \beta,
        \ \text{ and } \
        \alpha_0, \alpha_h \ge 0,
        \quad
        \alpha_1, \ldots, \alpha_{h - 1},
        \quad
        \sum_{i = 0}^h \alpha_i = \gamma + 1 - \beta.
    \end{equation*}
    The number of ways to choose such a tuple is
    \begin{equation*}
        \left| \{(\boldsymbol{\alpha}, \boldsymbol{\beta}) \in \Pi_h(\gamma): \beta_1 + \ldots + \beta_h = \beta \} \right|
        = \binom{\beta - 1}{h - 1} \binom{\gamma + 2 - \beta}{h}.
    \end{equation*}
    Plugging into Eq. \eqref{eq:main-proof-temp6}, we obtain
    \begin{equation}  \label{eq:main-proof-temp7}
    \begin{aligned}
        \left\| {\oldcomment{}M^T\myT_\nu^{(\gamma)}M'}
        \right\|
        & \le 2r\myL_\nu\myD\myD'
        \sum_{h = 1}^{\lfloor \gamma /2 \rfloor + 1}
        \sum_{\beta = h}^{\gamma + 2 - h}
        \binom{\gamma + \beta - 2}{\beta - 1} \binom{\beta - 1}{h - 1} \binom{\gamma + 2 - \beta}{h}
        R_1^{\gamma - 2h + 2}
        R_2^{2h - 2}
        \\
        & = 2r\myL_\nu\myD\myD'
        \sum_{\beta = 1}^{\gamma + 1}
        \binom{\gamma + \beta - 2}{\beta - 1}
        \sum_{h = 1}^{\beta \wedge (\gamma + 2 - \beta)}
        \binom{\beta - 1}{h - 1} \binom{\gamma + 2 - \beta}{h}
        R_1^{\gamma - 2h + 2}
        R_2^{2h - 2}.
    \end{aligned}
    \end{equation}
    Consider two cases for $h$ and $\gamma$:
    \begin{enumerate}
        \item $\gamma \ge 2h - 1$.
        Let $R \defeq R_1 \vee R_2$.
        The contribution is at most:
        \begin{equation*}
        \begin{aligned}
            & 2r\myL_\nu\myD\myD'
            \sum_{\beta = 1}^{\gamma + 1}
            \binom{\gamma + \beta - 2}{\beta - 1}
            \sum_{h = 1}^{\beta \wedge (\gamma + 2 - \beta)}
            \binom{\beta - 1}{h - 1} \binom{\gamma + 2 - \beta}{h}
            R_1R^{\gamma - 1}
            \\
            & \le 2r\myL_\nu\myD\myD' R_1R^{\gamma - 1}
            \sum_{\beta = 1}^{\gamma + 1}
            \binom{\gamma + \beta - 2}{\beta - 1}
            \binom{\gamma + 1}{\beta}
            \le 2r\myL_\nu\myD\myD' R_1R^{\gamma - 1}
            \sum_{\beta = 1}^{\gamma + 1}
            \binom{\gamma + 1}{\beta} 2^{\gamma + \beta - 2}
            \\
            & \le 2r\myL_\nu\myD\myD' R_1R^{\gamma - 1} 2^{\gamma - 2} 3^{\gamma + 1}
            = 9r\myL_\nu\myD\myD' R_1 (6R)^{\gamma - 1}.
        \end{aligned}
        \end{equation*}

        \item $\gamma = 2h - 2$.
        This can only happens if $\gamma$ is even.
        Then $\alpha_0 = \alpha_h = 0$ and $\alpha_1 = \ldots = \alpha_{h - 1} = \beta_1 = \ldots = \beta_h = 1$.
        Let $(\boldsymbol{\alpha}^*, \boldsymbol{\beta}^*)$ denote the corresponding tuple.
        The previous computations are bad for this case, so we will go back to the definition to bound $\myT_\nu(\boldsymbol{\alpha}^*, \boldsymbol{\beta}^*)$.
        Plugging into Eq. \eqref{eq:T(alpha,beta)-defn} and simplifying, we have
        \begin{equation*}
            \myT_\nu(\boldsymbol{\alpha}^*, \boldsymbol{\beta}^*)
            = \sum_{\mathbf{I} \in [2r]^h} 
            \myC_\nu(\mathbf{I})
            \left(M^Tw_{i_1}\right)
            \left(w_{i_{h\beta_h}}^TM'\right)
                \prod_{k = 1}^{h - 1} w_{i_k}^T \sym{E}^2 w_{i_{k + 1}}
            ,
        \end{equation*}
        where each block of consecutive indices in $\mathbf{I}$ now consists of only one index, so we can just denote $\mathbf{I} = (i_1, i_2, \ldots, i_h)$.
        Temporarily let $Y\in \R^{(m + n)\times (m + n)}$ be a matrix such that $Y_{ij} = w_i^T\sym{E}^2w_j$ for $|i - j| \in \{0, r\}$ and $0$ otherwise, and let $\lambda_+(\mathbf{I}) \defeq (|\lambda_{i_k}|)_k$ for $\mathbf{I} = (i_k)_k$.
        We further consider two subcases for $\mathbf{I}$:
        \begin{enumerate}[label*=\arabic*.]
            \item $\lambda_+(\mathbf{I})$ is non-uniform, i.e. there is $k$ so that $|\lambda_{i_k}| \neq |\lambda_{i_{k + 1}}|$, meaning $|i_k - i_{k + 1}| \notin \{0, r\}$.
            Then $\bigl|w_{i_k}^T \sym{E}^2 w_{i_{k + 1}} \bigr| \le \|Y\|_\infty$.
            The rest of the product at the end can be bounded by $\|E\|^2$.
            The total contribution of this subcase is at most
            \begin{equation*}
                2r\myL_\nu\myD\myD' \binom{\gamma + \beta - 2}{\beta - 1} R_2^{2h - 4} \frac{2r\|Y\|_\infty}{\lambda_S\Delta_S}
                \le 2r\myL_\nu\myD\myD'
                \binom{3\gamma/2}{\gamma/2} 
                R_3R_2^{\gamma - 2},
            \end{equation*}
            where we use the same computations leading up to Eq. \eqref{eq:main-proof-temp7}, noting that $h = \beta = \gamma/2 + 1$.

            \item $\lambda_+(\mathbf{I})$ is uniform, i.e. $\mathbf{I} = (i_k)_{k = 1}^h$ where each $i_k\in \{i, i + r\}$ for some $i\in [r]$.
            If $i\notin S$, then $\myC_\nu(\mathbf{I}) = 0$.
            Suppose $i\in S$, we can apply Eq. \eqref{eq:integral-coef-bound-1} in Lemma \ref{lem:integral-coef-bound}, noting that $\beta_S(\mathbf{I}) = \beta$, $\beta_{S^c}(\mathbf{I}) = 0$, and $\lambda_S(\mathbf{I}) = \Delta_S(\mathbf{I}) = \lambda_i$ in this case, to get
            \begin{equation*}
                |\myC_\nu(\mathbf{I})|
                \le \myL_\nu \binom{\gamma + \beta - 2}{\beta - 1}
                \frac{1}{
                    \lambda_S(\mathbf{I})^{\gamma + 1 - \beta}
                    \Delta_S(\mathbf{I})^{\beta - 1}
                }
                = \myL_\nu \binom{\gamma + h - 2}{h - 1} \frac{1}{\lambda_i^\gamma}.
            \end{equation*}
            Therefore the total contribution of this subcase is at most
            \begin{equation*}
            \begin{aligned}
                & \myD\myD' \binom{\gamma + h - 2}{h - 1} \sum_{i\in S}
                \sum_{\mathbf{I}\in \{i, i + 1\}^h}
                \frac{|w_i^T\sym{E}^2w_i|^{h - 1}}{\lambda_i^\gamma}
                \le s\myD\myD' \binom{3\gamma/2}{\gamma/2}
                \frac{2^h\|E\|^{2(h - 1)}}{\lambda_S^\gamma}
                \\
                & \le 2r\myL_\nu\myD\myD' \binom{3\gamma/2}{\gamma/2}
                \frac{2^{\gamma/2}\|E\|^\gamma}{\lambda_S^\gamma}
                \le 2r\myL_\nu\myD\myD' \binom{3\gamma/2}{\gamma/2}
                (R_1\sqrt{2})^\gamma.
            \end{aligned}
            \end{equation*}
            \end{enumerate}
            Therefore, the contribution of the case $h = \gamma/2 + 1$ is at most
            \begin{equation*}
                2r\myL_\nu\myD\myD' \binom{3\gamma/2}{\gamma/2}
                \left[ R_3R_2^{\gamma - 2} + (R_1\sqrt{2})^\gamma \right]
                \le 4r\myL_\nu\myD\myD'
                \left[
                    \frac{27}{4} R_3\left(
                        \tfrac{3\sqrt{3}}{2}R_2
                    \right)^{\gamma - 2}
                    + \left(
                        \tfrac{3\sqrt{3}}{\sqrt{2}}R_1
                    \right)^\gamma
                \right],
            \end{equation*}
            since $R_3 \ge 2r\|Y\| / (\lambda_S\Delta_S)$.
    \end{enumerate}
    Summing up the contributions from both cases, we obtain
    \begin{equation}  \label{eq:T-gamma-bound-temp1}
        \left\| {\oldcomment{}M^T\myT_\nu^{(\gamma)}M'}
        \right\| \le
        r\myL_\nu\myD\myD' \left[
            9 R_1 (6R)^{\gamma - 1}
            + \one\{\gamma \text{ even}\}
            \left(
                4\left(
                        \tfrac{3\sqrt{3}}{\sqrt{2}}R_1
                \right)^\gamma
                + 27R_3\left(
                        \tfrac{3\sqrt{3}}{2}R_2
                    \right)^{\gamma - 2}
            \right)
        \right].
    \end{equation}

    Now let us consider the case $\gamma > 10\log(m + n)$.
    Instead of using $\myD$ and $\myD'$, we can simply use the naive bounds
    \begin{equation*}
        \sum_{i = 1}^{2r} \|w_i^T\sym{E}^\alpha M\| \le 2r\|E\|^\alpha\|M\|
        \ \text{ and } \
        \sum_{i = 1}^{2r} \|w_i^T\sym{E}^\alpha M'\| \le 2r\|E\|^\alpha\|M'\|.
    \end{equation*}
    Plugging in these bounds into the previous computations has the same effect as using $\myD = \|M\|$, $\myD' = \|M'\|$, and $\uppnormE = \uppnormE' = \|E\|$.
    The conditions on $R_1$, $R_2$ and $R_3$ remain the same so they still hold.
    Thus we still have Eq. \eqref{eq:T-gamma-bound-temp1}, with the substitutions above.
    The proof is complete.
\end{proof}




\appendix

\section{Proof of the full matrix completion theorem}  \label{sec:matrix-completion-full-proof}

In this section, we prove Theorem \ref{thm:matrix-completion}, with the sampling density condition \eqref{eq:matcom-sample-density-cond-full} replacing \eqref{eq:matcom-sample-density-cond}.

\begin{proof}[Proof of the full Theorem \ref{thm:matrix-completion}]
    Let $C_2 = 1/c$ for the constant $c$ in Theorem \ref{thm:main-random}.
    We rewrite the assumptions below:
    \begin{enumerate}
        \item \emph{Signal-to-noise:} $\sigma_1 \ge 100r\kappa\sqrt{r_{\max}N}$.

        \item \emph{Sampling density:} this is equivalent to the conjunction of three conditions:
        \begin{align}
            \label{eq:matcom-sample-density-cond-1}
            p & \ge \frac{Cr^4r_{\max}\mu_0^2K_{A,\noise}^2}{\eps^2}
                \left(\frac{1}{m} + \frac{1}{n}\right),
            \\
            \label{eq:matcom-full-sample-density-cond-2}
            p & \ge C\left(\frac{1}{m} + \frac{1}{n}\right)\log^{10}N,
            \\
            \label{eq:matcom-full-sample-density-cond-3}
            p & \ge \frac{Cr^3 K_{A,\noise}^2}{\eps^2}
            \left(1 + \frac{\mu_0^2}{\log^2 N}\right)
            \left(1 + \frac{r^3\log N}{N}\right)
            \left(\frac{1}{m} + \frac{1}{n}\right)\log^6N.
        \end{align}
    \end{enumerate}
    Let $\rho \defeq \hat{p}/p$.
    From the sampling density assumption, a standard application of concentration bounds \cite{hoeffding1963,chernoff1952} guarantees that, with probability $1 - O(N^{-2})$.
    \begin{equation}  \label{eq:matcom-full-proof-fact1}
        0.9 \le 1 - \frac{1}{\sqrt{N}}
        \le 1 - \frac{\log N}{\sqrt{pmn}}
        \le \rho
        \le 1 + \frac{\log N}{\sqrt{pmn}}
        \le 1 + \frac{1}{\sqrt{N}}
        \le 1.1
        .
    \end{equation}
    Furthermore, an application of well-established bounds on random matrix norms gives
    \begin{equation}  \label{eq:matcom-full-proof-fact2}
        \|E\|
        \le 2\kappa\sqrt{N}
        ,
    \end{equation}
    with probability $1 - O(N^{-1})$.
    See \cite{bandeira2014,vu2005}, \cite[Lemma A.7]{tranVu2022} or \cite{bandeira2014} for detailed proofs.
    Therefore we can assume both Eqs. \eqref{eq:matcom-full-proof-fact1} and \eqref{eq:matcom-full-proof-fact2} at the cost of an $O(N^{-1})$ exceptional probability.

    Let $C_0 \defeq 40$.
    The index $s$ chosen in the SVD step of \nameref{algo:matrix-completion} is the largest such that
    \begin{equation*}
        \hat{\delta}_s \ge C_0K_{A,\noise}\sqrt{r_{\max}N/\hat{p}}
        = C_0\rho^{-1/2}\kappa\sqrt{r_{\max}N}.
    \end{equation*}

    Firstly, we show that SVD step is guaranteed to choose a valid $s\in [r]$.
    Choose an index $l\in [r]$ such that $
        \delta_l \ge \sigma_1/r \ge 100\kappa\sqrt{r_{\max}N}
    $, we have
    \begin{equation*}
    \begin{aligned}
        \hat{\delta}_l \ge \rho^{-1/2}\tilde{\delta_l}
        \ge \rho^{-1/2}(\delta_l - 2\|E\|)
        \ge (100r_{\max}^{1/2} - 4)\rho^{-1/2}\kappa\sqrt{N}
        \ge 2C_0\rho^{-1/2}\kappa\sqrt{r_{\max}N},
    \end{aligned}
    \end{equation*}
    so the cutoff point $s$ is guaranteed to exist.
    To see why $s\in [r]$, note that
    \begin{equation*}
        \hat{\delta}_{r + 1} \le \rho^{-1/2}\tilde{\sigma}_{r + 1}
        \le \rho^{-1/2}\|E\| \le 2\rho^{-1/2}\kappa\sqrt{r_{\max}N}
        < C_0\rho^{-1/2}\kappa\sqrt{r_{\max}N}.
    \end{equation*}
    
    We want to show that the first three steps of \nameref{algo:matrix-completion} recover $A$ up to an absolute error $\eps$, namely $\|\hat{A}_s - A\|_\infty \le \eps$, we will first show that $\|\tilde{A}_s - A\|_\infty \le \eps/2$ (with probability $1 - O(N^{-1})$).
    We proceed in two steps:

    \begin{enumerate}
        \item
        We will show that $\|A_s - A\|_\infty \le \eps/4$ when $C$ is large enough.
        To this end, we establish:
        \begin{equation}  \label{eq:matcom-proof-fact5}
            \sigma_{s + 1} \le r\delta_{s + 1}
            \le r(\tilde{\delta}_{s + 1} + 2\|E\|)
            \le r(C_0\rho^{-1/2}\sqrt{r_{\max}} + 4)\kappa\sqrt{N}
            \le 2rC_0K_{A,\noise}\sqrt{r_{\max}N/p}.
        \end{equation}
        For each fixed indices $j, k$, we have
        \begin{equation*}
        \begin{aligned}
            |(A_s - A)_{jk}|
            & = \left| U_{j, \cdot}^T\Sigma_{[s + 1, r]}V_{k, \cdot} \right|
            \le \sigma_{s + 1} \|U\|_{2, \infty} \|V\|_{2, \infty}
            \le 2rC_0K_{A,\noise}\sqrt{\frac{r_{\max}N}{p}} \frac{r\mu_0}{\sqrt{mn}}
            \\
            &
            = \sqrt{
                \frac{4C_0^2r^4r_{\max}\mu_0^2K_{A,\noise}^2}{p}
                \left(\frac{1}{m} + \frac{1}{n}\right)
            }
            \le \eps/4.
        \end{aligned}
        \end{equation*}
        where the last inequality comes from the assumption \eqref{eq:matcom-sample-density-cond-1} if $C$ is large enough.
        Since this holds for all pairs $(j, k)$, we have
        $\|A_s - A\|_\infty \le \eps/4$.

        \item Secondly, we will show that $\|\tilde{A}_s - A_s\|_\infty \le \eps/4$ with probability $1 - O(N^{-1})$.
        We aim to use Theorem \ref{thm:main-random}, so let us translate its terms into the current context.
        By the sampling density condition,
        we have the following lower bounds for $\delta_s$ and $\sigma_s$:
        \begin{equation}  \label{eq:delta-sigma-s-lower-bound}
        \begin{aligned}
            \sigma_s \ge \delta_s
            & \ge \tilde{\delta}_s - 2\|E\|
            \ge C_0\rho^{-1/2}\kappa\sqrt{r_{\max}N} - 2\|E\|
            \ge .9C_0\kappa\sqrt{r_{\max}N}.
        \end{aligned}
        \end{equation}
        
        Consider the condition \eqref{eq:main-random-S-cond}.
        If it holds, then we can apply Theorem \ref{thm:main-random}.
        We want
        \begin{equation*}
            \frac{\kappa\sqrt{N}}{\sigma_s}
            \vee \frac{r\kappa(\sqrt{\log N} + K\|U\|_\infty \|V\|_\infty\log N)}{\delta_s}
            \vee \frac{\kappa\sqrt{rN}}{\sqrt{\delta_s\sigma_s}}
            \le \frac{1}{16}
        \end{equation*}
        By Eq. \eqref{eq:delta-sigma-s-lower-bound}, we can replace all three denominators above with $.9C_0\kappa\sqrt{r_{\max}N}$.
        Additionally, $\|U\|_\infty \le \|U\|_{2,\infty} \le \sqrt{\frac{r\mu_0}{m}}$ and $\|V\|_\infty \le \|V\|_{2,\infty} \le \sqrt{\frac{r\mu_0}{n}}$, so we can replace them with these upper bounds.
        We also replace $K$ with $p^{-1/2}$ (its definition).
        We want
        \begin{equation*}
            \frac{
                \kappa\sqrt{N}
                \ \vee \ \kappa\sqrt{rN}
                \ \vee \
                r\kappa(\sqrt{\log N} + \frac{r\mu_0}{\sqrt{pmn}}\log N)
            }{.9C_0\kappa\sqrt{r_{\max}N}}
            \le \frac{1}{16},
        \end{equation*}
        which is equivalent to
        \begin{equation*}
            \frac{
                1 \ \vee \ \sqrt{r}
                \ \vee \
                r(\sqrt{\frac{\log N}{N}} + \frac{r\mu_0}{\sqrt{pmnN}}\log N)
            }{.9C_0\sqrt{r_{\max}}}
            \le \frac{1}{16}
        \end{equation*}
        which easily holds.
        Therefore we can apply Theorem \ref{thm:main-random}.
        We get, for a constant $C_1$,
        \begin{equation*}
             \| \tilde{A}_s - A_s \|_\infty
            \le C_1 \myD_{UV} \myD_{VU} \cdot
            r\sigma_s R_s
            + \frac{1}{N}.
        \end{equation*}

        Let us simplify the first term in the product, $\myD_{UV}\myD_{VU}$.
        \begin{equation*}
        \begin{aligned}
            & \myD_{UV}
            = \frac{K\|U\|_{2, \infty}\log^3N}{\sqrt{rN}}
            + \frac{\log^{3/2}N}{\sqrt{N}} + \frac{\|V\|_{2, \infty}\log N}{\sqrt{r}}
            \\
            & \le \frac{\sqrt{\mu_0}\log^3N}{\sqrt{pmN}}
            + \frac{\log^{3/2}N}{\sqrt{N}} + \frac{\sqrt{\mu_0}\log N}{\sqrt{n}}
            \le \frac{\log^{3/2}N}{\sqrt{N}}
            + \frac{\sqrt{2\mu_0}\log N}{\sqrt{n}}
            ,
        \end{aligned}
        \end{equation*}
        where the first inequality comes from \eqref{eq:matcom-full-sample-density-cond-2} if $C$ is large enough.
        Similarly,
        \begin{equation*}
            \myD_{VU}
            \le N^{-1/2}\log^{3/2}N
            + m^{-1/2}\sqrt{2\mu_0}\log N.
        \end{equation*}
        Therefore,
        \begin{equation*}
        \begin{aligned}
            & \myD_{UV}\myD_{VU}
            \le \frac{\log^3N}{N}
            + \frac{\sqrt{\mu_0}\log^{5/2}N}{\sqrt{N}}
                \cdot \frac{\sqrt{2m} + \sqrt{2n}}{\sqrt{mn}}
            + \frac{2\mu_0\log^2N}{\sqrt{mn}}
            \\
            & \le \log^2N
            \frac{\log N + 4\sqrt{\mu_0}\sqrt{\log N} + 4\mu_0}{2\sqrt{mn}}
            \le \log^2N
            \frac{\log N + 4\mu_0}{\sqrt{mn}}.
        \end{aligned}
        \end{equation*}

        For the second term, we have the following upper bound:
        \begin{equation*}
        \begin{aligned}
            r\sigma_s R_s
            & \le r\sigma_s \left(
                \frac{\kappa\sqrt{N}}{\sigma_s}
                + \frac{r\kappa(\sqrt{\log N} + \frac{r\mu_0}{\sqrt{mn}}K\log N}{\delta_s}
                + \frac{r\kappa^2N}{\delta_s\sigma_s}
            \right)
            \\
            & = r\left(
                \kappa\sqrt{N}
                + \frac{r\kappa\sigma_s}{\delta_s}
                \left(
                    \sqrt{\log N} + \frac{r\mu_0\log N}{\sqrt{pmn}}
                \right)
                + \frac{r\kappa^2 N}{\delta_s}
            \right)
            \\
            & \le r\left(
                \kappa\sqrt{N}
                + r^2\kappa
                \left(
                    \sqrt{\log N} + \frac{r\mu_0\log N}{\sqrt{pmn}}
                \right)
                + \frac{r\kappa^2 N}{.9C_0\kappa\sqrt{rN}}
            \right)
            \\
            & \le r^{3/2}\kappa \left(
                \sqrt{2N} + r^{3/2}
                \left(
                    \sqrt{\log N} + \frac{r\mu_0\log N}{\sqrt{pmn}}
                \right)
            \right).
        \end{aligned}
        \end{equation*}
        Under the condition \eqref{eq:matcom-full-sample-density-cond-3}, we have
        \begin{equation*}
            pmn \ge Cr^3\mu_0^2\log^4 N
            \implies
            \frac{r\mu_0\log N}{\sqrt{pmn}} < .1\sqrt{\log N},
        \end{equation*}
        so the above is simply upper bounded by
        \begin{equation*}
            \frac{\sqrt{2}r^{3/2}K_{A,\noise}}{\sqrt{p}}
            \left(
                \sqrt{N}
                + r^{3/2}\sqrt{\log N}
            \right).
        \end{equation*}

        
        
        Multiplying the two terms, we have by Theorem \ref{thm:main-random},
        \begin{equation}  \label{eq:matcom-proof-bound1}
        \begin{aligned}
            & \| \tilde{A}_s - A_s \|_\infty
            \le \log^2N
            \cdot\frac{\log N + 4\mu_0}{\sqrt{mn}}
            \cdot \frac{\sqrt{2}r^{3/2}K_{A,\noise}}{\sqrt{p}}
            \left(
                \sqrt{N}
                + r^{3/2}\sqrt{\log N}
            \right)
            \\
            & \ \le \sqrt{
                \frac{2r^3 K_{A,\noise}^2 \log^6N}{p}
                \left(1 + \frac{4\mu_0^2}{\log^2 N}\right)
                \left(1 + \frac{r^3\log N}{N}\right)
                \left(\frac{1}{m} + \frac{1}{n}\right)
            } \le \eps/4.
        \end{aligned}
        \end{equation}
    \end{enumerate}
    where the last inequality comes from the condition \eqref{eq:matcom-full-sample-density-cond-3} if $C$ is large enough.

    After the two steps above, we obtain $\|\tilde{A}_s - A\|_\infty \le \eps/2$ with probablity $1 - O(N^{-1})$.
    Finally, we get, using Fact \eqref{eq:matcom-full-proof-fact1} and the triangle inequality,
    \begin{equation*}
        \|\hat{A}_s - A\|_\infty = \left\| \rho^{-1}\tilde{A}_s - A \right\|_\infty
        \le \frac{1}{\rho}\|\tilde{A}_s - A\|_\infty
        + \left| \frac{1}{\rho} - 1 \right| \|A\|_\infty
        \le \frac{\eps/2}{.9} + \frac{K_A}{.9\sqrt{N}}
        < \eps.
    \end{equation*}
    This is the desired bound.
    The total exceptional probability is $O(N^{-1})$.
    The proof is complete.
\end{proof}

\section{Proofs of technical lemmas}  \label{sec:technical-lemmas}

\subsection{Proof of bound for contour integrals of polynomial reciprocals}  \label{sec:integral-bound}

In this section, we prove Lemma \ref{lem:integral-coef-bound}, which provides the necessary bounds on the integral coefficients to advance the second step of the main proof (Section \ref{sec:main-proof-deterministic-bound-series}).
Recall that the integrals we are interested in have the form
\begin{equation}  \label{eq:integral-coef-generic}
    \myC_\nu(\mathbf{I})
    \defeq \oint_{\Gamma_S} \frac{z^\nu \td z}{2\pi i}
    \frac{1}{z^{\gamma + 1}} \prod_{k = 1}^{\beta}
    \frac{\lambda_{i_k}}{z - \lambda_{i_k}},
    \quad \text{ where } \nu \in \{0, 1\}
    \ \text{ and } \beta \le \gamma + 1.
\end{equation}
Let the multiset $\{\lambda_{i_k}\}_{k\in [\beta]} = A \cup B$, where $A \defeq \{a_i\}_{i\in [l]}$ and $B \defeq \{b_j\}_{j\in [k]}$, where each $a_i\in S$ and each $b_j\notin S$, having multiplicities $m_i$ and $n_j$ respectively.
We can rewrite the above into
\begin{equation}  \label{eq:integral-coef-regroup-1}
    \myC_\nu(\mathbf{I})
    = \prod_{i = 1}^l a_i^{m_i} \prod_{j = 1}^k b_j^{n_j}
     C(n_0; A, \mathbf{m}; B, \mathbf{n}),
\end{equation}
where
\begin{equation}  \label{eq:integral-coef-regroup-2}
    C(n_0; A, \mathbf{m}; B, \mathbf{n})
    \defeq \oint_{\Gamma_A} \frac{\td z}{2\pi i}
    \frac{1}{z^{n_0}} \prod_{j = 1}^k \frac{1}{(z - b_j)^{n_j}}
    \prod_{i = 1}^l \frac{1}{(z - a_i)^{m_i}},
\end{equation}
where $n_0 = \gamma + 1 - \nu$.
The $m_i$'s and $n_j$'s satisfy $\sum_i m_i + \sum_j n_j \le \gamma + 1$.
We can remove the set $S$ and simply denote the contour by $\Gamma_A$ without affecting its meaning.
The next three results will build up the argument to bound these sums and ultimately prove the target lemmas.

\begin{lemma}  \label{lem:integral-bound}
 Let $A = \{a_i\}_{i\in [l]}$ and $B = \{b_j\}_{j\in [k]}$ be disjoint set of complex non-zero numbers and $\mathbf{m} = \{m_i\}_{i\in [l]}$ and $n_0$ and $\mathbf{n} = \{n_j\}_{j\in [k]}$ be nonnegative integers such that $m + n + n_0 \ge 2$, where $m = \sum_i m_i$ and $n \defeq \sum_{i\ge 1} n_i$.
Let $\Gamma_A$ be a contour encircling all numbers in $A$ and none in $B \cup \{0\}$.
Let $a, d > 0$ be arbitrary such that:
\begin{equation}  \label{eq:integral-bound-a-d-cond}
    d \le a,
    \quad \quad a \le \min_i |a_i|,
    \quad \quad d \le \min_{i, j} |a_i - b_j|.
\end{equation}
Suppose that $0\le m'_i \le m_i$ for each $i\in [l]$ and that $m'\defeq \sum_{i = 1}^k m'_i \le n_0$.
Then for $C(n_0; A, \mathbf{m}; B, \mathbf{n})$ defined Eq. \eqref{eq:integral-coef-regroup-2}, we have
\begin{equation}  \label{eq:integral-bound}
    \left|
        C(n_0; A, \mathbf{m}; B, \mathbf{n})
    \right|
    \le  \binom{m + n + n_0 - 2}{m - 1}
    \frac{1}{a^{n_0 - m'}d^{m + n - 1}}
    \prod_{i = 1}^l \frac{1}{|a_i|^{m'_i}}
\end{equation}
\end{lemma}

\begin{proof}
    Firstly, given the sets $A$ and $B$ and the notations and conditions in Lemma \ref{lem:integral-bound}, the weak bound below holds
    \begin{equation}  \label{eq:integral-bound-weakest}
        \left|
            C(n_0; A, \mathbf{m}; B, \mathbf{n})
        \right|
        \le \binom{m + n + n_0 - 2}{m - 1}
        \frac{1}{d^{m + n + n_0 - 1}}.
    \end{equation}
    We omit the details of the proof, which is a simple induction argument.
    We now use Eq. \eqref{eq:integral-bound-weakest} to prove the following:
    \begin{equation}  \label{eq:integral-bound-weak}
        \left|
            C(n_0; A, \mathbf{m}; B, \mathbf{n})
        \right|
        \le \binom{m + n + n_0 - 2}{m - 1}
        \frac{1}{a^{n_0}d^{m + n - 1}}.
    \end{equation}
    We proceed with induction.
    Let $P_1(N)$ be the following statement:
    ``For any sets $A$ and $B$, and the notations and conditions described in Lemma \ref{lem:integral-bound}, such that $m + n + n_0 = N$, Eq. \eqref{eq:integral-bound-weak} holds.''

    Since $m + n + n_0 \ge 2$, consider $N = 2$ for the base case.
    The only case where the integral is non-zero is when $m = 1$ and $n + n_0 = 1$, meaning $A = \{a_1\}$, $m_1 = 1$ and either $B = \varnothing$ and $n_0 = 1$, or $B = \{b_1\}$ and $n_1 = 1$, $n_0 = 0$.
    The integral yields $a_1^{-1}$ in the former case and $(a_1 - b_1)^{-1}$ in the latter, confirming the inequality in both.

    Consider $n \ge 3$ and assume $P_1(n - 1)$.
    If $m = 0$, the integral is again $0$.
    If $n_0 = 0$, Eq. \eqref{eq:integral-bound-weak} automatically holds by being the same as Eq. \eqref{eq:integral-bound-weakest}.
    Assume $m, n_0 \ge 1$. There must then be some $i\in [l]$ such that $m_i \ge 1$, without loss of generality let $1$ be that $i$.
    We have
    \begin{equation}  \label{eq:integral-bound-split}
    \begin{aligned}
        C(n_0; A, \mathbf{m}; B, \mathbf{n})
        = \frac{1}{a_1} \Bigl[
            C(n_0 - 1; A, \mathbf{m}; B, \mathbf{n})
            - C(n_0; A, \mathbf{m}^{(1)}; B, \mathbf{n})
        \Bigr]
    \end{aligned}
    \end{equation}
    where $\mathbf{m}^{(i)}$ is the same as $\mathbf{m}$ except that the $i$-entry is $m_i - 1$.
    
    Consider the first integral on the right-hand side.
    Applying $P_1(N - 1)$, we get
    \begin{equation}  \label{eq:integral-bound-1st}
        \left|
            C(n_0 - 1; A, \mathbf{m}; B, \mathbf{n})
        \right|
        \le \binom{m + n + n_0 - 3}{m - 1} \frac{1}{a^{n_0 - 1}d^{m + n - 1}}.
    \end{equation}
    %
    %
    Analogously, we have the following bound for the second integral:
    \begin{equation}  \label{eq:integral-bound-2nd}
    \begin{aligned}
        & \left|
            C(n_0; A, \mathbf{m}^{(1)}; B, \mathbf{n})
        \right|
        \le \binom{m + n + n_0 - 3}{m - 2}\frac{1}{a^{n_0}d^{m + n - 2}}
        \le \binom{m + n + n_0 - 3}{m - 2}\frac{1}{a^{n_0 - 1}d^{m + n - 1}}.
    \end{aligned}
    \end{equation}
    Notice that the binomial coefficients in Eqs. \eqref{eq:integral-bound-1st} and \eqref{eq:integral-bound-2nd} sum to the binomial coefficient in Eq. \eqref{eq:integral-bound-weak}, we get $P_1(N)$,
    which proves Eq. \eqref{eq:integral-bound-weak} by induction.
    
    Now we can prove Eq. \eqref{eq:integral-bound}.
    The logic is almost identical, with Eq. \eqref{eq:integral-bound-weak} playing the role of Eq. \eqref{eq:integral-bound-weakest} in its own proof, handling an edge case in the inductive step.
    Let $P_2(n)$ be the statement:
    ``For any sets $A$ and $B$, and the notations and conditions described in Lemma \ref{lem:integral-bound}, such that $m + n + n_0 = N$, Eq. \eqref{eq:integral-bound} holds.''

    The cases $N = 1$ and $N = 2$ are again trivially true.
    Consider $N \ge 3$ and assume $P_2(N - 1)$.
    Fix any sequence $m'_1, m'_2, \ldots, m'_l$ satisfying $0\le m'_i \le m_i$ for each $i\in [k]$ and $n_0 \ge m'_1 + \ldots + m'_k$.
    If $m'_1 = m'_2 = \ldots = m'_k = 0$, we are done by Eq. \eqref{eq:integral-bound-weak}.
    By symmetry among the indices, assume $m'_1 \ge 1$.
    This also means $n_0 \ge 1$.
    Consider Eq. \eqref{eq:integral-bound-split} again.
    For the first integral on the right-hand side, applying $P_2(N - 1)$ for the parameters $n_0 - 1, n_1, \ldots, n_k$, $m_1, \ldots, m_l$ and $m'_1 - 1, m'_2, \ldots, m'_k$ yields the bound
    \begin{equation}  \label{eq:integral-bound-1st-strong}
        \left|
            C(n_0 - 1; A, \mathbf{m}; B, \mathbf{n})
        \right|
        \le \binom{m + n + n_0 - 3}{m - 1}
        \frac{1}{a^{n_0 - m'}d^{m + n - 1}}
        \frac{1}{|a_1|^{m_1' - 1}}
        \prod_{i = 2}^l \frac{1}{|a_i|^{m'_i}}.
    \end{equation}
    Applying $P_2(N - 1)$ for the parameters $n_0, n_1, \ldots, n_k$, $m_1 - 1, \ldots, m_l$ and $m'_1 - 1, m'_2, \ldots, m'_k$,  we get the following bound for the second integral on the right-hand side of Eq. \eqref{eq:integral-bound-split}:
    \begin{equation*}
    \begin{aligned}
        \left|
            C(n_0; A, \mathbf{m}^{(1)}; B, \mathbf{n})
        \right|
        & \le \binom{m + n + n_0 - 3}{m - 2}\frac{1}{a^{n_0 - m' + 1}d^{m + n - 2}} \frac{1}{|a_1|^{m'_1 - 1}} \prod_{i = 2}^l \frac{1}{|a_i|^{m'_i}}
        \\
        & \le \binom{m + n + n_0 - 3}{m - 2}\frac{1}{a^{n_0 - m'}d^{m + n - 1}} \frac{1}{|a_1|^{m'_1 - 1}} \prod_{i = 2}^l \frac{1}{|a_i|^{m'_i}}.
    \end{aligned}
    \end{equation*}
    Summing up the bounds by summing the binomial coefficients, we get exactly $P_2(N)$, so Eq. \eqref{eq:integral-bound} is proven by induction.
\end{proof}

\begin{lemma}  \label{lem:integral-bound-1}
Let $A$, $B$, $\mathbf{m}$, $\mathbf{n}$, $n_0$, $\Gamma_A$ and $a$, $d$ be the same, with the same conditions as in Lemma \ref{lem:integral-bound}.
Suppose that 
$0 \le m'_i \le m_i$ and $0 \le n'_j \le n_j$ for each $i, j \ge 1$ and
\begin{equation*}
    m' + n' \le n_0
    \ \text{ for } \ m' \defeq \sum_i m'_i, \quad n' \defeq \sum_j n'_j.
\end{equation*}
Then for $C(n_0; A, \mathbf{m}; B, \mathbf{n})$ defined in Eq. \eqref{eq:integral-coef-regroup-2}, we have
\begin{equation}  \label{eq:integral-bound-1}
    \left|
        C(n_0; A, \mathbf{m}; B, \mathbf{n})
    \right|
    \le \binom{n + n_0 - n' + m - 2}{m - 1}
    \frac{\left(1 + d/a\right)^{n'}}{a^{n_0 - m' - n'}d^{m + n - 1}}
    \prod_{i = 1}^l \frac{1}{|a_i|^{m'_i}}
    \prod_{j = 1}^k \frac{1}{|b_j|^{n'_j}}.
\end{equation}
\end{lemma}

\begin{proof}
    We have the expansion
    \begin{equation*} 
    \begin{aligned}
        & \frac{1}{z^{n_0}} \prod_{j = 1}^k \frac{b_j^{n'_j}}{(z - b_j)^{n_j}} \prod_{i = 1}^l \frac{1}{(z - a_i)^{m_i}}
        = \frac{1}{z^{n_0 - n'}} \prod_{j = 1}^k \frac{1}{(z - b_j)^{n_j - n'_j}} \prod_{j = 1}^k \left(\frac{1}{z} - \frac{1}{z - b_j} \right)^{n'_j} \prod_{i = 1}^l \frac{1}{(z - a_i)^{m_i}}
        \\
        & = \frac{1}{z^{n_0 - n'}} \prod_{j = 1}^k \frac{1}{(z - b_j)^{n_j - n'_j}}
        \sum_{0\le r_j \le n'_j \forall j}
        \frac{(-1)^{r_1 + \ldots + r_k}}{z^{n' - r_1 - \ldots - r_k}} \prod_{j = 1}^k \binom{n'_j}{r_j}\frac{1}{(z - b_j)^{r_j}} \prod_{i = 1}^l \frac{1}{(z - a_i)^{m_i}}
        \\
        & = \sum_{0\le r_j \le n'_j \forall j} \frac{(-1)^{r_1 + \ldots + r_k}}{z^{n_0 - r_1 - \ldots - r_k}}
        \prod_{j = 1}^k \binom{n'_j}{r_j} \frac{1}{(z - b_j)^{r_j + n_j - n'_j}}
        \prod_{i = 1}^l \frac{1}{(z - a_i)^{m_i}}.
    \end{aligned}
    \end{equation*}
    Integrating both sides over $\Gamma_A$, we have
    \begin{equation*}  \label{eq:integral-bound-1-expansion}
        C(n_0; A, \mathbf{m}; B, \mathbf{n}) \prod_{j = 1}^k b_j^{n'_j}
        = \sum_{0\le r_j \le n'_j \forall j}
        (-1)^{\sum_j r_j} \binom{n'_j}{r_j}
        C\left(\textstyle n_0 - \sum_j r_j; A, \mathbf{m}; B, \mathbf{r} + \mathbf{n} - \mathbf{n}'\right),
    \end{equation*}
    where the $j$-entry of $\mathbf{r} + \mathbf{n} - \mathbf{n}'$ is simply $r_j + n_j - n'_j$.
    Applying Lemma \ref{lem:integral-bound} for each summand on the right-hand side and rearranging the powers, we get
    \begin{equation*}
        \Bigl|
            C\left(\textstyle n_0 - \sum_j r_j; A, \mathbf{m}; B, \mathbf{r} + \mathbf{n} - \mathbf{n}'\right)
        \Bigr|
        \le \binom{m + n + n_0 - n' - 2}{m - 1}
        \frac{(a/d)^{\sum_j r_j}}{a^{n_0 - m'}d^{n - n' + m - 1}}
        \prod_{i = 1}^l \frac{1}{|a_i|^{m'_i}}.
    \end{equation*}
    Summing up the bounds, we get
    \begin{equation*}
    \begin{aligned}
        \left|
            C(n_0; A, \mathbf{m}; B, \mathbf{n}) \prod_{j = 1}^k b_j^{n'_j}
        \right|
        & \le \binom{m + n + n_0 - n' - 2}{m - 1}
        \frac{\textstyle \prod_{i = 1}^l |a_i|^{-m'_i}}{a^{n_0 - m'}d^{n - n' + m - 1}}
        \sum_{0\le r_j\le n'_j \forall j}
        \prod_{j = 1}^k \binom{n'_j}{r_j}\frac{a^{r_j}}{d^{r_j}}
        \\
        & = \binom{m + n + n_0 - n' - 2}{m - 1}
        \frac{\textstyle \prod_{i = 1}^l |a_i|^{-m'_i}}{a^{n_0 - m'}d^{n - n' + m - 1}}
        \left(\frac{a}{d} + 1\right)^{n'}.
    \end{aligned}
    \end{equation*}
    Rearranging the term, we get precisely the desired inequality.
\end{proof}

The lemma above is the main ingredient in the proof of Lemma \ref{lem:integral-coef-bound}.

\begin{proof}[Proof of Eq. \eqref{eq:integral-coef-bound-1} of Lemma \ref{lem:integral-coef-bound}]
    First rewrite the integral into the forms of \eqref{eq:integral-coef-generic}, then \eqref{eq:integral-coef-regroup-1} and \eqref{eq:integral-coef-regroup-2}.
    Let us consider two cases for $\myC$:
    \begin{enumerate}
        \item $\nu = 0$, so $n_0 = \gamma + 1$.
        Let $a = \lambda_S(\mathbf{I})$, $d = \delta_S(\mathbf{I})$, $m = \beta_S(\mathbf{I})$, $n = n' = \beta_{S^c}(\mathbf{I})$, $m'_i = m_i$ and $n'_j = n_j$ for all $i, j$,
        then $m' + n' = \beta \le \gamma + 1 = n_0$, so we can apply Lemma \ref{lem:integral-bound-1} to get
        \begin{equation*}
            \left|
            C(n_0; A, \mathbf{m}; B, \mathbf{n})
            \right|
            \le \binom{n_0 + m - 2}{m - 1}
            \frac{(1 + d/a)^{n'}}{a^{n_0 - m - n}d^{m + n - 1}}
            \prod_{i = 1}^l \frac{1}{|a_i|^{m_i}}
            \prod_{j = 1}^k \frac{1}{|b_j|^{n_j}},
        \end{equation*}
        or equivalently,
        \begin{equation*}
            |\myC_0(\mathbf{I})|
            \le \left(
                1 + \frac{\Delta_S(\mathbf{I})}{\lambda_S(\mathbf{I})}
            \right)^{\beta_{S^c}(\mathbf{I})}
            \binom{
                \gamma + \beta_S(\mathbf{I}) - 1
            }{
                \beta_S(\mathbf{I}) - 1
            }
            \frac{1}{
                \lambda_S(\mathbf{I})^{\gamma + 1 - \beta}
                \Delta_S(\mathbf{I})^{\beta - 1}
            }.
        \end{equation*}
        Note that since $\beta_S(\mathbf{I}) \le \beta \le \gamma + 1$ and $\myL_0 = 2$,
        \begin{equation*}
            \binom{
                \gamma + \beta_S(\mathbf{I}) - 1
            }{
                \beta_S(\mathbf{I}) - 1
            }
            = \frac{
                \gamma + \beta_S(\mathbf{I}) - 1
            }{\gamma} \binom{
                \gamma + \beta_S(\mathbf{I}) - 2
            }{
                \beta_S(\mathbf{I}) - 1
            }
            \le \myL_0\binom{
                \gamma + \beta_S(\mathbf{I}) - 2
            }{
                \beta_S(\mathbf{I}) - 1
            },
        \end{equation*}
        so we can replace the former with the latter to the product on the right-hand side to get an upper bound, which is also the desired bound.

        \item $\nu = 1$, so $n_0 = \gamma$.
        Without loss of generality, assume $|a_1| = \lambda_S(\mathbf{I})$, then we are guaranteed $m_1 \ge 1$.
        Applying Lemma \ref{lem:integral-bound-1} for the same parameters as in the previous case, except that $m'_1 = m_1 - 1$, we get
        \begin{equation*}
            \left|
            C(n_0; A, \mathbf{m}; B, \mathbf{n})
            \right|
            \le |a_1|\binom{n_0 + m - 2}{m - 1}
            \frac{(1 + d/a)^{n'}}{a^{n_0 - m + 1 - n}d^{m + n - 1}}
            \prod_{i = 1}^l \frac{1}{|a_i|^{m_i}}
            \prod_{j = 1}^k \frac{1}{|b_j|^{n_j}},
        \end{equation*}
        which translates to
        \begin{equation*}
            \left|
                \myC_1(\mathbf{I})
            \right|
            \le \lambda_S(\mathbf{I})
            \binom{\gamma + \beta_S(\mathbf{I}) - 2}{\beta_S(\mathbf{I}) - 1}
            \left(
                1 + \frac{\Delta_S(\mathbf{I})}{\lambda_S(\mathbf{I})}
            \right)^{\beta_{S^c}(\mathbf{I})}
            \frac{1}{
                \lambda_S(\mathbf{I})^{\gamma + 1 - \beta}
                \Delta_S(\mathbf{I})^{\beta - 1}
            },
        \end{equation*}
        which is the desired bound since $\myL_1 = \lambda_S(\mathbf{I})$.
    \end{enumerate}

    Let us now prove Eq. \eqref{eq:integral-coef-bound-2}.
    We can assume $\beta_S(\mathbf{I}) \ge 1$, since the integral is $0$ otherwise, making the inequality trivial.
    It suffices to show that we can substitute $\lambda_S(\mathbf{I})$ with $\lambda_S$ and $\Delta_S(\mathbf{I})$ with $\Delta_S$ in Eq. \eqref{eq:integral-coef-bound-1} to make the right-hand side larger.
    Let us again split into the cases as above.

    \begin{enumerate}
        \item $\nu = 0$.
        We have
        \begin{equation*}
            \frac{
                \left(
                    1 + \frac{\Delta_S(\mathbf{I})}{\lambda_S(\mathbf{I})}
                \right)^{\beta_{S^c}(\mathbf{I})}
            }{
                \lambda_S(\mathbf{I})^{\gamma + 1 - \beta}
                \Delta_S(\mathbf{I})^{\beta - 1}
            }
            = 
            \frac{
                \left(
                   \frac{1}{\Delta_S(\mathbf{I})}  + \frac{1}{\lambda_S(\mathbf{I})}
                \right)^{\beta_{S^c}(\mathbf{I})}
            }{
                \lambda_S(\mathbf{I})^{\gamma + 1 - \beta}
                \Delta_S(\mathbf{I})^{\beta_S(\mathbf{I}) - 1}
            }.
        \end{equation*}
        From this new form, it is evident that the right-hand side will increase if we make the aforementioned substitutions, since $\lambda_S \le \lambda_S(\mathbf{I})$ and $\Delta_S \le \Delta_S(\mathbf{I})$.

        \item $\nu = 1$ and additionally, $S = [s]$ for some $s\in [r]$.
        If $\beta \le \gamma$, the rewriting in the previous case works in the same way.
        Suppose $\beta = \gamma + 1$, we now write
        \begin{equation*}
            \frac{
                \lambda_S(\mathbf{I})
                \left(
                    1 + \frac{\Delta_S(\mathbf{I})}{\lambda_S(\mathbf{I})}
                \right)^{\beta_{S^c}(\mathbf{I})}
            }{
                \lambda_S(\mathbf{I})^{\gamma + 1 - \beta}
                \Delta_S(\mathbf{I})^{\beta - 1}
            }
            = \frac{1}{\lambda_S(\mathbf{I})^{\gamma - 1}}
            \left(
                \frac{\lambda_S(\mathbf{I})}{\Delta_S(\mathbf{I})}
            \right)^{\beta_S(\mathbf{I}) - 1}
            \left(
                1 + \frac{\lambda_S(\mathbf{I})}{\Delta_S(\mathbf{I})}
            \right)^{\beta_{S^c}(\mathbf{I})}
            .
        \end{equation*}
        Since $\lambda_S(\mathbf{I}) \le \lambda_S = \lambda_s$, it suffices to show $\lambda_S(\mathbf{I}) / \Delta_S(\mathbf{I}) \le \lambda_s / \delta_s$ to make the substitution work as in the previous case.
        Choose $t\in [s]$ where $\lambda_t = \lambda_S(\mathbf{I})$,
        then $\Delta_S(\mathbf{I}) \ge \lambda_t - \lambda_{s + 1}$, thus
        \begin{equation*}
            \frac{\lambda_S(\mathbf{I})}{\Delta_S(\mathbf{I})}
            \le \frac{\lambda_t}{\lambda_t - \lambda_{s + 1}}
            \le \frac{\lambda_s}{\lambda_s - \lambda_{s + 1}}
            = \frac{\lambda_s}{\delta_s}.
        \end{equation*}
    \end{enumerate}
    Thus both Eqs. \eqref{eq:integral-coef-bound-1} and \eqref{eq:integral-coef-bound-2} hold.
    The proof is complete.
\end{proof}

\subsection{Proof of semi-isotropic bounds for powers of random matrices}  \label{sec:E-power-bound-proof}

In this section, we prove Lemma \ref{lem:E-power-union-bound}, which gives semi-itrosopic bounds for powers of $\sym{E}$ in the second step of the main proof strategy.

The form of the bounds naturally implies that we should handle the even and odd powers separately.
We split the two cases into the following lemmas.

\begin{lemma}  \label{lem:E-power-bound-odd}
    Let $m, r\in \N$ and $U\in \R^{m\times r}$ be a matrix whose columns $u_1, u_2, \ldots, u_r$ are unit vectors.
    Let $E$ be a $m\times n$ random matrix following Model \eqref{eq:noise-model} with parameters $\upperE$ and $\stdE = 1$, meaning $E$ has independent entries and
    \begin{equation*}
        \E{E_{ij}} = 0,
        \quad \E{\|E\|_{ij}^2} \le 1,
        \quad \E{\|E\|_{ij}^p} \le \upperE^{p - 2} \quad \text{ for all } \ p.
    \end{equation*}
    For any $a\in \N$, $k \in [n]$, for any $D > 0$, for any $p \in \N$ such that
    \begin{equation*}
        m + n \ge 2^8\upperE^2p^6(2a + 1)^4
        ,    
    \end{equation*}
    we have, with probability at least $1 - (2^5/D)^{2p}$,
    \begin{equation*}
        \left\|e_{n, k}^T(E^TE)^aE^TU\right\|
        \le
        Dr^{1/2}p^{3/2}\sqrt{2a + 1}
        \left(
            16p^{3/2}(2a + 1)^{3/2}\upperE\frac{\|U\|_{2, \infty}}{\sqrt{r}} + 1
        \right)
        [2(m + n)]^a.
    \end{equation*}
\end{lemma}

\begin{lemma}  \label{lem:E-power-bound-even}
    Let $E$ be a $m\times n$ random matrix following the model in Lemma \ref{lem:E-power-bound-odd}.
    For any matrix $V\in \R^{m\times l}$ with unit columns $v_1, v_2, \ldots, v_l$, any $a\in \N$, $k \in [n]$, any $D > 0$, and any $p \in \N$ such that
    \begin{equation*}
        m + n \ge 2^8\upperE^2p^6(2a)^4,    
    \end{equation*}
    we have, with probability at least $1 - (2^4/D)^{2p}$,
    \begin{equation*}
        \left\|e_{n, k}^T(E^TE)^a V \right\|
        \le Dp \|V\|_{2, \infty} [2(m + n)]^a.
    \end{equation*}
\end{lemma}

Let us prove the main objective of this section, Lemma \ref{lem:E-power-union-bound}, before delving into the proof of the technical lemmas.

\begin{proof}[Proof of Lemma \ref{lem:E-power-union-bound}]
    We follow two steps:
    \begin{enumerate}
        \item \textit{Assuming $M$}
    \end{enumerate}
    Consider the analogue of Eq. \eqref{eq:E-power-bound-myD0(U)-1} for $V$ (we wrote the proof for $V$ before the final edit, and wanted to save the energy of changing to $U$) and Eq. \eqref{eq:E-power-bound-myD1(U)-1}, and assume $\upperE \le \log^{-2-\eps}(m + n)\sqrt{m + n}$.
    Fix $k\in [n]$.
    It suffices to prove the following two bounds uniformly over all $a\in [\lfloor t\log (m + n)\rfloor]$:
    \begin{align}
        \label{eq:E-power-bound-proof-odd}
        \left\|e_{n, k}^T(E^TE)^aE^T U\right\|
        & \le C\myD_1(U, \log\log(m + n)) (1.9\stdE\sqrt{m + n})^{2a + 1}\sqrt{r}
        \\
        \label{eq:E-power-bound-proof-even}
        \left\|e_{n, k}^T(E^TE)^a V\right\|
        & \le C\myD_0(V, \log\log(m + n)) (1.9\stdE\sqrt{m + n})^{2a}\sqrt{r}.
    \end{align}

    Fix $a\in [\lfloor t\log (m + n)\rfloor]$.
    Let $p = \log\log (m + n)$.
    We can assume $p$ is an integer for simplicity without any loss.
    This choice ensures
    \begin{equation*}
        \upperE^2p^6(2a)^4 < \upperE^2p^6(2a + 1)^4 \le \frac{(m + n)t^4\log^4(m + n)\log^6\log (m + n)}{\log^{4 + 2\eps}(m + n)} = o(m + n),
    \end{equation*}
    so we can apply both Lemmas \ref{lem:E-power-bound-odd} and \ref{lem:E-power-bound-even}.
    
    Let us prove Eq. \eqref{eq:E-power-bound-proof-odd} for $a$.
    Applying Lemma \ref{lem:E-power-bound-odd} for the random matrix $E/\stdE$ and $D = 2^{13}$ gives, with probability $1 - \log^{-4.04}(m + n)$,
    \begin{equation*}
    \begin{aligned}
        & \frac{\|e_{n, k}^T(E^TE)^aE^TU\|}{(1.9\stdE\sqrt{m + n})^{2a + 1}}
        \le
        \frac{Dr^{1/2}p^{3/2}\sqrt{2a + 1}}{1.9\sqrt{m + n}}
        \left(
            16p^{3/2}(2a + 1)^{3/2}\upperE\frac{\|U\|_{2, \infty}}{\sqrt{r}} + 1
        \right) \left(
            \frac{2}{3.61}
        \right)^a
        \\
        & \quad \quad
        \le
        \frac{Dr^{1/2}p^{3/2}}{\sqrt{m + n}}
        \left(
            16p^{3/2}\upperE\frac{\|U\|_{2, \infty}}{\sqrt{r}} + 1
        \right)
        \le 2^{17}\sqrt{r} \left(
            \frac{\upperE p^3\|U\|_{2, \infty}}{\sqrt{r(m + n)}}
            + \frac{p^{3/2}}{\sqrt{m + n}}
        \right),
    \end{aligned}
    \end{equation*}
    where the second inequality is due to $\alpha \le (\sqrt{2}/1.9)^\alpha$.
    A union bound over all $a \in [\lfloor t\log(m + n)\rfloor]$ makes the bound uniform, with probability at least $1 - \log^{-3}(m + n)$.
    The term inside parentheses in the last expression is less than $D_{U, V, \log\log (m + n)}$, so Eq. \eqref{eq:E-power-bound-proof-odd}, and thus Eq. \eqref{eq:E-power-bound-myD1(U)-1} follows.
    
    Let us prove Eq. \eqref{eq:E-power-bound-proof-even}.
    Applying Lemma \ref{lem:E-power-bound-odd} for the random matrix $E/\stdE$ and $D = 2^{10}$ gives, with probability $1 - \log^{-8}(m + n)$,
    \begin{equation*}
        \frac{\|e_{n, k}^T(E^TE)^aV\|}{(1.9\stdE\sqrt{m + n})^{2a + 1}}
        \le
        Dp\|V\|_{2, \infty}
        \left(
            \frac{2}{3.61}
        \right)^a
        \le 2^{10}p\|V\|_{2, \infty}
        \le 2^{10}\sqrt{r}D_{U, V, p},
    \end{equation*}
    proving Eq. \eqref{eq:E-power-bound-proof-even} and thus Eq. \eqref{eq:E-power-bound-myD0(U)-1} after a union bound, similar to the previous case.

    Let us now prove Eqs \eqref{eq:E-power-bound-myD0(U)-2} and \eqref{eq:E-power-bound-myD1(U)-2}, focusing on the former first.
    Since the 2-to-$\infty$ norm is the the largest norm among the rows, it suffices to prove Eq. \eqref{eq:E-power-bound-myD0(U)-1} holds uniformly over all $k\in [n]$ for $p = \log(m + n)$.
    Substituting this new choice of $p$ into the previous argument, for a fixed $k$, we have Eq. \eqref{eq:E-power-bound-myD0(U)-1}, but with probability at least $1 - (m + n)^{-4.04}$.
    Applying another union bound over $k\le [n]$ gives Eq. \eqref{eq:E-power-bound-myD0(U)-2} with probability at least $1 - (m + n)^{-3}$.
    The proof of \eqref{eq:E-power-bound-myD1(U)-2} is analogous.
    The proof of Lemma \ref{lem:E-power-union-bound} is complete.
\end{proof}

Now let us handle the technical lemmas \ref{lem:E-power-bound-odd} and \ref{lem:E-power-bound-even}.
The odd case (Lemma \ref{lem:E-power-bound-odd}) is more difficult, so we will handle it first to demonstrate our technique.
The argument for the even case (Lemma \ref{lem:E-power-bound-even}) is just a simpler version of the same technique.

\subsubsection{Case 1: odd powers}

\begin{proof}
    Without loss of generality, let $k = 1$.
    Let us fix $p\in \N$ and bound the $(2p)^{th}$ moment of the expression of concern.
    We have
    \begin{equation}
    \begin{aligned}
        \E{
            \bigl\| e_{n, 1}^T(E^TE)^aE^TU \bigr\|^{2p}
        }
        & = \E{
            \Bigl( \sum_{l = 1}^r (e_{n, 1}^T(E^TE)^aE^Tu_l)^2 \Bigr)^p
        }
        \\
        & = \sum_{l_1, \ldots, l_p \in [r]}
        \E{
            \prod_{h = 1}^p (e_{n, 1}^T(E^TE)^aE^Tu_{l_h})^2
        }.
    \end{aligned}
    \end{equation}
    Temporarily let $\mathcal{W}$ be the set of walks $W = (j_0i_0j_1i_1\ldots i_a)$ of length $2a + 1$ on the complete bipartite graph $\upperE_{m, n}$ such that $j_0 = 1$.
    Here the two parts of $\upperE$ are $I = \{1', 2', \ldots, m'\}$ and $J = \{1, 2, \ldots, n\}$, where the prime symbol serves to distinguish two vertices on different parts with the same number.
    Let $E_W = E_{i_0j_0}E_{i_0j_1}\ldots E_{i_{a - 1}j_a}E_{i_aj_a}$.
    We can rewrite the final expression in the above as
    \begin{equation*}
        \sum_{l_1, l_2, \ldots, l_p \in [r]} \
        \sum_{W_{11}, W_{12}, W_{21}, \ldots, W_{p2} \in \mathcal{W}}
        \E{
            \prod_{h = 1}^p E_{W_{h1}}E_{W_{h2}} u_{l_hi_{(h1)a}}u_{l_hi_{(h2)a}}
        },
    \end{equation*}
    where we denote $W_{hd} = (j_{(hd)0}, i_{(hd)0}, \ldots, i_{(hd)a})$.
    We can swap the two summation in the above to get
    \begin{equation*}
        \sum_{W_{11}, W_{12}, W_{21}, \ldots, W_{p2} \in \mathcal{W}}
        \E{
            \prod_{h = 1}^p E_{W_{h1}}E_{W_{h2}}
        }
        \sum_{l_1, l_2, \ldots, l_p \in [r]}
        \prod_{h = 1}^p u_{l_hi_{(h1)a}}u_{l_hi_{(h2)a}}.
    \end{equation*}
    The second sum can be recollected in the form of a product, so we can rewrite the above as
    \begin{equation*}
        \sum_{W_{11}, W_{12}, W_{21}, \ldots, W_{p2} \in \mathcal{W}}
        \E{
            \prod_{h = 1}^p E_{W_{h1}}E_{W_{h2}}
        }
        \prod_{h = 1}^p U_{\cdot, i_{(h1)a}}^T U_{\cdot, i_{(h2)a}}
    \end{equation*}
    Define the following notation:
    \begin{enumerate}
        \item $\mathcal{P}$ is the set of all \emph{star}, i.e. tuples of walks $P = (P_1, \ldots, P_{2p})$ on the complete bipartite graph $\upperE_{m, n}$, such that each walk $P_r \in \mathcal{W}$ and each edge appears at least twice.
        
        Rename each tuple $(W_{h1}, W_{h2})_{h = 1}^p$ as a star $P$ with $W_{hd} = P_{2h - 2 + d}$.
        
        For each $P$, let $V(P)$ and $E(P)$ respectively be the set of vertices and edges involved in $P$.
        
        Define the partition $V(P) = V_I(P) \cup V_J(P)$, where $V_I(P) \defeq V(P)\cap I$ and $V_J(P) \defeq V(P)\cap J$.

        \item $E_P \defeq E_{P_1}E_{P_2}\ldots E_{P_{2p}}$.

        \item $\bound{P} \defeq (i_{1a}, i_{2a}, \ldots, i_{(2p)a})$, which we call the \emph{boundary} of $P$.
        Then $u_Q \defeq \prod_{r = 1}^{2p} u_{q_r}$ for any tuple $Q = (q_1, \ldots, q_r)$.

        \item $\mathcal{S}$ is the subset of ``shapes'' in $\mathcal{P}$.
        A shape is a tuple of walks $S = (S_1, \ldots, S_{2p})$ such that all $S_r$ start with $1$ and for all $r\in [2p]$ and $s\in [0, a]$, if $i_{rs}$ appears for the first time in $\{i_{r's'}: r'\le r, s'\le s\}$, then it is stricly larger than all indices before it, and similarly for $j_{rs}$.
        We say a star $P \in \mathcal{P}$ has shape $S\in \mathcal{S}$ if there is a bijection from $V(P)$ to $[|V(P)|]$ that transforms $P$ into $S$.
        The notations $V(S)$, $V_I(S)$, $V_J(S)$, $E(S)$ are defined analogously.
        Observe that the shape of $P$ is unique, and $\mathcal{S}$ forms a set of equivalent classes on $\mathcal{P}$.

        \item Denote by $\mathcal{P}(S)$ the class associated with the shape $S$, namely the set of all stars $P$ having shape $S$.
    \end{enumerate}

    We can rewrite the previous sum as:
    \begin{equation*}
        \sum_{P\in \mathcal{P}}
        \E{E_P} \prod_{h = 1}^p U_{\cdot, i_{(2h - 1)a}}^T U_{\cdot, i_{(2h)a}}
    \end{equation*}
    Using triangle inequality and the sub-multiplicity of the operator norm, we get the following upper bound for the above:
    \begin{equation}  \label{eq:E-power-bound-odd-temp1}
        \sum_{P\in \mathcal{P}}
        |\E{E_P}| \prod_{h = 1}^p \|U_{\cdot, i_{(2h - 1)a}}\| \|U_{\cdot, i_{(2h)a}}\|
        = r^p \sum_{P\in \mathcal{P}}
        u_{\bound{P}} |\E{E_P}|
        = r^p \sum_{S\in \mathcal{S}}
        \sum_{P\in \mathcal{P}(S)} u_{\bound{P}} |\E{E_P}|,
    \end{equation}
    where the vector $u$ is given by $u_i = r^{-1/2}\|U_{\cdot, i}\|$ for $i\in [m]$.
    Observe that
    \begin{equation*}
        \|u\| = 1 \ \text{ and } \ \|u\|_\infty = r^{-1/2}\|U\|_{2, \infty}.
    \end{equation*}
    Fix $P\in \mathcal{P}$.
    Let us bound $\E{E_P}$.
    For each $(i, j)\in E(P)$, let $\mu_P(i, j)$ be the number of times $(i, j)$ is traversed in $P$.
    We have
    \begin{equation*}
        \left|\E{E_P}\right| = \prod_{(i, j)\in E(P)} \E{\left|E_{ij}\right|^{\mu_P(i, j)}}
        \le \prod_{(i, j)\in E(P)} \upperE^{\mu_P(i, j) - 2}
        = \upperE^{2p(2a + 1) - 2|E(P)|}.
    \end{equation*}
    Since the entries $u_i$ are related by the fact their squares sum to $1$, it will be better to bound their symmetric sums rather than just a product $u_{\bound{P}}$.
    Fix a shape $S$,
    we have
    \begin{equation*}
    \begin{aligned}
        \sum_{P\in \mathcal{P}(S)}\left|u_{\bound{P}}\right|
        & = \sum_{f: V(S)\hookrightarrow [m]} \prod_{k = 1}^{|V(\bound{S})|} |u_{f(k)}|^{\mu_{\bound{S}}(k)}
        \le m^{|V_I(S)| - |V(\bound{S})|} n^{|V_J(S)| - 1} \prod_{k = 1}^{|V(\bound{S})|} \sum_{i = 1}^m |u_i|^{\mu_{\bound{S}}(k)}
        \\
        & = m^{|V_I(S)| - |V(\bound{S})|} n^{|V_J(S)| - 1} \prod_{k = 1}^{|V(\bound{S})|} \|u\|_{\mu_{\bound{S}}(k)}^{\mu_{\bound{S}}(k)},
    \end{aligned}
    \end{equation*}
    where we slightly abuse notation by letting $\mu_Q(k)$ be the number of time $k$ appears in $Q$.
    
    Consider $\|u\|_l^l$ for an arbitrary $l\in \N$.
    When $l = 1$, $\|u\|_l^l \le \sqrt{m}$ by Cauchy-Schwarz.
    When $l \ge 2$, we have $\|u\|_l^l \le \|u\|_\infty^{l - 2}\|u\|_2^2 = \|u\|_\infty^{l - 2}$.
    Thus
    \begin{equation*}
        \sum_{P\in \mathcal{P}(S)}\left|u_{\bound{P}}\right|
        \le \prod_{k = 1}^{|V(S)|} \|u\|_{\mu_{\bound{S}}(k)}^{\mu_{\bound{S}}(k)}
        \le \prod_{k\in V_2(S)}\|u\|_\infty^{\mu_{\bound{S}}(k) - 2} (\sqrt{m})^{|V_1(\bound{S})|}
        = \|u\|_\infty^{2p - \nu(S)} m^{|V_1(\bound{S})|/2},
    \end{equation*}
    where, we define $V_1(Q)$ as the set of vertices appearing in $Q$ exactly once and $V_2(Q)$ as the set of vertices appearing at least twice, and to shorten the notation, we let $\nu(S) \defeq |V_1(\bound{S})| + 2|V_2(\bound{S})|$.
    Combining the bounds, we get the upper bound below for \eqref{eq:E-power-bound-odd-temp1}:
    \begin{equation*}
    \begin{aligned}
        & \upperE^{2p(2a + 1)} \sum_{S\in \mathcal{S}}
        \upperE^{-2|E(S)|}
        m^{|V_I(S)| - |V(\bound{S})|} n^{|V_J(S)| - 1}
        \|u\|_\infty^{2p - \nu(S)} m^{|V_1(\bound{S})|/2}
        \\
        & = \upperE^{2p(2a + 1) + 2} \sum_{S\in \mathcal{S}} \upperE^{-2|V(S)|} m^{|V_I(S)| - \nu(S)/2} n^{|V_J(S)| - 1} \|u\|_\infty^{2p - \nu(S)}.
    \end{aligned}
    \end{equation*}
    Suppose we fix $|V_1(\bound{S})| = x$, $|V_2(\bound{S})| = y$, $|V_I(S)| = z$, $|V_J(S)| = t$.
    Let $\mathcal{S}(x, y, z, t)$ be the subset of shapes having these quantities.
    To further shorten the notation, let $\upperE_1 \defeq \upperE^{2p(2a + 1)}\|u\|_\infty^{2p}$.
    Then we can rewrite the above as:
    \begin{equation}  \label{eq:main-proof-temp4}
        \upperE_1 \sum_{x, y, z, t\in \mathcal{A}} \upperE^{-2(z+t)}m^{z - x/2 - y}n^{t - 1}\|u\|_\infty^{-x - 2y} |\mathcal{S}(x, y, z, t)|,
    \end{equation}
    where $\mathcal{A}$ is defined, somewhat abstractly, as the set of all tuples $(x, y, z, t)$ such that $\mathcal{S}(x, y, z, t) \neq \varnothing$.
    We first derive some basic conditions for such tuples.
    Trivially, one has the following initial bounds:
    \begin{equation*}
        0\le x, y,
        \quad \quad 1\le x + y \le z,
        \quad \quad x + 2y \le 2p,
        \quad \quad 0\le z, t,
        \quad \quad z + t\le p(2a + 1) + 1,
    \end{equation*}
    where the last bound is due to $z + t = |V(S)| \le |E(S)| + 1 \le p(2a + 1) + 1$, since each edge is repeated at least twice.
    However, it is not strong enough, since we want the highest power of $m$ and $n$ combined to be at most $2ap$, so we need to eliminate a quantity of $p$.
    
    \begin{claim}  \label{claim:shape-vertices-bound}
        When each edge is repeated at least twice, we have $z - x/2 - y + t - 1 \le 2ap$.
    \end{claim}

    \begin{proof}[Proof of Claim \ref{claim:shape-vertices-bound}]
        Let $S = (S_1, \ldots, S_{2p})$, where $S_r = j_{r0}i_{r0}j_{r1}i_{r1}\ldots j_{ra}i_{ra}$.
        We have $j_{r0} = 1$ for all $r$.
        It is tempting to think (falsely) that when each edge is repeated at least twice, each vertex appears at least twice too.
        If this were to be the case, then each vertex in the set
        \begin{equation*}
            A(S) \defeq \{i_{rs}: 1\le r\le 2p, 0\le s\le a - 1\}
            \cup \{j_{rs}: 1\le r\le 2p, 1\le s\le a\}
            \cup V_1(\bound{S})
        \end{equation*}
        appears at least twice.
        The sum of their repetitions is $4ap + x$, so the size of this set is at most $2ap + x/2$.
        Since this set covers every vertex, with the possible exceptions of $1\in I$ and $V_2(\bound{S})$, its size is at least $z - y + t - 1$, proving the claim.
        In general, there will be vertices appearing only once in $S$.
        However, we can still use the simple idea above.
        Temporarily let $A_1(S)$ be the set of vertices appearing once in $S$ and $f(S)$ be the sum of all edges' repetitions in $S$.
        Let $S^{(0)} \defeq S$.
        Suppose for $k \ge 0$, $S^{(k)}$ is known and satisfies $|A(S^{(k)})| = |A(S)| - k$, $f(S^{(k)}) = 4pa + x - 2k$ and each edge appears at least twice in $S^{(k)}$.
        If $A_1(S^{(k)}) = \varnothing$, then by the previous argument, we have
        \begin{equation*}
            2(z - y + t - 1 - k) \le 4pa + x - 2k \implies z - x/2 - y + t - 1 \le 2pa,
        \end{equation*}
        proving the claim.
        If there is some vertex in $A_1(S^{(k)})$, assume it is some $i_{rs}$, then we must have $s\le a - 1$ and $j_{rs} = j_{r(s + 1)}$, otherwise the edge $j_{rs}i_{rs}$ appears only once.
        Create $S^{(k + 1)}$ from $S^{(k)}$ by removing $i_{rs}$ and identifying $j_{rs}$ and $j_{r(s + 1)}$, we have
        $|A(S^{(k + 1)})| = |A(S)| - (k + 1)$ and $|f(S^{(k)}) = 4pa + x - 2(k + 1)$.
        Further, since $i_{rs}$ is unique, $j_{rs}i_{rs} \equiv i_{rs}j_{r(s + 1)}$ are the only 2 occurences of this edge in $S^{(k)}$, thus the edges remaining in $S^{(k + 1)}$ also appears at least twice.
        Now we only have $|A_1(S^{(k + 1)})| \le |A_1(S^{(k)})|$, with possible equality, since $j_{rs}$ can be come unique after the removal, but since there is only a finite number of edges to remove, eventually we have $A_1(S^{(k)}) = \varnothing$, completing the proof of the claim.
    \end{proof}

    Claim \ref{claim:shape-vertices-bound} shows that we can define the set $\mathcal{A}$ of \emph{eligible sizes} as follows:
    \begin{equation}  \label{eq:eligible-sizes-defn}
        \mathcal{A} = \left\{
            (x, y, z, t)\in \N_{\ge 0}^4: \
            \quad 1\le t;
            \quad 1\le x + y \le z;
            \quad x + 2y \le 2p;
            \quad z - x/2 - y + t - 1\le 2ap
        \right\}.
    \end{equation}
    Now it remains to bound $|\mathcal{S}(x, y, z, t)|$.

    \begin{claim}  \label{claim:num-shapes-bound}
        Given a tuple $(x, y, z, t) \in \mathcal{A}$, where $\mathcal{A}$ is defined in Eq. \eqref{eq:eligible-sizes-defn}, we have
        \begin{equation*}
            |\mathcal{S}(x, y, z, t)| \le
            \frac{2^{l + 1}(2p(a + 1))!(2pa)!(l + 1)^{2p(2a + 1) - 2l}}{(2p(2a + 1) - 2l)!l!z!(t - 1)!}
            (16p(a + 1) - 8l - 2)^{4p(a + 1) - 2l - 1}.
        \end{equation*}
    \end{claim}

    \begin{proof}
        We use the following coding scheme for each shape $S\in \mathcal{S}(x, y, z, t)$:
        Given such an $S$, we can progressively build a codeword $W(S)$ and an associated tree $T(S)$ accoding to the following scheme:
        \begin{enumerate}
            \item Start with $V_J = \{1\}$ and $V_I = \varnothing$, $W = []$ and $T$ being the tree with one vertex, $1$.
            \item For $r = 1, 2, \ldots, 2p$:
            \begin{enumerate}
                \item Relabel $S_r$ as $1k_1k_2\ldots k_{2a}$.
                \item For $s = 1, 2, \ldots, 2a$:
                \begin{itemize}
                    \item If $k_s \notin V(T)$ then add $k_s$ to $T$ and draw an edge connecting $k_{s - 1}$ and $k_s$,
                    then mark that edge with a $(+)$ in $T$, and append $(+)$ to $W$.
                    We call its instance in $S_r$ a \emph{plus edge}.
                    \item If $k_s \in V(T)$ and the edge $k_{s - 1}k_s \in E(T)$ and is marked with $(+)$: unmark it in $T$, and append $(-)$ to $W$.
                    We call its instance in $S_r$ a \emph{minus edge}.
                    \item If $k_s \in V(T)$ but either $k_{s - 1}k_s \notin E(T)$ or is unmarked, we call its instance in $S_r$ a \emph{neutral edge}, and append the symbol $k_s$ to $W$.
                \end{itemize}
            \end{enumerate}
        \end{enumerate}
        
        This scheme only creates a \emph{preliminary codeword} $W$, which does not yet uniquely determine the original $S$.
        To be able to trace back $S$, we need the scheme in \cite{vu2005} to add more details to the preliminary codewords.
        For completeness, we will describe this scheme later, but let us first bound the number of preliminary codewords.

        \begin{claim}  \label{claim:num-prelim-codewords-bound}
            Let $\mathcal{PC}(x, y, z, t)$ denote the set of preliminary codewords generable from shapes in $\mathcal{S}(x, y, z, t)$.
            Then for $l \defeq z + t - 1$ we have
            \begin{equation*}
            |\mathcal{PC}(x, y, z, t)|
            \le \frac{2^{l}(2p(a + 1))!(2pa)!(l + 1)^{2p(2a + 1) - 2l}}{(2p(2a + 1) - 2l)!l!z!(t - 1)!}.
            \end{equation*}
        \end{claim}
        Note that the bound above does not depend on $x$ and $y$.
        In fact, for fixed $z$ and $t$, the right-hand side is actually an upper bound for the sum of $|\mathcal{S}(x, y, z, t)|$ over all pairs $(x, y)$ such that $(x, y, z, t)$ is eligible.
        We believe there is plenty of room to improve this bound in the future.

        \begin{proof}
            To begin, note that there are precisely $z$ and $t - 1$ plus edges whose right endpoint is respectively in $I$ and $J$.
            Suppose we know $u$ and $v$, the number of minus edges whose right endpoint is in $I$ and $J$, respectively.
            Then
            \begin{itemize}
                \item The number of ways to place plus edges is at most $\binom{2p(a + 1)}{z}\binom{2pa}{t - 1}$.
                \item The number of ways to place minus edges, given the position of plus edges, is at most $\binom{2p(a + 1)  - z}{u}\binom{2pa - t + 1}{v}$.
                \item The number of ways to choose the endpoint for each neutral edge is at most $z^{2p(a + 1) - z - u}t^{2pa - t + 1 - v}$.
            \end{itemize}
            Combining the bounds above, we have
            \begin{equation}  \label{eq:num-shapes-bound-expansion1}
                |\mathcal{S}(x, y, z, t)|
                \le \binom{2p(a + 1)}{z}\binom{2pa}{t - 1}
                \sum_{u + v = z + t - 1} \binom{2p(a + 1)  - z}{u}\binom{2pa - t + 1}{v} z^{f(z, u)} t^{g(t, v)},
            \end{equation}
            where $f(z, u) = 2p(a + 1) - z - u$ and $g(u, v) = 2pa - t + 1 - v$.
            Let us simplify this bound.
            The sum on the right-hand side has the form
            \begin{equation*}
                \sum_{i + j = k} \binom{N}{i}\binom{M}{j}z^i t^j,
            \end{equation*}
            where $k = 2(p(2a + 1) - (z + t - 1))$, $N = 2p(a + 1) - z$, $M = 2pa - t + 1$.
            We have
            \begin{equation*}
            \begin{aligned}
                & \sum_{i + j = k} \binom{N}{i} \binom{M}{j} z^i t^j
                = \sum_{i + j = k} \frac{N!M!}{k!(N - i)!(M - j)!} \binom{k}{i} z^i t^j
                \le \sum_{i + j = k} \frac{N!M!}{k!} \frac{(z + t)^k}{(N - i)(M - j)!}
                \\
                & \le \frac{N!M!(z + t)^k}{k!(M + N - k)!} \sum_{i + j = k} \binom{M + N - k}{N - i}
                \le \frac{2^{M + N - k}N!M!(z + t)^k}{k!(M + N - k)!}.
            \end{aligned}
            \end{equation*}
            Replacing $M$, $N$ and $k$ with their definitions, we get
            \begin{equation*}
            \begin{aligned}
                & \sum_{u + v = z + t - 1} \binom{2p(a + 1)  - z}{u}\binom{2pa - t + 1}{v} z^{f(z, u)} t^{g(t, v)}
                \\
                & \le \frac{2^{z + t - 1}(2p(a + 1) - z)!(2pa - t + 1)!(z + t)^{2p(2a + 1) - 2(z + t - 1)}}{(2p(2a + 1) - 2(z + t - 1))!(z + t - 1)!},
            \end{aligned}
            \end{equation*}
            replacing $z + t - 1$ with $l$, we prove the claim.
        \end{proof}

        Back to the proof of Claim \ref{claim:num-shapes-bound}, to uniquely determine the shape $S$, the general idea is the following.
        We first generated the preliminary codeword $W$ from $S$, then attempt to decode it.
        If we encounter a plus or neutral edge, we immediately know the next vertex.
        If we see a minus edge that follows from a plus edge $(u, v)$, we know that the next vertex is again $u$.
        Similarly, if there are chunks of the form $(++\ldots +--\ldots -)$ with the same number of each sign, the vertices are uniquely determined from the first vertex.
        Therefore, we can create a condensed codeword $W^*$ repeatedly removing consecutive pairs of $(+-)$ until none remains.
        For example, the sections $(-+-+-)$ and $(-++--)$ both become $(-)$.
        Observe that the condensed codeword is always unique regardless of the order of removal, and has the form
        \begin{equation*}
            W^* = [(+\ldots+) \text{ or } (-\ldots-)] \ (\text{neutral}) \ [(+\ldots+) \text{ or } (-\ldots-)] \ldots (\text{neutral}) \ [(+\ldots+) \text{ or } (-\ldots-)],
        \end{equation*}
        where we allow blocks of pure pluses and minuses to be empty.
        The minus blocks that remain in $W^*$ are the only ones where we cannot decipher.

        Recall that during decoding, we also reconstruct the tree $T(S)$, and the partial result remains a tree at any step.
        If we encounter a block of minuses in $W^*$ beginning with the vertex $i$, knowing the right endpoint $j$ of the last minus edge is enough to determine the rest of the vertices, which is just the unique path between $i$ and $j$ in the current tree.
        We call the last minus edge of such a block an \emph{important edge}.
        There are two cases for an important edge.
        \begin{enumerate}
            \item If $i$ and all vertices between $i$ and $j$ (excluding $j$) are only adjacent to at most two plus edges in the current tree (exactly for the interior vertices), we call this important edge \emph{simple} and just mark the it with a direction (left or right, in addition to the existing minus).
            For example, $(--\ldots-)$ becomes $(--\ldots(-dir))$ where $dir$ is the direction.

            \item If the edge is non-simple, we just mark it with the vertex $j$, so $(--\ldots-)$ becomes $(--\ldots(-j))$.
        \end{enumerate}
        It has been shown in \cite{vu2005} that the fully codeword $\overline{W}$ resulting from $W$ by marking important edges uniquely determines $S$, and that when the shape of $S$ \emph{is that of a single walk}, the cost of these markings is at most a multiplicative factor of $2(4N + 8)^N$, where $N$ is the number of neutral edges in the preliminary $W$.
        To adapt this bound to our case, we treat the star shape $S$ as a single walk, with a neutral edge marked by $1$ after every $2a + 1$ edges.
        There are $2p - 1$ additional neutral edges from this perspective, making $N = 4p(a + 1) - 2l - 1$ in total.
        Combining this with the bound on the number of preliminary codewords (Claim \ref{claim:num-prelim-codewords-bound}) yields
        \begin{equation*}
            |\mathcal{S}(x, y, z, t)|
            \le \frac{2^{l + 1}(2p(a + 1))!(2pa)!(l + 1)^{2p(2a + 1) - 2l}}{(2p(2a + 1) - 2l)!l!z!(t - 1)!}
            (16p(a + 1) - 8l - 2)^{4p(a + 1) - 2l - 1},
        \end{equation*}
        where $l = z + t - 1$.
        Claim \ref{claim:num-shapes-bound} is proven.
        \end{proof}

        Back to the proof of Lemma \ref{lem:E-power-bound-odd}.
        Temporarily let
        \begin{equation*}
            G_l \defeq 2p(2a + 1) - 2l
            \ \text{ and } \
            F_l \defeq \frac{2^{l + 1}(l + 1)^{G_l}}{G_l!l!}
            (4G_l + 8p - 2)^{G_l + 2p - 1}.
        \end{equation*}
        Note that $(2p(a + 1))! (2pa)!F_l$ is precisely the upper bound on $|\mathcal{S}(x, y, z, t)|$ in Claim \ref{claim:num-shapes-bound}.
        Also let
        \begin{equation*}
            \upperE_2 = \upperE_1 (2p(a + 1))! (2pa)! = \upperE^{2p(2a + 1)}(2p(a + 1))! (2pa)!\|u\|_\infty^{2p}.
        \end{equation*}
        Replacing the appropriate terms in the bound in Claim \ref{claim:num-shapes-bound} with these short forms, we get another series of upper bounds for the last double sum in Eq. \eqref{eq:E-power-bound-odd-temp1}:
        \begin{equation*}
        \begin{aligned}
            & \upperE_2 \sum_{x, y} \|u\|_\infty^{-x - 2y}
            \sum_{l = x + y}^{\lfloor 2pa + x/2 + y \rfloor} \upperE^{-2(l + 1)} F_l
            \sum_{z + t = l + 1} \frac{m^{z - x/2 - y}n^{t - 1}}{z!(t - 1)!}
            \\
            & \le \upperE_2 \sum_{x, y} \|u\|_\infty^{-x - 2y}
            \sum_{l = x + y}^{\lfloor 2pa + x/2 + y \rfloor}
            \frac{\upperE^{-2(l + 1)} F_l}{(l - \lfloor \frac{x}{2} \rfloor - y)!}
            \sum_{z + t = l + 1} \binom{l - \lfloor \frac{x}{2} \rfloor - y}{z - \lfloor \frac{x}{2} \rfloor - y} m^{z - \lfloor \frac{x}{2} \rfloor - y}n^{t - 1}
            \\
            & \le \upperE_2 \sum_{x, y} \|u\|_\infty^{-x - 2y}
            \sum_{l = x + y}^{\lfloor 2pa + x/2 + y \rfloor}
            \frac{\upperE^{-2(l + 1)} F_l}{(l - \lfloor \frac{x}{2} \rfloor - y)!}
            (m + n)^{l - \lfloor \frac{x}{2} \rfloor - y}.
        \end{aligned}
        \end{equation*}
        Temporarily let $C_l$ be the term corresponding to $l$ in the sum above.
        For $l \ge x + y + 1$, we have
        \begin{equation*}
        \begin{aligned}
            \frac{C_l}{C_{l - 1}}
            = \frac{2(m + n)(G_l + 1)(G_l + 2)}{\upperE^2l^3(4G_l + 8p - 2)^2 (l - \lfloor \frac{x}{2} \rfloor - y)}
            \left(1 + \frac{1}{l} \right)^{G_l}
            \left( 1 - \frac{4}{2G_l + 4p + 3}
            \right)^{G_l + 2p + 1}.
        \end{aligned}
        \end{equation*}
        The last power is approximately $e^{-2} \approx 0.135$, and for $p \ge 7$ a routine numerical check shows that it is at least $1/8$.
        The second to last power is at least $1$.
        The fraction be bounded as below.
        \begin{equation*}
            \frac{2(m + n)(G_l + 1)(G_l + 2)}{\upperE^2l^3(4G_l + 8p - 2)^2 (l - \lfloor \frac{x}{2} \rfloor - y)}
            \ge \frac{2(m + n)\cdot 1\cdot 2}{\upperE^2l^4(8p - 2)^2}
            \ge \frac{m + n}{16\upperE^2l^4p^2}
            \ge \frac{m + n}{16\upperE^2p^6(2a + 1)^4}.
        \end{equation*}
        Therefore, under the assumption that $m + n \ge 256\upperE^2p^6(2a + 1)^4$, we have $C_l \ge 2C_{l -1}$ for all $l\ge 1$, so
        $\sum_l C_l \le 2C_{l^*}$, where $l^* = \lfloor 2pa + x/2 + y \rfloor$, the maximum in the range.
        We have
        \begin{equation*}
        \begin{aligned}
            2C_{l^*}
            \le 
            & \ 2(m + n)^{2pa}
            \frac{(2\upperE^{-2})^{2pa + \lfloor \frac{x}{2} \rfloor + y + 1} (2pa + \lfloor \frac{x}{2} \rfloor + y + 1)^{2(p - \lfloor \frac{x}{2} \rfloor - y)}}{(2(p - \lfloor \frac{x}{2} \rfloor - y))! \cdot (2pa + \lfloor \frac{x}{2} \rfloor + y)! \cdot (2pa)!}
            \\
            & \ \cdot \left( 16p - 8\left\lfloor \frac{x}{2} \right\rfloor - 8y - 2
            \right)^{4p - 2\lfloor \frac{x}{2} \rfloor - 2y - 1}.
        \end{aligned}
        \end{equation*}
        Temporarily let $d = p - (\lfloor \frac{x}{2} \rfloor + y)$ and $N = p(2a + 1)$, we have
        \begin{equation*}
            2C_{l^*} \le 2(m + n)^{2pa}
            \frac{(2\upperE^{-2})^{N - d + 1} (N - d + 1)^{2d} (8p + 8d - 2)^{2p + 2d - 1}}{(2pa)! \cdot (2d)! \cdot (N - d)!}.
        \end{equation*}
        For each $d$, there are at most $2(p - d)$ pairs $(x, y)$ such that $d = p - (\lfloor \frac{x}{2} \rfloor + y)$, so overall we have the following series of upper bounds for the last double sum in Eq. \eqref{eq:E-power-bound-odd-temp1}:
        \begin{equation}  \label{eq:E-lemma-bound-temp10}
        \begin{aligned}
            & \upperE_2 (m + n)^{2pa} \sum_{d = 0}^{p - 1}
            4(p - d)\|u\|_\infty^{-2(p - d)} \cdot \frac{(2\upperE^{-2})^{N - d + 1} (N - d + 1)^{2d} (8p + 8d - 2)^{2p + 2d - 1}}{(2pa)! \cdot (2d)! \cdot (N - d)!}
            \\
            & \le \upperE_3 (m + n)^{2pa} \sum_{d = 0}^{p - 1}
            \|u\|_\infty^{2d} \cdot \frac{2^{-d}\upperE^{2d} (N - d + 1)^{2d} (8p + 8d - 2)^{2p + 2d - 1}}{(2d)! \cdot (N - d)!},
        \end{aligned}
        \end{equation}
        where
        \begin{equation*}
            \upperE_3 = 4p \frac{\upperE_2\|u\|_\infty^{-2p}(2\upperE^{-2})^{N + 1}}{(2pa)!}
            = 2^{p(2a + 1) + 3} p \upperE^{-2} (2p(a + 1))!.
        \end{equation*}
        Let us bound the sum at the end of Eq. \eqref{eq:E-lemma-bound-temp10}.
        Temporarily let $A_d$ be the term corresponding to $d$ and $x \defeq 2^{-1/2}\upperE\|u\|_\infty$.
        We have
        \begin{equation*}
        \begin{aligned}
            A_d & = \frac{x^{2d}(N - d + 1)^{2d}}{(2d)!(N - d)!}(8p + 8d - 2)^{2p + 2d - 1}
            \le \frac{x^{2d}N^{3d}}{(2d)!N!}\frac{(16p)^{2p + 2d}}{8p}.
        \end{aligned}
        \end{equation*}
        Therefore
        \begin{equation*}
        \begin{aligned}
            \sum_{d = 0}^{p - 1} A_d
            & \le \frac{(16p)^{2p}}{8p N!} \sum_{d = 0}^{p - 1} \frac{(16pN^{3/2}x)^{2d}}{(2d)!}
            \le \frac{(16p)^{2p}}{8p N!} \sum_{d = 0}^{p - 1} \binom{2p}{2d} (16pN^{3/2}x)^{2d} \frac{e^{2d}}{(2p)^{2d}}
            \\
            & = \frac{(16p)^{2p}}{8p N!} (8eN^{3/2}x + 1)^{2p}
            \le \frac{(16p)^{2p}}{8p N!} (16N^{3/2} \upperE\|u\|_\infty + 1)^{2p}.
        \end{aligned}
        \end{equation*}
        Plugging this into Eq. \eqref{eq:E-lemma-bound-temp10}, we get another upper bound for \eqref{eq:E-power-bound-odd-temp1}:
        \begin{equation*}
        \begin{aligned}
            \upperE_4 (16N^{3/2} \upperE\|u\|_\infty + 1)^{2p} (m + n)^{2ap},
        \end{aligned}
        \end{equation*}
        where
        \begin{equation*}
        \begin{aligned}
            \upperE_4
            & \defeq \upperE_3 \frac{(16p)^{2p}}{8p N!}
            = 2^{p(2a + 1) + 3} p \upperE^{-2} (2p(a + 1))! \frac{(16p)^{2p}}{8p (2ap + p)!}
            \le \frac{2^{2ap} 2^{10p} p^{3p} (a + 1)^p}{8\upperE^2}.
        \end{aligned}
        \end{equation*}
        To sum up, we have
        \begin{equation*}
        \begin{aligned}
            & \E{
            \bigl\|e_{n, 1}^T(E^TE)^aE^T U\bigr\|^{2p}
            }
            \le
            r^p \sum_{S\in \mathcal{S}} \sum_{P\in \mathcal{P}(S)} u_{\bound{P}} |\E{E_P}|
            \\
            & \quad \quad \le
            \frac{r^p 2^{2ap} 2^{10p} p^{3p} (a + 1)^p}{8\upperE^2}
            (16N^{3/2} \upperE\|u\|_\infty + 1)^{2p} (m + n)^{2ap}
            \\
            & \quad \quad \le
            \left(
                2^5r^{1/2}p^{3/2}\sqrt{2a + 1}
                (2^4p^{3/2}(2a + 1)^{3/2} \upperE\|u\|_\infty + 1)
                \cdot [2(m + n)]^a
            \right)^{2p}.            
        \end{aligned}
        \end{equation*}
        Let $D > 0$ be arbitrary.
        By Markov's inequality, for any $p$ such that $m + n \ge 2^8\upperE^2p^6(2a + 1)^4$, the moment bound above applies, so we have
        \begin{equation*}
            \bigl\|e_{n, 1}^T(E^TE)^aE^T U\bigr\|
            \le
            Dr^{1/2}p^{3/2}\sqrt{2a + 1}(16p^{3/2}(2a + 1)^{3/2}\upperE\|u\|_\infty + 1)[2(m + n)]^a
        \end{equation*}
        with probability at least $1 - (2^5/D)^{2p}$.
        Replacing $\|u\|_\infty$ with $\frac{1}{\sqrt{r}}\|U\|_{2, \infty}$, we complete the proof.
    \end{proof}

\subsubsection{Case 2: even powers}

\begin{proof}
    Without loss of generality, assume $k = 1$.
    We can reuse the first part and the notations from the proof of Lemma \ref{lem:E-power-bound-odd} to get the bound
    \begin{equation*}
        \E{\left\|e_{n, 1}^T(E^TE)^a V\right\|^{2p}}
        \le r^p 
        \sum_{S\in \mathcal{S}} \sum_{P\in \mathcal{P}(S)} v_{\bound{P}} |\E{E_P}|,
    \end{equation*}
    where $v_i = r^{-1/2}\|V_{\cdot, i}\|$.
    Again,
    \begin{equation*}
        \|v\| = 1 \ \text{ and } \|v\|_\infty = r^{-1/2}\|V\|_{2, \infty},
    \end{equation*}
    and $\mathcal{S}$ is the set of shapes such that every edge appears at least twice, $\mathcal{P}(S)$ is the set of stars having shape $S$, and
    \begin{equation*}
        E_P = \prod_{ij\in E(P)}E_{ij}^{m_P(ij)},
        \ \text{ and } \
        v_Q = \prod_{j\in V(Q)}v_j^{m_Q(j)}.
    \end{equation*}
    Note that a shape for a star now consists of walks of length $2a$:
    \begin{equation*}
        S = (S_1, S_2, \ldots, S_{2p}) \ \text{    where } S_r = j_{r0}i_{r0}j_{r1}i_{r1}\ldots j_{ra}.
    \end{equation*}
    We have, for any shape $S$ and $P\in \mathcal{P}(S)$,
    \begin{equation*}
        \E{E_P} \le \upperE^{4pa - 2|E(S)|} \le \upperE^{2pa - 2|V(S)| + 2},
        \quad
        |v_{\bound{P}}| \le \|v\|_\infty^{2p},
        \ \text{ and }
        |\mathcal{P}(S)| \le m^{|V_I(S)|}n^{|V_J(S)| - 1},
    \end{equation*}
    where the power of $n$ in the last inequality is due to $1$ having been fixed in $V_J(S)$. Therefore
    \begin{equation*}
        \sum_{S\in \mathcal{S}} \sum_{P\in \mathcal{P}(S)} v_{\bound{P}} |\E{E_P}|
        \le \upperE_1 \sum_{S\in \mathcal{S}} \upperE^{-2|V(S)|} m^{|V_I(S)|}n^{|V_J(S)| - 1},
        \ \text{ where } \upperE_1 \defeq \upperE^{4pa + 2}\|v\|_\infty^{2p}.
    \end{equation*}
    Let $\mathcal{S}(z, t)$ be the set of shapes $S$ such that $|V_I(S)| = z$ and $|V_J(S)| = t$.
    Let $\mathcal{A}$ be the set of eligible indices:
    \begin{equation*}
        \mathcal{A} \defeq \Bigl\{
            (z, t)\in \N^2: 0\le z, \ 1\le t, \ \text{ and } z + t \le 2pa + 1
        \Bigr\}.
    \end{equation*}
    Using the previous argument in the proof of Lemma \ref{lem:E-power-bound-odd} for counting shapes, we have for $(z, t)\in \mathcal{A}$:
    \begin{equation*}
        |\mathcal{S}(z, t)|
        \le \frac{[(2pa)!]^2 F_l}{z! \cdot (t - 1)!}
        m^z n^{t - 1},
        \ \text{ where } l \defeq z + t - 1 \in [2pa],
    \end{equation*}
    where
    \begin{equation*}
        G_l \defeq 4ap - 2l
        \ \text{ and } \
        F_l \defeq \frac{2^{l + 1}(l + 1)^{G_l}}{G_l!l!}
        (4G_l + 8p - 2)^{G_l + 2p - 1}.
    \end{equation*}
    We have
    \begin{equation*}
    \begin{aligned}
        & \sum_{S\in \mathcal{S}} \sum_{P\in \mathcal{P}(S)} v_{\bound{P}} |\E{E_P}|
        \le \upperE_1 \sum_{l = 0}^{2ap} \upperE^{-2(l + 1)} [(2ap)!]^2 F_l \sum_{z + t = l + 1}
        \frac{m^z n^{t - 1}}{z! \cdot (t - 1)!}
        \\
        & = \upperE_2 \sum_{l = 0}^{2ap} \frac{\upperE^{-2l} F_l}{l!} \sum_{z + t = l + 1} \binom{l}{z} m^z n^{t - 1}
        = \upperE_2 \sum_{l = 0}^{2ap} \frac{\upperE^{-2l} F_l}{l!} (m + n)^l,
    \end{aligned}
    \end{equation*}
    where $\upperE_2 \defeq \upperE_1[(2pa)!]^2 \upperE^{-2} = \upperE^{4ap}[(2pa)!]^2 \|v\|_\infty^{2p}$.
    Let $C_l$ be the term corresponding to $l$ in the last sum above.
    An analogous calculation from the proof of Lemma \ref{lem:E-power-bound-odd} shows that
    under the assumption that $m + n \ge 256\upperE^2p^6(2a)^4$, $C_l \ge 2C_{l - 1}$ for each $l$, so $\sum_{l = 0}^{2pa} C_l \le 2C_{2pa}$, where
    \begin{equation*}
        C_{2pa} = \frac{\upperE^{-4ap} 2^{2ap + 1} (8p - 2)^{2p - 1}}{[(2ap)!]^2} (m + n)^{2ap}.
    \end{equation*}
    Therefore
    \begin{equation*}
    \begin{aligned}
        & \E{\left\|e_{n, 1}^T(E^TE)^a V\right\|^{2p}}
        \le
        r^p \sum_{S\in \mathcal{S}} \sum_{P\in \mathcal{P}(S)} v_{\bound{P}} |\E{E_P}|
        \\
        & \le 2r^p \upperE_2 \frac{\upperE^{-4ap} 2^{2ap + 1} (8p - 2)^{2p - 1}}{[(2ap)!]^2} (m + n)^{2ap}
        = 4\left(
            2^3pr^{1/2} \|v\|_\infty [2(m + n)]^a
        \right)^{2p}.
    \end{aligned}
    \end{equation*}
    Pick $D > 0$, by Markov's inequality, we have
    \begin{equation*}
        \Pr\left(
            \left\|e_{n, 1}^T(E^TE)^a V\right\|
            \ge Dpr^{1/2}\|v\|_\infty [2(m + n)]^a
        \right)
        \le \left( \frac{16p}{D}
        \right)^{2p}.
    \end{equation*}
    Replacing $\|v\|_\infty$ with $r^{-1/2}\|V\|_{2, \infty}$, we complete the proof.
\end{proof}

\bibliographystyle{amsplain}
\bibliography{references}

\end{document}